\documentclass[11pt,reqno,a4paper]{amsart}
\usepackage[T1]{fontenc}
\usepackage[english]{babel}
\usepackage[babel]{microtype}

\usepackage{tikz}
\usetikzlibrary{positioning}

\usepackage{mathtools}
\usepackage{enumerate}
\usepackage{amsmath}
\usepackage{amsthm}
\usepackage{amsfonts}
\usepackage{amssymb}
\usepackage[numbers]{natbib}
\usepackage{xcolor}
\usepackage{graphicx}

\usepackage{fullpage}
\usepackage{xcolor}
\usepackage[pdfborderstyle={/S/U/W 0},unicode]{hyperref}
\hypersetup{
    colorlinks=true,
    linkcolor=blue!85!black,
    citecolor=blue!85!black,
    urlcolor=blue!85!black,
}
\usepackage{bbm}
\usepackage{algorithm}
\usepackage{algpseudocode}

\usepackage{cleveref}

\newtheorem{question}{Question}
\numberwithin{question}{section}
\newtheorem{conjecture}[question]{Conjecture}
\newtheorem{problem}[question]{Problem}

\newtheorem{corollary}[question]{Corollary}
\newtheorem{theorem}[question]{Theorem}
\newtheorem{proposition}[question]{Proposition}
\newtheorem{lemma}[question]{Lemma}
\newtheorem{claim}[question]{Claim}

\newtheorem{definition}[question]{Definition}

\newtheorem{remark}[question]{Remark}
\numberwithin{equation}{section}

\crefname{problem}{problem}{problems}

\newcommand{\eps}{\varepsilon}

\newcommand{\NN}{\mathbb{N}}

\newcommand{\QQ}{\mathbb{Q}}
\newcommand{\RR}{\mathbb{R}}

\newcommand{\ZZ}{\mathbb{Z}}

\newcommand{\cA}{\mathcal{A}}

\newcommand{\cC}{\mathcal{C}}

\newcommand{\cF}{\mathcal{F}}

\newcommand{\cP}{\mathcal{P}}
\newcommand{\cQ}{\mathcal{Q}}

\newcommand{\rT}{\mathrm{T}}

\newcommand{\bfr}{\mathbf{r}}

\newcommand{\bfx}{\mathbf{x}}

\newcommand{\defined}{\mathrel{\coloneqq}}

\DeclarePairedDelimiter{\p}{\lparen}{\rparen}

\newcommand{\st}{\mathbin{\colon}}
\undef{\set}
\DeclarePairedDelimiter{\set}{\lbrace}{\rbrace}

\DeclarePairedDelimiter{\abs}{\lvert}{\rvert}

\undef{\emptyset}
\newcommand{\emptyset}{\varnothing}

\DeclareMathOperator{\ind}{\mathbf{1}}

\newcommand{\from}{\colon}

\DeclarePairedDelimiter{\floor}{\lfloor}{\rfloor}

\DeclarePairedDelimiter{\ceil}{\lceil}{\rceil}

\DeclarePairedDelimiter{\card}{\lvert}{\rvert}

\DeclareMathDelimiter{\given}
  {\mathbin}{symbols}{"6A}{largesymbols}{"0C}
\DeclareMathOperator{\Prob}{\mathbb{P}}
\DeclarePairedDelimiterXPP{\prob}[1]
  {\Prob}{\lparen}{\rparen}{}
  {\renewcommand{\given}{\nonscript\;\delimsize\vert\nonscript\;\mathopen{}}#1}
\DeclareMathOperator{\Expec}{\mathbb{E}}
\DeclarePairedDelimiterXPP{\expec}[1]
  {\Expec}{\lparen}{\rparen}{}
  {\renewcommand{\given}{\nonscript\;\delimsize\vert\nonscript\;\mathopen{}}#1}
\DeclareMathOperator{\Var}{Var}
\DeclarePairedDelimiterXPP{\var}[1]
  {\Var}{\lparen}{\rparen}{}
  {\renewcommand{\given}{\nonscript\;\delimsize\vert\nonscript\;\mathopen{}}#1}
\DeclareMathOperator{\Cov}{Cov}
\DeclarePairedDelimiterXPP{\cov}[2]
  {\Cov}{\lparen}{\rparen}{}{#1,#2}

\DeclareMathOperator{\ex}{ex}
\DeclareMathOperator{\Blow}{Blow}
\DeclareMathOperator{\Petersen}{Petersen}
\DeclareMathOperator{\Forb}{Forb}

\newcommand{\SIGMA}[2]{\Sigma_{#2}(#1)}
\newcommand{\PI}[2]{\Pi_{#2}(#1)}
\newcommand{\TUR}[2]{\rT[#1 ; #2]}

\begin{document}

\title{On problems in extremal multigraph theory}

\author{Victor Falgas--Ravry}
\address{Institutionen för matematik och matematisk statistik, Umeå Universitet, Sweden}
\email{victor.falgas-ravry@umu.se}

\author{Adva Mond}
\address{Department of Mathematics, King's College London, UK}
\email{adva.mond@kcl.ac.uk}

\author{Rik Sarkar}
\address{Indian Institute of Science Education and Research (IISER), Pune, India.}
\email{rik.sarkar@students.iiserpune.ac.in}

\author{Victor Souza}
\address{Department of Computer Science and Technology, University of Cambridge, UK.}
\email{vss28@cam.ac.uk}

\begin{abstract}
A multigraph $G$ is said to be an $(s,q)$-graph if every $s$-set of vertices in $G$ supports at most $q$ edges (counting multiplicities). In this paper we consider the maximal sum and product of edge multiplicities in an $(s,q)$-graph on $n$ vertices. These are multigraph analogues of a problem of Erdős raised by Füredi and Kündgen and Mubayi and Terry respectively, with applications to counting problems and extremal hypergraph theory.

We make major progress, settling conjectures of Day, Falgas--Ravry and Treglown and of Falgas-Ravry, establishing intricate behaviour for both the sum and the product problems, and providing both a general picture and evidence that the problems may prove computationally intractable in general.
\end{abstract}

\maketitle


\section{Introduction}
We investigate the extremal theory of \emph{$(s,q)$-sparse multigraphs}, that is, multigraphs for which every $s$-set of vertices supports a total of at most $q$ edges, counting multiplicities. For brevity's sake, we refer to such multigraphs as \emph{$(s,q)$-graphs}. We determine the asymptotically maximal sum and product of edge multiplicities in $(s,q)$-graphs on $n$ vertices for a wide range of $s$ and $q$. We begin this paper in Section~\ref{subsection: background and motivation} by giving some background and motivation for considering these problems, which are related to a number of other well-studied questions in extremal combinatorics. In Section~\ref{subsection: previous work}, we present previous work and conjectures on our multigraph problems, and introduce some of the notions and notation needed to state our results. Our own contributions can then be found in Section~\ref{subsection: our contributions}, and the structure of the rest of the paper is described in Section~\ref{subsection: structure of the paper}.

\subsection{Background and motivation}\label{subsection: background and motivation}

In the 1960's, Erdős~\cite{erdos1967,erdos1964} raised the problem of determining $\ex(n,s,q)$, the maximum number of edges in an $n$-vertex graph in which every $s$-set of vertices spans at most $q$ edges, for some integer $0 \leq q < \binom{s}{2}$.
This can be viewed as asking for a generalisation of Turán's theorem~\cite{turan1941}, one of the foundational results of extremal graph theory.
Indeed, a $K_s$-free graph is a graph in which every $s$-set of vertices spans at most $q = \binom{s}{2} - 1$ edges, and thus in this special case $\ex(n,s,q)$ is the maximum number of edges in a $K_s$-free graph, i.e.\ the Tur\'an number $\mathrm{ex}(n, K_s)$.

This problem of Erd{\H o}s has received a significant amount of attention in the literature over the years.
Erdős~\cite{erdos1964} resolved the case $s \leq 5$, while a theorem of Dirac~\cite{Dirac1963-vv} dealt with the cases $q \geq \binom{s}{2} - s/2$.
In the range $q \geq \floor{s^2/4}$, the function $n \mapsto \ex(n,s,q)$ is quadratic in $n$, and its asymptotic rate of growth follows easily from the Erdős--Stone theorem~\cite{Erdos1946-df}.
In the same regime, there are also a number of results about the size of the lower order terms of $n \mapsto \ex(n,s,q)$, see Gol'berg and Gurvich~\cite{gol1987maximum}, and Griggs, Simonovits and Thomas~\cite{griggs1998extremal}.
On the other hand, it can be shown that $n \mapsto \ex(n,s,q)$ is linear when $q < s - 1$, and superlinear but subquadratic in the range $s - 1 \leq q < \floor{s^2 / 4}$, see~\cite{gol1987maximum}.
In the case $q < s - 1$, the value of $\ex(n, s, q)$ is known exactly, however, in the `polynomial zone' $s - 1 \leq q< \floor{s^2 /4}$, the situation is much more difficult: as Erdős himself observed, the case $(s, q) = (s, s - 1)$ is equivalent to the notoriously difficult and still open problem of determining the Turán number of the even cycle $C_{2\floor{s/2}}$.

In the late 1990's, Bondy and Tuza~\cite{Bondy1997-ji} and Kuchenbrod~\cite{Kuchenbrod1999-as} independently began to investigate the integer-weighted version of Erdős's problem.
Recall that an \emph{integer-weighted graph} is a pair $(V, w)$, where $V$ is a set of vertices and $w \from \binom{V}{2} \to \ZZ$ assigns to each edge (i.e.\ each pair of vertices from $V$) an integer weight.
Let $\ex_\ZZ(n, s, q)$ denote the maximum sum of the edge-weights in an integer-weighted graph on $n$ vertices in which the sum of the edge-weights inside any $s$-set of vertices is at most $q$. Bondy and Tuza showed that an analogue of Turán's theorem holds in this setting, with $\ex_\ZZ(n, s, \binom{s}{2} - 1) = \ex(n, s, \binom{s}{2} - 1)$, but that unlike in the ordinary graph setting, there may be several non-isomorphic extremal constructions that attain $\ex_\ZZ(n, s, \binom{s}{2} - 1)$.
Kuchenbrod, for his part, determined $\ex_\ZZ(n, s, q)$ for all $s \leq 7$.

A major advance on this problem came in 2002, when Füredi and Kündgen~\cite{Furedi2002-fi} determined the asymptotic growth rate of $n \mapsto \ex_\ZZ(n, s, q)$.
More precisely, they showed that for all integers $s \geq 2$ and $q \geq 0$, we have $\ex_\ZZ(n, s, q) = m(s, q) \binom{n}{2} + O(n)$, where $m(s, q) \in \QQ_{\geq 0}$ is given by $m(s, q) \defined \min \set{ m \st \sum_{i < s} \floor{1 + mi} > q}$.
Füredi and Kündgen observed that for $q > (s - 1) \binom{s}{2}$, their constructions of integer-weighted graphs giving asymptotically tight lower bounds for $\ex_\ZZ(n, s, q)$ have non-negative edge weights --- in particular, their results on integer-weighted graphs asymptotically resolve Erdős's problem for multigraphs for those pairs $(s, q)$ satisfying $q > (s - 1) \binom{s}{2}$.

Recall here that a multigraph is a pair $G = (V, w)$, where $V = V(G)$ is a set of vertices, and $w = w_G$ is a function $w \from \binom{V}{2} \mapsto \ZZ_{\geq 0}$.
For each $u, v \in V$, the value $w(uv)$ is a non-negative integer called the \emph{multiplicity} of the edge $uv \defined \set{u,v}$ in $G$.
For integers $s \geq 2$, $q \geq 0$, we say a multigraph $G =(V, w)$ is an \emph{$(s,q)$-graph} if for every $s$-set of vertices $X$, the sum of the multiplicities of the edges inside $X$ is at most $q$.

Let $\ex_\Sigma(n,s,q)$ denote the maximum sum of the edge-multiplicities in an $(s,q)$-graph on $n$ vertices.
Clearly, $\ex_\Sigma(n,s,q)\leq \ex_\ZZ(n,s,q)$, and Füredi and Kündgen's results show that for $q>(s-1)\binom{s}{2}$, these quantities are equal up to a $O(n)$ error term, and thus $\ex_\Sigma(n,s,q)=m(s,q)\binom{n}{2}+O(n)$.
As Füredi and Kündgen noted, however (see the remarks at the end of~\cite[Section 11]{Furedi2002-fi}), for smaller values of $q$, it is not necessarily the case that $\ex_\Sigma(n,s,q)$ and $\ex_\ZZ(n,s,q)$ differ by $o(n^2)$. 
The multigraph version of Erdős's problem remains wide open for $q$ in this range.

\begin{problem}[Füredi--Kündgen multigraph problem]
\label{problem: Furedi Kundgen}
Given a pair of non-negative integers $(s,q)$ with $0\leq q \leq(s-1)\binom{s}{2}$, determine the growth rate of $\ex_\Sigma(n,s,q)$.
\end{problem}

More recently, Mubayi and Terry \cite{Mubayi2019-le} introduced a product version of the Erdős problem for multigraphs.
Let $P(G)$ denote the product of the edge multiplicities of a multigraph $G$, that is, $P(G) \defined \prod_{uv \in \binom{V}{2}} w_G(uv)$.
Define $\ex_\Pi(n,s,q)$ to be the maximum of $P(G)$ over all $(s,q)$-graphs $G$ on $n$ vertices.

\begin{problem}[Mubayi--Terry multigraph problem]
\label{problem: Mubayi Terry}
Given a pair of non-negative integers $(s,q)$, determine the growth rate of $\ex_\Pi(n,s,q)$.
\end{problem}

Mubayi and Terry's motivation for studying the quantity $\ex_\Pi(n,s,q)$ was to prove counting theorems for $(s,q)$-graphs.
In an influential 1976 paper, Erdős, Kleitman and Rothschild~\cite{Erdos1976-ad} showed that the number of $K_{r}$-free graphs on $n$ vertices is $2^ {\ex(n,K_{r})+o(n^2)}$. Since their foundational result, it has been a major goal of extremal combinatorics to prove similar counting results for other classes of graphs and discrete structures. Indeed, such results are a necessary prerequisite to a characterisation of the typical structure of graphs satisfying a given property.

Counting results of the Erdős--Kleitman--Rothschild-type are too numerous to list here. In keeping with the themes treated in this paper, we content ourselves with the observation that the Alekseev--Bollobás--Thomason theorem~\cite{alekseev1993entropy, bollobas1997hereditary} implies that the number of graphs on $n$ vertices in which every $s$-set spans at most $q$ edges is $2^{\ex(n,s,q)+o(n^2)}$. What can be said about multigraphs with the same property? More precisely, let $\cA(n,s,q)$ denote the collection of $(s,q)$-graphs on the vertex-set $[n] \defined \set{1, \dotsc, n}$.

\begin{problem}[Erdős--Kleitman--Rothschild problem for $(s,q)$-graphs]
\label{problem: EKR for multigraphs}
Determine the asymptotic growth rate of $\card{\cA(n,s,q)}$.
\end{problem}

Using the powerful hypergraph container theories developed by Balogh, Morris and Samotij~\cite{Balogh2015-al} and Saxton and Thomason \cite{Saxton2015-cr}, Mubayi and Terry~\cite{Mubayi2019-le} showed that for $q > \binom{s}{2}$, 
\begin{align}
\label{eq: MT reduction of counting to extremal}
    \card[\big]{ \cA\p[\big]{n,s,q - \tbinom{s}{2}} } = \ex_\Pi(n, s, q)e^{o(n^2)}.
\end{align}
Thus, \Cref{problem: EKR for multigraphs} is actually equivalent to the Mubayi--Terry problem, \Cref{problem: Mubayi Terry}. In particular, with regards to counting and typical structure results, the (multiplicative) Mubayi--Terry problem is the `right' generalisation of Turán's theorem to multigraphs, rather than the (additive) Füredi--Kündgen problem. There are, however, some links between the two problems, and while our main focus in this paper is the Mubayi--Terry problem, whose history we discuss in the next subsection, we will naturally be led to study the Füredi--Kündgen problem as well in order to address the Mubayi--Terry problem.

Before closing our discussion of the general background of the questions studied in this paper, we should make two further remarks. First of all, the counting \Cref{problem: EKR for multigraphs} is related to the well-studied Erdős--Rothschild problem~\cite{erdos1974some} of counting Gallai colourings of graphs. Indeed, it follows from an observation of Falgas--Ravry, O'Connell and Uzzell~\cite{Falgas-Ravry2019-ne} that Gallai colourings in complete graphs can be counted by solving an entropy maximisation problem in graphs bearing a close resemblance to \Cref{problem: Mubayi Terry}. For more on the Erdős--Rothschild problem and its generalisations, see e.g.\ the works of Hoppen, Lefmann and Odermann~\cite{hoppen2017rainbow}, Bastos, Benevides and Han~\cite{de2020number}, and Pikhurko and Staden~\cite{pikhurko2024exact}, and the references therein.

Secondly, \Cref{problem: Furedi Kundgen} is a natural intermediate problem between the well-understood Turán problem for graphs and the much more poorly understood Turán problem for hypergraphs. Indeed, this was pointed out by Füredi and Kündgen, who noted that \Cref{problem: Furedi Kundgen} does not exhibit \emph{stability}: there are instances $(s,q)$ for which there are near-extremal multigraphs that lie at edit distance $\Omega(n^2)$ from each other. This stands in contrast to the situation for graphs, for which we have the Simonovits stability theorem~\cite{simonovits1968method}, and resembles rather the situation for hypergraphs where such instability was established recently by Liu and Mubayi~\cite{liu2022hypergraph}. Extremal multigraph problems mirror extremal hypergraph problems in other ways too: Rödl and Sidorenko~\cite{rodl1995jumping} showed in a 1995 paper that the collection of limiting edge densities of properties of multigraphs with edge multiplicities bounded by $4$ is not well-ordered, providing an analogue of the celebrated ``Hypergraphs do not jump'' result of Frankl and Rödl~\cite{frankl1984hypergraphs}.

Further, an initial motivation for the study of the Füredi--Kündgen problem was the investigation of problems in extremal hypergraph theory via link graphs, and Füredi and Kündgen's work was a key tool in de Caen and Füredi's determination of the Turán density of the Fano plane~\cite{de2000maximum}. Finally, it would be remiss not to mention in this context that the hypergraph version of Erdős's original problem (so the hypergraph analogue of \Cref{problem: Furedi Kundgen}) is both well-studied and poorly understood --- in $3$-uniform hypergraphs, this is known as the Brown--Erdős--Sós problem~\cite{brown1973some}, and the special case $(s,q)=(6,3)$ corresponds to the celebrated Ruzsa--Szemerédi theorem~\cite{ruzsa1978triple}. For recent work on the Brown--Erdős--Sós problem, see~\cite{delcourt2024limit, glock20246, keevash2020brown, shapira2025new}.

\subsection{Previous work on the Mubayi--Terry problem}\label{subsection: previous work}

Adapting a classical averaging argument of Katona, Nemetz and Simonovits~\cite{katona1964graph}, one can show that the sequences $\ex_\Sigma(n,s, q)/\binom{n}{2}$ and $\left(\ex_\Pi(n,s,q)\right)^{1/\binom{n}{2}}$ are non-increasing in $n$ and hence tend to limits, which we refer to as the (extremal) arithmetic density and geometric density of $(s,q)$-graphs respectively and denote by
\begin{align*}
    \ex_\Sigma(s,q) \defined \lim_{n \to \infty} \frac{\ex_\Sigma(n,s,q)}{\binom{n}{2}},
    \quad \text{and}
    \quad \ex_\Pi(s,q) \defined \lim_{n \to \infty} \p[\big]{\ex_\Pi(n,s,q)}^{1/\binom{n}{2}}.
\end{align*}
These two limit quantities can be thought of as additive and multiplicative analogues of the classical Turán density for graphs; $\ex_\Sigma(s,q)$ is the asymptotically maximal arithmetic mean edge multiplicity in an $(s,q)$-graph, while $\ex_\Pi(s,q)$ is the asymptotically maximal geometric mean edge multiplicity in an $(s,q)$-graph.
It is thus immediate from the AM-GM inequality that we have
\begin{align*}
    \ex_\Pi(s,q)\leq \ex_\Sigma(s,q).
\end{align*}
What is more, one can show that equality is achieved if and only if $q = a\binom{s}{2} + r$ for some integers $a \geq 1$ and $s - 2 \geq r \geq 0$, in which case both quantities are equal to $a$, and extremal multigraphs are within edit distance $o(n^2)$ of the multigraph in which every edge has multiplicity $a$ (for $r = 0$, this is an easy averaging argument, and for $1\leq r\leq s-2$, this follows easily from results of Mubayi and Terry~\cite{Mubayi2020-fb}).

The goal in \Cref{problem: Furedi Kundgen,problem: Mubayi Terry} is thus to understand the behaviour of the arithmetic and geometric densities in $(s,q)$-graphs for $q$ outside the narrow range $[0,s-2]+ \binom{s}{2}\ZZ_{>0}$, and to quantify the extent to which sum-maximisation and product-maximisation in $(s,q)$-graphs differ for those values of $q$.
To facilitate a discussion of previous known results on \Cref{problem: Furedi Kundgen,problem: Mubayi Terry}, it is helpful to introduce the notion of a \emph{(multigraph) pattern}.

\begin{definition}[Patterns]
A (multigraph) pattern $P$ on a vertex set $V=V(P)$ is a function $P \from V \sqcup \binom{V}{2} \to \ZZ_{\geq 0}$ assigning non-negative integer multiplicities to the singletons (loops) and pairs (edges) from $V$. 
\end{definition}

One can identify a pattern $P$ with its adjacency matrix, the $\card{V} \times \card{V}$ symmetric matrix $A_P$ whose non-negative integer entries are given by $(A_P)_{uv} \defined P(\set{u,v})$.

\begin{definition}[Blow-ups]
A blow-up of a pattern $P$ is a multigraph $G$ such that there exists a partition $f \from V(G) \to V(P)$ with the property that for every pair $\set{v,v'} \in \binom{V(G)}{2}$, we have $w_G(\set{v,v'})= P(\set{f(v), f(v')})$. 
\end{definition}

In other words, $G$ is a blow-up of $P$ if one can partition the vertices of $G$ into sets $U_v$, $v\in V(P)$, such that all edges internal to $U_v$ have multiplicity $P(\set{v})$, and all edges from $U_{v}$ to $U_{v'}$ have multiplicity $P(\set{v,v'})$. We refer to any such partition $V(G)=\sqcup_{v\in V(P)} U_v$ as a $P$-partition of $G$. We write $\Blow_n(P)$ for the collection of all blow-ups of $P$ on the vertex-set $[n]$, and $\Blow(P) = \left(\Blow_n(P)\right)_{n \in \NN}$ for the multigraph property of being a blow-up of $P$.
We define
\begin{equation*}
    \SIGMA{P}{n} \defined \max\set[\big]{e(G) \st  G\in \Blow_n(P) },
    \quad \text{and} \quad
    \PI{P}{n} \defined \max\set[\big]{P(G):\  G\in \Blow_n(P)}
\end{equation*}
to be the maximal sum and product of multiplicities achievable by $n$-vertex blow-ups of a pattern $P$, respectively. As we show in \Cref{subsection: patterns} (see~\eqref{eq: optimisation problem for arithmetic problem} and~\eqref{eq: optimisation problem for geometric problem}), for any pattern $P$, there exist explicitly computable constants $\sigma_P\in \QQ_{\geq 0}$ and $\pi_P\in \RR_{\geq 0}$ such that $\SIGMA{P}{n}= \sigma_P \binom{n}{2}+o(n^2)$ and $\PI{P}{n}= e^{\pi_P \binom{n}{2}+o(n^2)}$.

Clearly, if $\Blow_n(P) \subseteq \cA(n,s,q)$, then we have $\ex_\Sigma(n,s,q)\geq \SIGMA{P}{n}$ and $\ex_\Pi(n,s,q)\geq \PI{P}{n}$. Blow-ups of finite multigraph patterns are thus a natural source of lower-bound constructions, and all asymptotically tight results on \Cref{problem: Furedi Kundgen,problem: Mubayi Terry} have been obtained in this way. Moreover, before the present paper, all patterns were of the following particular form.

\begin{definition}[Generalised Turán pattern]
Given a $(d+1)$-tuple of non-negative integers $\bfr=(r_0, r_1, r_2, \dotsc, r_d)$ with $r_0,r_d>0$, we define the generalised Turán pattern $\TUR{\bfr}{a}$ as follows. 
Set $R \defined \sum_{i=0}^d r_i$ and $V(\TUR{\bfr}{a}) \defined [R]$. For every $v\in [R]$ with $\sum_{i<j}r_i <v \leq \sum_{i\leq j}r_i$ and every $v'$: $v<v'\leq R$, set $\TUR{\bfr}{a}(\set{v})=a-j$ and  $\TUR{\bfr}{a}(\set{v,v'})= a-j+1$.
\end{definition}

\begin{figure}[!ht]
\centering
\begin{tikzpicture}[scale=1.0]
    \tikzstyle{vertex}=[circle, fill=black, inner sep=0pt, minimum size=6pt, line width=1pt, white];
    
    \colorlet{col1}{blue!70!green!80!white};
    \colorlet{col2}{red!85!black};
    \colorlet{col4}{green!70!black};
    \colorlet{col3}{orange!80!white};
    \colorlet{col5}{violet!80!white};
    
    \tikzstyle{edge-prep}=[white!20, line width=2pt];
    \tikzstyle{edge-c1}=[col1, very thick];
    \tikzstyle{edge-c2}=[col2, very thick];
    \tikzstyle{edge-c3}=[col3, very thick];
    \tikzstyle{edge-c4}=[col4, very thick];
    \tikzstyle{edge-c5}=[col5, very thick];

    \begin{scope}[xshift = 0cm]
        \begin{scope}[rotate=30]
            \foreach \x/\y in {0/1,60/2,120/3,180/4,240/5,300/6}{
                \node[vertex] (v\y) at (canvas polar cs: radius=1.05cm, angle=\x) {};
            }
        \end{scope}
        
        \draw[edge-prep] (v1)--(v2)--(v3)--(v4)--(v5)--(v6)--(v1)
        (v1)--(v3)--(v5)--(v1) (v2)--(v4)--(v6)--(v2) (v1)--(v4) (v2)--(v5) (v3)--(v6);
        \draw[edge-c1] (v1)--(v2)--(v3)--(v4)--(v5)--(v6)--(v1)
        (v1)--(v3)--(v5)--(v1) (v2)--(v4)--(v6)--(v2) (v1)--(v4) (v2)--(v5) (v3)--(v6);
    
        \node[vertex, fill=col2] at (v1) {};
        \node[vertex, fill=col2] at (v2) {};
        \node[vertex, fill=col2] at (v3) {};
        \node[vertex, fill=col2] at (v4) {};
        \node[vertex, fill=col2] at (v5) {};
        \node[vertex, fill=col2] at (v6) {};

        \node[] at (0,-1.75) {$\TUR{(6)}{a}$};
    \end{scope}

    \begin{scope}[xshift=3.15cm]
        \node[vertex] at (-0.75,-1) (v1) {};
        \node[vertex] at (-1.25,-0.5) (v2) {};
        \node[vertex] at (-1.25,0.5) (v3) {};
        \node[vertex] at (-0.75,1) (v4) {};
        \node[vertex] at (0.75,-1) (v5) {};
        \node[vertex] at (1.25,-0.5) (v6) {};
        \node[vertex] at (1.25,0.5) (v7) {};
        \node[vertex] at (0.75,1) (v8) {};
        
        \draw[edge-prep] (v1)--(v2)--(v3)--(v4)--(v1)--(v3) (v2)--(v4);
        \draw[edge-c1] (v1)--(v2)--(v3)--(v4)--(v1)--(v3) (v2)--(v4);
        
        \draw[edge-prep] (v1)--(v5) (v1)--(v6) (v1)--(v7) (v1)--(v8) (v4)--(v8) (v4)--(v7) (v4)--(v6) (v4)--(v5);
        \draw[edge-c1] (v1)--(v5) (v1)--(v6) (v1)--(v7) (v1)--(v8) (v4)--(v8) (v4)--(v7) (v4)--(v6) (v4)--(v5);
        
        \draw[edge-prep] (v2)--(v5) (v2)--(v6) (v2)--(v7) (v2)--(v8) (v3)--(v8) (v3)--(v7) (v3)--(v6) (v3)--(v5);
        \draw[edge-c1] (v2)--(v5) (v2)--(v6) (v2)--(v7) (v2)--(v8) (v3)--(v8) (v3)--(v7) (v3)--(v6) (v3)--(v5);

        \draw[edge-prep] (v5)--(v6)--(v7)--(v8)--(v5)--(v7) (v6)--(v8);
        \draw[edge-c2] (v5)--(v6)--(v7)--(v8)--(v5)--(v7) (v6)--(v8);
    
        \node[vertex, fill=col2] at (v1) {};
        \node[vertex, fill=col2] at (v2) {};
        \node[vertex, fill=col2] at (v3) {};
        \node[vertex, fill=col2] at (v4) {};
        \node[vertex, fill=col3] at (v5) {};
        \node[vertex, fill=col3] at (v6) {};
        \node[vertex, fill=col3] at (v7) {};
        \node[vertex, fill=col3] at (v8) {};

        \node[] at (0,-1.75) {$\TUR{(4,4)}{a}$};
    \end{scope}

    \begin{scope}[xshift = 6.0cm, yshift = 0cm]
        \node[vertex] at (0:-0.8) (v3) {};
        \node[vertex] at (80:-1.25) (v4) {};
        \node[vertex] at (144:-1.25) (v5) {};
        \node[vertex] at (216:-1.25) (v1) {};
        \node[vertex] at (-80:-1.25) (v2) {};
        \node[vertex, right = 1cm of v1] (v6) {};
        \node[vertex, right = 1cm of v5] (v7) {};
        
        \draw[edge-prep] (v2)--(v1) (v2)--(v3) (v2)--(v4) (v2)--(v5) (v2)--(v6) (v2)--(v7);
        \draw[edge-c1] (v2)--(v1) (v2)--(v3) (v2)--(v4) (v2)--(v5) (v2)--(v6) (v2)--(v7);
        \draw[edge-prep] (v3)--(v1) (v3)--(v2) (v3)--(v4) (v3)--(v5) (v3)--(v6) (v3)--(v7);
        \draw[edge-c1] (v3)--(v1) (v3)--(v2) (v3)--(v4) (v3)--(v5) (v3)--(v6) (v3)--(v7);
        \draw[edge-prep] (v4)--(v1) (v4)--(v2) (v4)--(v3) (v4)--(v5) (v4)--(v6) (v4)--(v7);
        \draw[edge-c1] (v4)--(v1) (v4)--(v2) (v4)--(v3) (v4)--(v5) (v4)--(v6) (v4)--(v7);
        
        \draw[edge-prep] (v5)--(v6) (v5)--(v7);
        \draw[edge-c2] (v5)--(v6) (v5)--(v7);
        \draw[edge-prep] (v1)--(v5) (v1)--(v6) (v1)--(v7);
        \draw[edge-c2] (v1)--(v5) (v1)--(v6) (v1)--(v7);
    
        \draw[edge-prep] (v6)--(v7);
        \draw[edge-c3] (v6)--(v7);
        \draw[edge-prep] (v7)--(v6);
        \draw[edge-c3] (v7)--(v6);
    
        \node[vertex, fill=col3] at (v1) {};
        \node[vertex, fill=col2] at (v2) {};
        \node[vertex, fill=col2] at (v3) {};
        \node[vertex, fill=col2] at (v4) {};
        \node[vertex, fill=col3] at (v5) {};
        \node[vertex, fill=col4] at (v6) {};
        \node[vertex, fill=col4] at (v7) {};

        \node[] at (0.8,-1.8) {$\TUR{(3,2,2)}{a}$};
    \end{scope}

    \begin{scope}[xshift=10.35cm, yshift = 0cm]
        \node[vertex] at (0:1.35) (v7) {};
        \node[vertex] at (360/7:1.35) (v5) {};
        \node[vertex] at (2*360/7:1.35) (v3) {};
        \node[vertex] at (3*360/7:1.35) (v1) {};
        \node[vertex] at (-3*360/7:1.35) (v2) {};
        \node[vertex] at (-2*360/7:1.35) (v4) {};
        \node[vertex] at (-360/7:1.35) (v6) {};
        
        \draw[edge-prep] (v2)--(v1)--(v3)--(v2)--(v4)--(v1)--(v5)--(v2)--(v6)--(v1)--(v7)--(v2);
        \draw[edge-c1] (v2)--(v1)--(v3)--(v2)--(v4)--(v1)--(v5)--(v2)--(v6)--(v1)--(v7)--(v2);
        \draw[edge-prep] (v4)--(v3)--(v5)--(v4)--(v6)--(v3)--(v7)--(v4);
        \draw[edge-c2] (v4)--(v3)--(v5)--(v4)--(v6)--(v3)--(v7)--(v4);
        \draw[edge-prep] (v5)--(v6)--(v7)--(v5);
        \draw[edge-c3] (v5)--(v6)--(v7)--(v5);
        
        \node[vertex, fill=col2] at (v1) {};
        \node[vertex, fill=col2] at (v2) {};
        \node[vertex, fill=col3] at (v3) {};
        \node[vertex, fill=col3] at (v4) {};
        \node[vertex, fill=col4] at (v5) {};
        \node[vertex, fill=col4] at (v6) {};
        \node[vertex, fill=col5] at (v7) {};

        \node[] at (0,-1.98) {$\TUR{(2,2,2,1)}{a}$};
    \end{scope}

    \begin{scope}[xshift = 12.35cm, yshift = -0cm, scale=1]
        \node[] at (0,0.8) (l1) {};
        \node[] at (0.8,0.8) (l2) {};

        \node[] at (0,0.4) (l3) {};
        \node[] at (0.8,0.4) (l4) {};

        \node[] at (0,0.0) (l5) {};
        \node[] at (0.8,0.0) (l6) {};

        \node[] at (0,-0.4) (l7) {};
        \node[] at (0.8,-0.4) (l8) {};

        \node[] at (0,-0.8) (l9) {};
        \node[] at (0.8,-0.8) (l10) {};
    
        \draw[edge-c1, ultra thick] (l1)--(l2);
        \draw[edge-c2, ultra thick] (l3)--(l4);
        \draw[edge-c3, ultra thick] (l5)--(l6);
        \draw[edge-c4, ultra thick] (l7)--(l8);
        \draw[edge-c5, ultra thick] (l9)--(l10);

        \node[anchor=west] at (0.8, 0.8) {$a + 1$};
        \node[anchor=west] at (0.8, 0.4) {$a$};
        \node[anchor=west] at (0.8, 0.0) {$a - 1$};
        \node[anchor=west] at (0.8, -0.4) {$a - 2$};
        \node[anchor=west] at (0.8, -0.8) {$a - 3$};
    \end{scope}
\end{tikzpicture}
\caption{Some examples of generalised Turán patterns. The colours of vertices and edges indicates their multiplicity.}
\label{fig:generalized-turan}
\end{figure}

\begin{remark}
\hfill
\begin{enumerate}[ --- ]

\item
For $d=0$ and $a=0$, the generalised Turán pattern is exactly the graph pattern whose blow-ups give rise to the $r_0$-partite Turán graph, while for $d=0$ and $r_0=1$, it is the `all $a$' pattern, whose blow-ups are multigraphs in which every edge has multiplicity equal to $a$.
    
\item
We may think of the quantity $a$ as the `ambient edge multiplicity' of a generalised Turán pattern, and of the pattern itself as recording  deviations from $a$. Indeed, it is easily checked that $a\leq \sigma_{\TUR{\bfr}{a}} < a+1$, while $\log (a)\leq \pi_{\TUR{\bfr}{a}} < \log (a+1)$.
    
\item
Day, Falgas-Ravry and Treglown provided a recursive algorithm for computing $\pi_{\TUR{\bfr}{a}}$~\cite[Corollary 12.6]{day2022extremal}. Further, they introduced a linear order on the pairs $(\bfr, a)$ by setting $(\bfr, a) \prec (\bfr', a')$ if either (i) $a<a'$, (ii) $a=a'$, $r_j=r'_j$ for all $j$ with $j<i$ and $r_i <r'_i$, or (iii) $a=a'$ and $\bfr$ is an initial segment of $\bfr'$, and showed~\cite[Proposition 12.7]{day2022extremal} that this implies $\pi_{\TUR{\bfr}{a}} \leq \pi_{\TUR{\bfr'}{a'}}$. Their results extend mutatis mutandis to $\sigma_{\TUR{\bfr}{a}}$.

\end{enumerate}
\end{remark}
 
For their additive problem for integer-weighted graphs, Füredi and Kündgen provided recursive constructions of patterns of exactly this form --- though it requires a little work to read this out of their argument, and they allowed $a< d$ since they were working with integer-weighted graphs rather than just multigraphs. 
 One of the major advantages of the integer-weighted graph setting is that they could exploit the fact that $\ex_\ZZ(n, s, q+\binom{s}{2})= \ex_\ZZ(n, s, q)+\binom{n}{2}$ to restrict their investigations to the case $0\leq q< \binom{s}{2}$, a relation that fails for multigraphs. Their upper-bound followed from an elementary averaging argument, and some careful estimates of floor functions --- see \Cref{cor: arithmetic averaging}. As a consequence of their work, it follows that for all $s\geq 2$ fixed and all $q>(s-1)\binom{s}{2}$, there is a generalised Turán pattern $T = \TUR{\bfr}{a}$ such that $\Blow_n(T)\subseteq \mathcal{A}(n,s,q)$ (i.e.\ blow-ups of $T$ are $(s,q)$-graphs) and $\sigma_T= \ex_\Sigma(s,q)$ both hold.

For the multiplicative problem for multigraphs, as noted above, we have that $\ex_\Pi\p[\big]{n, s, a\binom{s}{2}} = a^{\binom{n}{2}}$ for all integers $s\geq 2$ and $a\geq 1$. In a first paper on the subject~\cite{Mubayi2020-fb}, Mubayi and Terry extended this in two ways.

\begin{proposition}[Mubayi--Terry]
\label{prop: MT results 1}
Let $s\in \ZZ_{\geq 2}$ and $a,q\in \ZZ_{\geq 1}$. The following hold:
\begin{enumerate}[(i)]
    \item If $a \binom{s}{2} \leq q \leq a\binom{s}{2} + s - 2$, then we have $\ex_\Pi(n, s, q) = a^{\binom{n}{2}+o(n^2)} = \left(\PI{\TUR{(1)}{a}}{n} \right)^{1+o(1)}$.
    \item If $q = (a + 1)\binom{s}{2} - t$ for some $1 \leq t \leq s/2$, then we have $\ex_\Pi(n, s, q) = \PI{\TUR{(s-t)}{a}}{n}$.    
\end{enumerate}
\end{proposition}

The smallest case of \Cref{problem: Mubayi Terry} left open by these results was 
\begin{equation*}
    (s,q)
    = (4, 15)
    = \p[\big]{4, 2 \tbinom{4}{2} +3 }
    = \p[\big]{4, \SIGMA{\TUR{(1,1)}{2}}{4}}.
\end{equation*}
This was the subject of a second paper of Mubayi and Terry~\cite{Mubayi2019-le}, who showed that
\begin{align}
\label{eq: MT 4 15}
    \ex_\Pi\p[\big]{n, 4, 2\tbinom{4}{2} + 3} = \PI{\TUR{(1,1)}{2}}{n} 
\end{align}
holds for all $n$ sufficiently large.
They also observed that $\PI{\TUR{(1,1)}{2}}{n} = e^{\pi_{\TUR{(1,1)}{2}} \binom{n}{2} + o(n^2)}$, where $e^{\pi_{\TUR{(1,1)}{2}}}$ is a transcendental number (assuming Schanuel's conjecture from number theory), and that in product-optimal blow-ups of $\TUR{(1,1)}{2}$, the proportion of vertices assigned to each of the pattern's two parts is transcendental.

Mubayi and Terry conjectured that~\eqref{eq: MT 4 15} could be generalised to the case $(s,q) = (4, a \binom{4}{2} + 3)$ for all $a\in \ZZ_{\geq 3}$, with $\TUR{(1,1)}{a}$ replacing $\TUR{(1,1)}{2}$. 
This conjecture was proved in full by Day, Falgas--Ravry and Treglown~\cite{day2022extremal}, alongside a number of additional cases of \Cref{problem: Mubayi Terry}.
Day, Falgas--Ravry and Treglown further proposed in~\cite[Conjecture 3.2]{day2022extremal} a more general form of Mubayi and Terry's conjecture:

\begin{conjecture}[Day, Falgas--Ravry and Treglown]
\label{conj: main}
Let $a, r, s, d \in \ZZ_{\geq 0}$ with $a, r \geq 2$, $a > d$ and $s \geq (r - 1)(d + 1) + 2$.
Let $\bfr$ denote the $(d + 1)$-tuple $\bfr = (r - 1, 0, 0, \dotsc, 0, 1)$.
Then for all $n$ sufficiently large, we have
\begin{align*}
    \ex_\Pi\p[\big]{n, s, \SIGMA{\TUR{\bfr}{a}}{s} } = \PI{\TUR{\bfr}{a}}{n}.
\end{align*}
\end{conjecture}

\begin{remark}
\vbox{ 
\hfill
\begin{enumerate}[ --- ]
    \item
    \Cref{conj: main} implies the strongly counter-intuitive fact that to maximise the geometric mean of the edge multiplicities, it can be advantageous to use constructions in which the span of multiplicities is large --- quite the opposite of what one would expect given the integer AM--GM inequality (see \Cref{prop: integer AM-GM}).
    
    \item
    The condition $s\geq (r-1)(d+1)+2$ is strictly necessary: this is the smallest $s$ which can `detect' or `tell apart' blow-ups of the generalised Turán pattern $\TUR{\bfr}{a}$ from other generalised Turán patterns $T$ with $\SIGMA{T}{s} = \SIGMA{\TUR{\bfr}{a}}{s}$ and $\pi_T > \pi_{\TUR{\bfr}{a}}$.
    
    \item
    Day, Falgas--Ravry and Treglown showed in~\cite[Theorem 3.11]{day2022extremal} that for $\bfr$ fixed and all $a$ sufficiently large, the base case $s=(r-1)(d+1)+2$ implies the case $S> (r-1)(d+1)+2$.
    
    \item Generalising the earlier work of Mubayi and Terry, Day, Falgas--Ravry and Treglown also proved the case $d=0$ of \Cref{conj: main} in full.
\end{enumerate}
}
\end{remark}

Day, Falgas--Ravry and Treglown observed that their conjecture did not, however, cover all pairs $(s,q)$ in \Cref{problem: Mubayi Terry}, and they asked in~\cite[Question 12.3]{day2022extremal} whether an even more general form of \Cref{conj: main} might hold, with $\bfr$ allowed to take values among more general $(d+1)$-tuples $(r_0, r_1, \dotsc, r_d)$ with non-negative integer entries and $r_0, r_d>0$. 

Even a positive answer to that question would not, however, cover all pairs $(s, q)$, and the authors of~\cite{day2022extremal} posed in~\cite[Section 13]{day2022extremal} a number of further questions regarding these uncovered pairs, with the cases $(s,q) = (5, 2\binom{5}{2}+4)$ 
and $(5, a\binom{5}{2}+7)$, 
singled out 
as particularly interesting 
cases covered by neither of~\cite[Conjecture 3.2, Question 12.3]{day2022extremal}.

In~\cite{falgas2024extremal}, Falgas--Ravry asymptotically resolved the case $d=1$ of \Cref{conj: main}.
More precisely, he proved~\cite[Theorem 1.6]{falgas2024extremal} that for all integers $a \geq 2$ and the base case $s = (r-1)2 + 2 = 2r$, we have $\ex_\Pi\p[\big]{s, \SIGMA{\TUR{(r-1,1)}{a}}{s} } = e^{\pi_{\TUR{(r-1,1)}{a}}}$.
Further, he conjectured~\cite[Section 5]{falgas2024extremal} that one could prove stability versions of the results of~\cite{falgas2024extremal} and that $\ex_\Pi(s,q)$ should exhibit `flat intervals', that is, intervals of values of $q$ during which $\ex_\Pi(s,q)$ remains constant. In particular, he predicted in~\cite[Conjecture 5.1]{falgas2024extremal} the following:

\begin{conjecture}[Falgas--Ravry]
\label{conj: flat interval}
For every $r,a\in \ZZ_{\geq 1}$ and every $s\geq 2r+1$, we have that
\begin{equation*}
    \ex_\Pi \p[\big]{s, \SIGMA{\TUR{(r)}{a}}{s} }
    = \ex_\Pi \p[\big]{s, \SIGMA{\TUR{(r)}{a}}{s} + 1 }
    = \dotsb
    = \ex_\Pi \p[\big]{s, \SIGMA{\TUR{(r)}{a}}{s} + \floor{(s-1)/r}-1 }.
\end{equation*}
In particular, for all $0 \leq t \leq \floor{(s-1)/r}-1$, we have
\begin{equation*}
    \ex_\Pi \p[\big]{s, \SIGMA{\TUR{(r)}{a}}{s} + t} = a\p[\Big]{\frac{a+1}{a}}^{(r-1)/r}.
\end{equation*}
\end{conjecture}

The intuition given in~\cite{falgas2024extremal} in favour of the conjecture was that for values of $q$ between $\SIGMA{\TUR{(r)}{a}}{s}$ and $\SIGMA{\TUR{(r)}{a}}{s} + \floor{(s-1)/r}-1$, there is not enough room for a generalised Turán pattern that is both $(s,q)$-admissible and improves upon $\TUR{(r)}{a}$.

Finally and more recently, Gu and Wang~\cite{gu2024extremal} investigated the value of $\ex_\Pi(n,s,q)$ in the case $(s,q)=(5, a\binom{5}{2}+4)$ for $n$ fixed and $a \to \infty$, providing some evidence in favour of the truth of \Cref{conj: main} when $\bfr=(1,0,1)$ and $s=5$.



\subsection{Our contributions}\label{subsection: our contributions}

\subsubsection{Overview}

We contribute to the study of \Cref{problem: Mubayi Terry,problem: Mubayi Terry} in two different directions: we make a major advance in the case where the `ambient edge multiplicity' is large, i.e. for pairs $(s,q)$ where $a\binom{s}{2}\leq q<(a+1)\binom{s}{2}$ and $a\in \mathbb{Z}_{>0}$ is `suitably large' ($a=\Omega(d^2)$, where $a-d$ is the lowest edge multiplicity of the generalised Turán pattern whose optimality we are trying to establish).

In this regime, we are able to give a satisfactory general picture, and to prove asymptotically a more general form of \Cref{conj: main}. Furthermore, the tools we develop in doing so are in principle able to handle any specific case we do not cover in the large $a$ regime --- though in practice the computations involved become too horrible for this to be palatable. We also resolve \Cref{conj: flat interval}, and in addition we obtain stability versions of all of our results in the large $a$ regime, confirming another conjecture from~\cite{falgas2024extremal}.

On the other hand, we also consider the case where the `ambient edge multiplicity' $a$ is equal to $2$. For most pairs $(s,q)$, this is covered neither by \Cref{conj: main}, nor, in the additive case, by the previous work of Füredi and Kündgen. In this setting we demonstrate the subtlety of \Cref{problem: Furedi Kundgen,problem: Mubayi Terry} by reducing them to finite combinatorial problems with a design-theoretic flavour. We are able to fully resolve the said problems for $2\binom{s}{2}\leq q \leq (2+\frac{3}{13})\binom{s}{2}+O(s)$, and show inter alia that for an interval of $q$ of length $\Omega(s^2)$, blow-ups of a Petersen-like pattern are optimal for the multiplicative \Cref{problem: Mubayi Terry}, while the additive \Cref{problem: Furedi Kundgen} exhibits extreme forms of instability.

\subsubsection{Large ambient multiplicity}

\begin{theorem}
\label{theorem: base case main conj is true}
Let $r, d \in \ZZ_{\geq 0}$ with $r\geq 2$. Set $s= (r-1)(d+1)+2$, and let $\bfr$ denote the $(d+1)$-tuple $\bfr=(r-1,0,0, \dotsc, 0, 1)$. Then for all integers $a> d^2+d$, we have: 
\begin{align*}
\ex_\Pi\left(s, \SIGMA{\TUR{(\bfr)}{a}}{s} \right)  = e^{\pi_{\TUR{\bfr}{a}}}.
\end{align*}
\end{theorem}

\begin{corollary}
\label{cor: main conj asymptotically true}
\Cref{conj: main} is asymptotically true: for all $r, d\in \ZZ_{\geq 0}$ with $r \geq 2$ and all $s\geq (r-1)(d+1)+2$, letting $\bfr$ denote the $(d+1)$-tuple $\bfr=(r-1,0,0, \dotsc, 0, 1)$, we have
\begin{align*}
\ex_\Pi\left(n, s, \SIGMA{\TUR{(\bfr)}{a}}{s} \right)  = \PI{\TUR{\bfr}{a}}{n}^{1+o(1)}
\end{align*}
holding for all $a$ sufficiently large.
\end{corollary}

Further, we conjecture that our results form a part of the following more general picture completely describing the behaviour of the Mubayi--Terry problem for sufficiently large ambient edge multiplicities.

\begin{conjecture}[Large $a$ conjecture]
\label{conj: large a regime generalised Turan optimal}
Let $s\in \ZZ_{\geq 2}$ and let $0\leq q < \binom{s}{2}$ be an integer. Then there exists a unique generalised Turán pattern $T=\TUR{\bfr}{a}$ such that blow-ups of $T$ are $(s,a\binom{s}{2}+q)$-graphs and for all sufficiently large $a$ we have:
\begin{align}
\label{eq: conjecture large a}
    \ex_\Pi\p[\big]{n, s, a\tbinom{s}{2} + q}  = \PI{\TUR{\bfr}{a}}{n}^{1+o(1)}.
\end{align}
\end{conjecture}

Using our methods in the proof of \Cref{theorem: base case main conj is true}, we prove several further cases of \Cref{conj: large a regime generalised Turan optimal}.

\begin{theorem}
\label{theorem: large a arbitrary number of deficient blobs}
For all $r_0, r_d, d\in \ZZ_{> 0}$ and all $s\geq r_0(d r_d+1)+r_d+1$, letting $\bfr$ denote the $(d+1)$-tuple $\bfr=(r_0, 0,0,0, \dotsc, 0, r_d)$, we have that
\begin{align*}
\ex_\Pi\left(n, s, \SIGMA{\TUR{(\bfr)}{a}}{s} \right)  = \PI{\TUR{\bfr}{a}}{n}^{1+o(1)}
\end{align*}
holds for all $a$ sufficiently large.
\end{theorem}

\begin{remark}
\hfill
\begin{enumerate}[ --- ]
    \item
    The elaborate structure of the extremal examples in \Cref{theorem: large a arbitrary number of deficient blobs} should be highlighted: product-maximising blow-ups of the generalised Turán patterns $\TUR{\bfr}{a}$ with these parameters involve edges with multiplicities $a+1, a, a-d+1$ and $a-d$, and two kinds of parts. The fact that it can be optimal to use four different edge multiplicities (including, for $d$ large, multiplicities vastly apart) in order to maximise the product of the multiplicities is striking.
    
    \item
    $\PI{\TUR{\bfr}{a}}{n} = e^{\pi_{\TUR{\bfr}{a}} \binom{n}{2} +o(n^2)}$, and we can give an explicit formula for $e^{\pi_{\TUR{\bfr}{a}}}$.
    Indeed,
    \begin{align*}
    e^{\pi_{\TUR{\bfr}{a}}} = a \p[\Big]{\frac{a+1}{a}}^{(r_0 - 1 + r_d x_\star)/r_0}
    \end{align*}
    where $0 \leq x_\star \leq 1$ is defined as
    \begin{equation*}
        x_\star \defined  \frac{\log((a+1)/a)}{\log \p[\big]{ (a+1)^{r_d(r_0+1)}  \big/ a^{r_d} (a-d+1)^{r_0(r_d -1)}(a-d)^{r_0}}}.
    \end{equation*}
    (See \Cref{prop: product-optimal blow ups of generalised Turan}.)
    Further, by generalisations of~\cite[Proposition A2]{day2022extremal} all of $x_\star=x_\star(\bfr, a)$, $\pi_{\TUR{\bfr}{a}}$ and $e^{\pi_{\TUR{\bfr}{a}}}$ are transcendental (conditional on Schanuel's conjecture, and relying crucially on Mih\v{a}ilescu's proof of the Catalan conjecture).         
         
    \item
    The condition $s\geq r_0(d r_d+1)+r_d+1$ is exactly the right generalisation of the condition $s\geq (r-1)(d+1)+2$ in \Cref{conj: main}: it is the smallest value of $s$ able to `detect'  $\TUR{\bfr}{a}$.
     
    
    \item
    Our proof of \Cref{theorem: large a arbitrary number of deficient blobs} also provides a unified, streamlined and much simpler proof of the main results of~\cite{day2022extremal, falgas2024extremal, Mubayi2019-le}.  Indeed, specialising our proof to the case of what was the previously most general result~\cite[Theorem 1.6]{falgas2024extremal}, we more than halve the case analysis and calculations required for the result.
\end{enumerate}
\end{remark}

Settling a conjecture of Falgas--Ravry~\cite{falgas2024extremal} in a strong form, we obtain a stability version of \Cref{theorem: large a arbitrary number of deficient blobs}.

\begin{theorem}
\label{theorem: stability}
For all $r_0, r_d, d\in \ZZ_{> 0}$, all $s\geq r_0(d r_d+1)+r_d+1$ and all $a$ sufficiently large, letting $\bfr$ denote the $(d+1)$-tuple $\bfr \defined (r_0, 0,0,0, \dotsc, 0, r_d)$, and $q \defined \SIGMA{\TUR{(\bfr)}{a}}{s} $, we have that every $(s,q)$-graph $G$ on $n$ vertices with $P(G)\geq \PI{\TUR{\bfr}{a}}{n}^{1-o(1)}$ is $o(n^2)$-close in edit distance to a product-maximising blow-up of $\TUR{\bfr}{a}$.
\end{theorem}

Finally, we prove the flat intervals conjecture of Falgas--Ravry, \Cref{conj: flat interval}, which in particular answers a question of Day, Falgas--Ravry and Treglown about the case $(s,q)=(5, a\binom{5}{2}+7)$:

\begin{theorem}
\label{theorem: flat intervals}
For every $r,a \in \ZZ_{\geq 1}$ and every $s\geq 2r+1$, we have that
\begin{equation*}
    \ex_\Pi \p[\big]{s, \SIGMA{\TUR{(r)}{a}}{s} }
    = \ex_\Pi \p[\big]{s, \SIGMA{\TUR{(r)}{a}}{s} + 1 }
    = \dotsb
    = \ex_\Pi \p[\big]{s, \SIGMA{\TUR{(r)}{a}}{s} + \floor{(s-1)/r}-1 }.
\end{equation*}
In particular, for all $0 \leq t \leq \floor{(s-1)/r}-1$, we have
\begin{equation*}
    \ex_\Pi \p[\big]{s, \SIGMA{\TUR{(r)}{a}}{s} + t} = a\p[\Big]{\frac{a+1}{a}}^{(r-1)/r}.
\end{equation*}
\end{theorem}

Our methods are in principle able to show that $\ex_\Pi\left(s, \SIGMA{\TUR{(\bfr)}{a}}{s} \right)=e^{\pi_{\TUR{(\bfr)}{a}}}$ for arbitrary fixed $\bfr$ and $s$ and $a$ sufficiently large, though the computations involved become increasingly intricate.
We discuss further cases of \Cref{conj: large a regime generalised Turan optimal} we were able to verify as well as the limitations of our approach in \Cref{section: concluding remarks}.
One of our main innovations on previous work is the notion of `heavy sets' --- large regular structures that one can find in near extremal multigraphs; see \Cref{section: large a strategy} for details.

\subsubsection{Small ambient multiplicity}

For small $a$, many of the constructions corresponding to generalised Turán patterns are not available, since they would involve non-positive edge multiplicities.  In this paper, we focus on the case $a=1$ for the additive Füredi--Kündgen multigraph problem and on the case $a=2$ for the multiplicative Mubayi--Terry multigraph problem --- the first challenging cases for these two problems.

In both cases, for the given value of $a$, we fully resolve the problem for all pairs $(s,q)$ with $a\binom{s}{2}\leq q< a\binom{s}{2}+\lfloor\frac{3s^2+13s}{26}\rfloor$, and show that we see extremal constructions very different from blow-ups of generalised Turán patterns --- and in particular that we get significant instability for the additive case, and blow-ups of Petersen-like patterns for the multiplicative case.

To state our results more precisely, we must introduce some notation and patterns.
Given a graph $G$ and an integer $a\in \ZZ_{\geq 1}$, we denote by $G^{(a)}$ the multigraph pattern on $V(G)$ in which
\begin{enumerate}[ --- ]
    \item
    for every vertex $v \in V(G)$, the loop $\set{v}$ has multiplicity $a - 1$;
    
    \item
    for every non-edge $uv \in E(\overline{G})$, the pair $\set{u,v}$ has multiplicity $a$;
    
    \item
    for every edge $uv \in E(G)$, the pair $\set{u,v}$ has multiplicity $a + 1$.
\end{enumerate}
We shall consider several patterns of this form in the special cases $a=1$ and $a=2$, and in this subsection as well as in \Cref{section: small a}, we write $G$ to denote the pattern $G^{(a)}$ whenever there is no risk of confusion and the value of $a$ is clear from context.

We write $K_r$ for the complete graph on $r$ vertices, $P_\ell$ for the path on $\ell$ vertices, $C_\ell$ for the cycle on $\ell$ vertices, and $K_{1, \ell}$ for the star on $\ell +1$ vertices.  We further consider a family $\cC_5$ of \emph{subcubic} (maximum degree at most $3$) graphs of \emph{girth} (length of the shortest cycle) at least $5$, which we now define.

\begin{figure}[ht]
\centering
\begin{tikzpicture}[scale=1]
    \tikzstyle{vertex_style}=[circle, fill = black, inner sep=0pt, minimum size=6pt]
    \tikzstyle{edge_style}=[very thick, black]

    \begin{scope}[xshift = 0cm]
        \node[vertex_style] (0) at (-0.5,0){};
        \node[vertex_style] (1) at (0.5,-0.5){};
        \node[vertex_style] (2) at (0.5,0){};
        \node[vertex_style] (3) at (0.5,0.5){};
    
        \draw[edge_style] (0) -- (1);
        \draw[edge_style] (0) -- (2);
        \draw[edge_style] (0) -- (3);
        \node[] at (0,-0.9) {$K_{1,3}$};
    \end{scope}

    \begin{scope}[xshift = 2.25cm]
        \begin{scope}[rotate=90]
            \foreach \x/\y in {0/1,72/2,144/3,216/4,288/5}{
                \node[vertex_style] (\y) at (canvas polar cs: radius=0.65cm, angle=\x){};
            }
        \end{scope}
        \draw[edge_style] (1) -- (2) -- (3) -- (4) -- (5) -- (1);
        \node[] at (0,-0.9) {$C_5$};
    \end{scope}

    \begin{scope}[xshift = 4.25cm]
        \node[vertex_style] (0) at (-0.4,-0.5){};
        \node[vertex_style] (1) at (-0.4,0){};
        \node[vertex_style] (2) at (-0.4,0.5){};
        \node[vertex_style] (3) at (0.4,-0.5){};
        \node[vertex_style] (4) at (0.4,0){};
        \node[vertex_style] (5) at (0.4,0.5){};
    
        \draw[edge_style] (0) -- (1) -- (2);
        \draw[edge_style] (3) -- (4) -- (5);
        \draw[edge_style] (1) -- (4);
        \node[] at (0,-1) {$H_6$};
    \end{scope}

    \begin{scope}[xshift = 6.35cm]
        \begin{scope}[rotate=0]
            \foreach \x/\y in {0/1,60/2,120/3,180/4,240/5,300/6}{
                \node[vertex_style] (\y) at (canvas polar cs: radius=0.75cm, angle=\x){};
            }
        \end{scope}
        \node[vertex_style] (0) at (0,0){};
        \draw[edge_style] (1) -- (2) -- (3) -- (4) -- (5) -- (6) -- (1);
        \draw[edge_style] (1) -- (0) -- (4);
        \node[] at (0,-1.) {$H_7$};
    \end{scope}

    \begin{scope}[xshift = 8.7cm]
        \begin{scope}[rotate=0]
            \foreach \x/\y in {0/1,60/2,120/3,180/4,240/5,300/6}{
                \node[vertex_style] (\y) at (canvas polar cs: radius=0.85cm, angle=\x){};
            }
        \end{scope}
        \begin{scope}[rotate=-30]
            \foreach \x/\y in {0/7,120/8,240/9}{
                    \node[vertex_style] (\y) at (canvas polar cs: radius=0.35cm, angle=\x){};
                }
        \end{scope}
        \draw[edge_style] (1) -- (2) -- (3) -- (4) -- (5) -- (6) -- (1);
        \draw[edge_style] (1) -- (8) -- (4);
        \draw[edge_style] (2) -- (7) -- (5);
        \draw[edge_style] (3) -- (9) -- (6);
        \node[] at (0,-1.1) {$H_9$};
    \end{scope}

    \begin{scope}[xshift = 11.35cm]
        \begin{scope}[rotate=90]
            \foreach \x/\y in {0/1,72/2,144/3,216/4,288/5}{
                \node[vertex_style] (\y) at (canvas polar cs: radius=0.5cm, angle=\x){};
            }
            \foreach \x/\y in {0/6,72/7,144/8,216/9,288/10}{
                \node[vertex_style] (\y) at (canvas polar cs: radius=1cm, angle=\x){};
            }
        \end{scope}
     
        \foreach \x/\y in {1/6,2/7,3/8,4/9,5/10}{
            \draw[edge_style] (\x) -- (\y);
        }
        \foreach \x/\y in {1/3,2/4,3/5,4/1,5/2}{
            \draw[edge_style] (\x) -- (\y);
        }
        
        \foreach \x/\y in {6/7,7/8,8/9,9/10,10/6}{
            \draw[edge_style] (\x) -- (\y);
        }
        \node[] at (0,-1.2) {$\Petersen$};
    \end{scope}
\end{tikzpicture}
\caption{The family of graphs $\cC_5$.}
\label{fig:family c5}
\end{figure}

Recall that an $(r,g)$-cage is an $r$-regular graph of girth $g$.
We let $\Petersen$ be the unique up to isomorphism $(3,5)$-cage on $10$ vertices.
We let $H_6$ denote the graph on $[6]$ with edges $\set{12, 13, 14, 25, 26}$ (the `$H$-graph'), $H_7$ the graph obtained by glueing two copies of $C_5$ along a path on three vertices (i.e.\ the graph on $[7]$ with edges $\set{12, 23, 34, 45, 15, 16, 67, 74}$), and $H_9$ the graph obtained from the hexagon by joining opposite vertices by a copy of $P_2$ (i.e.\ the graph on $[9]$ with edges $\set{12, 23, 34, 45, 56, 16, 17, 74, 28, 85, 39, 96}$).
Finally, we let $\cC_5$ denote the family $\set{K_{1,3}, C_5, H_6, H_7, H_9, \Petersen}$.
See \Cref{fig:family c5}.

For each of the patterns $P$ associated to the graphs defined above, we determine the maximal sum of edge multiplicities $\SIGMA{P}{s}$ in a blow-up of $P$ on $s$ vertices. 

\begin{proposition}
\label{prop: sum for optimal blow-ups of stars paths and cycles}
The following hold for all $a\geq 1$:
\begin{enumerate}[(i)]
    \item For all $\ell \geq 2$, we have $\SIGMA{K_{1,\ell}^{(a)}}{s} \leq a\binom{s}{2}+ \left \lfloor \frac{\ell-1}{6\ell+2}s^2+\frac{s}{2} \right \rfloor$, with equality for $\ell \in \set{2,3,4}$.

    \item For all $\ell_1 \geq 4$ and all $\ell_2 \geq 6$, we have $\SIGMA{P_{\ell_1}^{(a)}}{s} = \SIGMA{C_{\ell_2}^{(a)}}{s} = a\binom{s}{2}+\left \lfloor \frac{s^2}{12}+\frac{s}{2}\right \rfloor$.

    \item For all $3 \leq \ell \leq 5$, we have $\SIGMA{C_\ell^{(a)}}{s} = a \binom{s}{2}+ \left \lfloor \frac{s^2}{2\ell}+\frac{s}{2} \right \rfloor$.

    \item For the pattern $P_{+2}=P_{+2}^{(a)}$ consisting of an edge of multiplicity $a+2$ with multiplicity $a-1$ loops around its endpoints, we have $\SIGMA{P_{+2}}{s} = a\binom{s}{2}+ \left \lfloor \frac{s^2}{4}+\frac{s}{2} \right \rfloor$.
\end{enumerate}
\end{proposition}

\begin{proposition}
\label{prop: sum-optimal blow-ups for C5}
For all $G \in \cC_5$, we have $\SIGMA{G^{(a)}}{s} = a\binom{s}{2} + \left \lfloor \frac{s^2}{10}+\frac{s}{2}\right \rfloor$.
\end{proposition}

With these results in hand, we can state our results.

\subsubsection{Small ambient multiplicity on the Füredi--Kündgen multigraph problem}

We first observe that the case $a=0$ of \Cref{problem: Furedi Kundgen} follows from previous work on the Erdős graph problem for $\ex(n,s,q)$ mentioned in the introduction. (Recall that the asymptotics of $\ex(n,s,q)$ follow easily from the Erdős--Stone theorem.)

\begin{proposition}
\label{prop: a=0 case additive}
Let $s\in \ZZ_{\geq 2}$, and let $q$ be an integer with: $0\leq q<\binom{s}{2}$. Then
\begin{align*}
   \ex_\Sigma(n,s, q) =\ex(n,s,q)+o(n^2). 
\end{align*}
\end{proposition}

The first non-trivial case of \Cref{problem: Furedi Kundgen} is thus when the ambient edge multiplicity $a=1$. For this case, we prove the following:

\begin{theorem}
\label{theorem: additive}
Let $s\in \ZZ_{\geq 2}$, and let the ambient edge multiplicity be $a=1$ in all graph patterns below. Then the following hold:
\begin{enumerate}[(i)]
    \item for $\binom{s}{2}\leq q< \SIGMA{K_{1,2}}{s}$, we have $\ex_\Sigma(s, q) = 1$;
    
    \item for $\SIGMA{K_{1,2}}{s} \leq q < \SIGMA{P_4}{s}$, we have $\ex_\Sigma(s, q) = 1 + 1/7$;
    
    \item for $\SIGMA{P_4}{s} \leq q < \SIGMA{K_{1,3}}{s}$, we have $\ex_\Sigma(s, q) = 1 + 1/6$;

    \item for $\SIGMA{K_{1,3}}{s} \leq q < \SIGMA{K_{1,4}}{s} $, we have $\ex_\Sigma(s, q) = 1 + 1/5$.
\end{enumerate}
\end{theorem}

\begin{remark}[Instability]
\Cref{theorem: additive} shows that the Füredi--Kündgen problem exhibits massive instability:
\begin{enumerate}[ --- ]
    \item Tight lower-bound constructions for (iii) include sum-optimal blow-ups of the graph patterns $P_4$, $P_5$ and $C_6$, which are pairwise far from each other in edit distance (i.e.\ edit distance $\Omega(n^2)$ for $n$-vertex multigraphs). 
    
    \item Tight lower-bound constructions for (iv) include blow-ups of every graph pattern in $\cC_5$ (all of which are uniform blow-ups, with the exception of the blow-up of the $K_{1,3}$ pattern). What is more, for every fixed $x\in [0,\frac{1}{5}]$, there exist asymptotically sum-optimal blow-ups of the Petersen pattern on $n$ vertices in such a way that half of the vertices $v$ are blown-up to sets $U_v$ containing  $xn +O(1)$ vertices, while the other half are blown-up to sets containing $(\frac{1}{5}-x)n +O(1)$ vertices. This gives a whole continuum of asymptotically extremal examples for (iv), all of which are pairwise far from each other in edit distance!
\end{enumerate}

In addition, tight lower-bounds for (i) and (ii) are given by sum-optimal blow-ups of the all $1$ pattern and of the $K_{1,2}$-pattern respectively. For some values of $q$ within the intervals covered by (i) and (ii), other near-extremal constructions are possible --- for example, a sum-optimal blow up of the pattern $K_2$ in the case where $\binom{s}{2}+\left(2\lfloor \frac{s^2}{4}\rfloor -\binom{s}{2}\right) \leq q < \SIGMA{K_{1,2}}{s} $.

In principle, our results allow us to determine the structure of all near-extremal multigraphs in terms of the solutions to an explicit, bounded optimisation problem --- see \Cref{remark: computational reduction in additive case}. 
\end{remark}

\subsubsection{Small ambient multiplicity on the Mubayi--Terry multigraph problem}

Much as for the additive problem, one can show that the case $a=1$ of \Cref{problem: Mubayi Terry} follows from previous work on the Erdős problem on $\ex(n,s,q)$. (Recall that the asymptotics of $\ex(n,s,q)$ follow easily from the Erdős--Stone theorem.)

\begin{proposition}
\label{prop: a=1 case multiplicative}
Let $s\in \ZZ_{\geq 2}$, and let $q$ be an integer with $\binom{s}{2}\leq q<2\binom{s}{2}$. Then
\begin{align*}
   \ex_\Pi(n,s, q) =2^{\ex(n,s,q-\binom{s}{2})+o(n^2)}. 
\end{align*}
\end{proposition}

Our results on the first non-trivial case $a=2$ are then as follows.

\begin{theorem}
\label{theorem: multiplicative}
Let $s\in \ZZ_{\geq 2}$, and let the ambient edge multiplicity be $a=2$ in all graph patterns below.
Then the following hold:
\begin{enumerate}[(i)]
    \item for $2\binom{s}{2}\leq q< \SIGMA{P_4}{s} $, $\ex_\Pi(s, q)= 2$;
    
    \item for $\SIGMA{P_4}{s} \leq q<\SIGMA{K_{1,3}}{s}  $, $\ex_\Pi(s, q)=  2 \left(\frac{9}{8} \right)^{\frac{1}{6}}$;

    \item for $\SIGMA{K_{1,3}}{s} \leq q< \SIGMA{K_{1,4}}{s} $, $\ex_\Pi(s, q)=  2 \left( \frac{27}{16}\right)^{\frac{1}{10}}$.
\end{enumerate}
\end{theorem}

\begin{remark}
One can read out of our proof of \Cref{theorem: multiplicative} that in this case we have stability (see \Cref{remark: extremal examples in small a}):
\begin{enumerate}[ --- ]
    \item in case (i), all near-extremal examples are close in edit distance (i.e.\ edit distance $o(n^2)$ for $n$-vertex multigraphs) to blow-ups of the all $2$ pattern;
    
    \item in case (ii), all near-extremal examples are close in edit distance to balanced blow-ups of the graph pattern $C_6^{(2)}$;
    
    \item in case (iii), all near-extremal examples are close in edit distance to balanced blow-ups of the $\Petersen^{(2)}$ pattern.
\end{enumerate}
\end{remark}

One can contrast the instability witnessed in \Cref{theorem: additive} with the stability seen in \Cref{theorem: multiplicative}. We believe that this is a more general phenomenon, part of a fundamental difference between the Füredi--Kündgen and Mubayi--Terry multigraph problems:

\begin{conjecture}[Geometric stability]
\label{conj: geometric stability}
For every pair of non-negative integers $(s,q)$ with $s\geq 2$, there exists a unique multigraph pattern $P$ such that every $(s,q)$-graph $G$ on $n$-vertices with $P(G)\geq \ex_\Pi(s,q)^{\binom{n}{2}+o(n^2)}$ lies within $o(n^2)$ edit distance of a product-optimal blow-up of $P$. Furthermore, if $a\binom{s}{2}\leq q < (a+1)\binom{s}{2}$, then every edge in $P$ has multiplicity at most $a+1$ while every loop in $P$ has multiplicity at most $a$.
\end{conjecture}

Finally, we note that in both \Cref{theorem: additive} and \Cref{theorem: multiplicative}, we see `flat intervals' of $q$-values of length $\Omega(s^2)$ in which neither of $\ex_\Sigma(s,q)$ nor $\ex_\Pi(s,q)$ increases, and that $\ex_\Sigma(s,q)$ jumps to a value strictly greater than $1$ for $q = \SIGMA{K_{1,2}^{(1)}}{s} = (1+\frac{1}{7})\binom{s}{2}+O(s)$, while $\ex_\Pi(s,q)$ jumps to a value strictly greater than $2$ only for $q = \SIGMA{P_4^{(2)}}{s} = (2+\frac{1}{6})\binom{s}{2}+O(s)$.

\subsection{Organisation of this paper}\label{subsection: structure of the paper}
In Section~\ref{section: general tools}, we develop some general tools for studying $(s,q)$-graphs. Our results for the small ambient multiplicity regime are gathered in Section~\ref{section: small a}, while those for the large ambient multiplicity regime can be found in Section~\ref{section: large a results}. We end the paper with some remarks and several open problems in Section~\ref{section: concluding remarks}.

\section{General preliminaries: tools in extremal multigraph theory}\label{section: general tools}

\subsection{Basic notions and notation}

Given a set $X$ and a positive integer $r$, we let $\binom{X}{r}$ denote the collection of subsets of $X$ of size $r$.
We set $[n] \defined \set{1,2,\dotsc, n}$.

As mentioned in the introduction, a multigraph is a pair $G=(V,w)$, where $V = V(G)$ is a set of vertices, and $w = w_G$ is a function $w \from \binom{V}{2} \to \ZZ_{\geq 0}$.
We usually write $ab$ for $\set{a,b}$ and, when the host multigraph is clear from context, we omit the subscript $G$ and simply write $w(ab)$ for $w_G(\set{a,b})$. A submultigraph of $G$ is a multigraph $G'=(V', w')$ with $V' \subseteq V$ and $w' \leq w_{\restriction V'}$, where $w_{\restriction V'}$ denotes the restriction of $w$ to $V'$. 
Such a submultigraph is called an induced submultigraph of $G$ if $w'=w_{\restriction V'}$.

Given a multigraph $G$ and a set $X \subseteq V(G)$, we define
\begin{align*}
    e_G(X) \defined \sum_{uv \in \binom{X}{2}} w(uv),
    \quad \text{and} \quad
    P_G(X) \defined \prod_{uv \in \binom{X}{2}} w(uv),
\end{align*}
to be the sum and product of the edge multiplicities inside $X$ respectively. Further, given a vertex $v \in V(G)$, we define 
\begin{align*}
    d_{G,X}(v) \defined \sum_{u\in X}w(uv),
    \quad \text{and} \quad
    p_{G, X}(v) \defined \prod_{u \in X} w(uv),
\end{align*}
to be the \emph{degree} and \emph{product degree} of $v$ in $X$ respectively. When the host multigraph $G$ is clear from context, we simply write $e(X)$, $P(X)$, $d_X(v)$ and $p_X(v)$ for $e_G(X)$, $P_G(X)$, $d_{G,X}(v)$ and $p_{G,X}(v)$. Further when $X=V$, we drop the subscript $X$ and write $d(v)$ and $p(v)$ rather than $d_V(v)$ and $p_V(v)$.

Given two multigraphs $G$ and $G'$ on a common vertex set $[n]$, we define the edit distance $\Delta(G,G')$ between them to be $\Delta(G,G') \defined \vert \set[\big]{xy \in \binom{[n]}{2} \st w_G(xy) \neq w_{G'}(xy) }\vert$.
We say that $G$ and $G'$ are $\delta n^2$-\emph{close} if $\Delta(G,G') \leq \delta n^2$.

\begin{definition}
A \emph{multigraph property} $\cP$ is a collection of multigraphs.
A property is:
\begin{enumerate}[ --- ]
    
    \item \emph{hereditary} if $\cP$ is closed under taking induced submultigraphs;
    
    \item \emph{monotone decreasing} if $\cP$ is closed under taking submultigraphs; 
    
    \item \emph{bounded} if there exists some absolute constant $M$ such that all multigraphs in $\cP$ have maximum edge multiplicity at most $M$.	
\end{enumerate}
\end{definition}

Given a family of multigraphs $\cF$, we denote by $\Forb(\cF)$ the family of multigraphs not containing a copy of a member of $\cF$ as a submultigraph, and observe that $\Forb(\cF)$ is a decreasing property of multigraphs. The property $\cA(s,q)$ of being an $(s,q)$-graph can be viewed as $\Forb(\cF)$, where $\cF$ is the finite family of multigraphs on $[s]$ with edge-sum $q+1$. Thus $\cA(s,q)$ is decreasing, and in addition is bounded (since it does not allow edges with multiplicity greater than $q$).

Given a bounded hereditary multigraph property $\cP$, we write $\cP_n$ for the collection of members of $\cP$ on the vertex set $[n]$, and we define
\begin{align*}
    \ex_\Sigma(n, \cP):= \max\set[\big]{e(G) \st G \in \cP_n },
    \quad \text{and} \quad
    \ex_\Pi(n, \cP) := \max\set[\big]{P(G) \st G\in \cP_n}.
\end{align*}
Adapting the classical averaging argument of Katona--Nemetz--Simonovits~\cite{katona1964graph}, one obtains the following result on limiting arithmetic and geometric densities for bounded symmetric hereditary properties of multigraphs.

\begin{lemma}
\label{lemma: multigraph Katona Nemetz Simonovits}
For every bounded symmetric hereditary property $\cP$ of multigraphs, the functions  $n \mapsto \ex_\Sigma(n, \cP)/ \binom{n}{2}$ and $n \mapsto \p[\big]{\ex_\Pi(n, \cP)}^{1/\binom{n}{2}}$ are non-increasing and tend to limits $\sigma(\cP)$ and $e^{\pi(\cP)}$ respectively as $n \to \infty$.\qed
\end{lemma}

We refer to the quantities $\sigma(\cP)$ and $e^{\pi(\cP)}$ as the (asymptotically extremal) arithmetic and geometric mean densities of $\cP$.
By the AM-GM inequality, it is immediate that $\pi(\cP) \leq \log (\sigma(\cP))$.
We shall often need the following integer-variant of the AM-GM inequality.

\begin{proposition}[Integer AM--GM inequality]
\label{prop: integer AM-GM}
Let $a, n \in \ZZ_{\geq 1}$, $0 \leq t \leq n$, and let $w_1, \dotsc, w_n$ be non-negative integers with $\sum_{i = 1}^n w_i = an + t$.
Then $\prod_{i=1}^n w_i \leq a^{n - t}(a + 1)^t$, with equality attained if and only if $t$ of the $w_i$ are equal to $a+1$ and the remaining $n-t$ are equal to $a$.
\end{proposition}

	

\subsection{Patterns}
\label{subsection: patterns}

Given a pattern $P$, observe that $\Blow(P)$ is a  bounded, hereditary property of multigraphs.
By \Cref{lemma: multigraph Katona Nemetz Simonovits}, it follows that there exist non-negative reals $\sigma_P$ and $\pi_P$ such that $\SIGMA{P}{n}= \sigma_P \binom{n}{2}+o(n^2)$ and $\PI{P}{n} = e^{\pi_P\binom{n}{2}+o(n^2)}$.
What is more, $\sigma_P$ and $\pi_P$ are obtained as the solutions to the following  quadratic optimisation problems:
\begin{align}
\label{eq: optimisation problem for arithmetic problem}
    \sigma_P \defined \max\set[\Big]{ \bfx^\rT A_P \, \bfx \st \bfx \in \RR_{\geq 0}^{V(P)} ,\, \textstyle{\sum_{v \in V(P)} x_v = 1}},
\end{align}
and, letting $A'_P$ be the matrix with entries $(A'_P)_{uv} \defined \log P(\set{u, v})$,
\begin{align}
\label{eq: optimisation problem for geometric problem}
    \pi_P \defined \max\set[\Big]{ \bfx^\rT A'_P \, \bfx \st \bfx \in \RR_{\geq 0}^{V(P)},\, \textstyle{\sum_{v \in V(P)} x_v = 1}}.
\end{align}

We call an optimal choice $\bfx$ for~\eqref{eq: optimisation problem for arithmetic problem} and~\eqref{eq: optimisation problem for geometric problem} a sum-optimal and a product-optimal weighting of $V(P)$ respectively.
As we show below, such weightings have a regularity property which helps computing them in practice.	

\begin{lemma}[Balanced degrees]
\label{lemma: almost balanced degrees in blow-ups of patterns}
Let $P$	be a multigraph pattern.
Then the following hold:
\begin{enumerate}[(i)]
    \item If $G\in \Blow_n(P)$ is sum-extremal, then it is almost regular: for every pair of distinct vertices $v,v'$ in $G$, we have $\abs{d(v)-d(v')} \leq w(vv')$;
    
    \item If $G\in \Blow_n(P)$ is product-extremal and both pairs and loop labels in $P$ are strictly positive, then $G$  is almost product-degree regular: for every pair of distinct vertices $v,v'$ in $G$, we have $ w(vv')^{-1}p(v')\leq p(v)\leq w(vv')p(v')$.
\end{enumerate}
\end{lemma}
\begin{proof}
In a blow-up of $P$, all vertices belonging to the same part of the $P$-partition have the same degree and the same product-degree. Shifting a vertex $v$ from its part in the $P$-partition of a multigraph $G$ to the part of some other vertex $v'$ increases $e(G)$ by at least $d(v')-d(v)-w(vv')$ (which is tight if edges inside $v'$'s part all have multiplicity $0$), and increases $P(G)$ by a multiplicative factor of at least $\frac{p(v')}{p(v)w(vv')}$ (which is tight if edges inside $v'$'s part all have multiplicity $1$). Both parts of the lemma then follow from the hypothesized sum- and product-extremality repectively.
\end{proof}
\noindent An essentially identical weight-shifting argument yields the following corollary on optimal weightings.
\begin{corollary}
\label{cor: balanced degrees in optimal pattern weightings}
Let $P$ be a multigraph pattern. Then any sum-optimal weighting $\bfx$ of $V(P)$ satisfies,
\begin{equation*}
    \sigma_P \defined P(\set{v})x_v+ \sum_{v'\in V(P)\setminus \set{v}} P(\set{v,v'})x_{v'}
    \quad \text{for all $v \in V(P)$ with $x_v > 0$}.
\end{equation*}
Similarly,  any product-optimal weighting $\bfx$ of $V(P)$ satisfies,
\begin{equation*}
    \pi_P \defined \log (P(\set{v}))x_v +\sum_{v'\in V(P)\setminus \set{v}} \log(P(\set{v,v'}))x_{v'}
    \quad \text{for all $v \in V(P)$ with $x_v > 0$}.\qed
\end{equation*}
\end{corollary}

\begin{proposition}
\label{prop: sigma P is rational}
Let $P$ be a multigraph pattern. Then $\sigma_P$ is rational, and there exists a sum-optimal weighting $\bfx$ of $V(P)$ with rational weights.
\end{proposition}
\begin{proof}
For $\bfx=(x_v)_{v\in V(P)}$, consider the polynomial $f(\bfx) \defined \bfx^t A_P \bfx$. This is a degree $2$ polynomial in $\card{V(P)}$ variables with integer coefficients. 
We can use the method of Lagrangian multipliers to determine the maximum of $f(\bfx)$ subject to the constraint that $g(\bfx) \defined 1-\sum_{v\in V(P)} x_v$ satisfies $g(\bfx)=0$. Now the coordinates of $\nabla (f(\bfx) +\lambda g(\bfx))$ corresponding to the variables $(x_v)_{v\in V(P)}$ are linear polynomials in $\bfx$ with integer coefficients that vanish at a maximum of $f$. It follows that an optimal choice of $\bfx$ is the solution to a system of linear equations with integer coefficients, and hence that there exists such an optimal $\bfx$ with coordinates in $\QQ$. The result follows. 
\end{proof}


\subsection{Properties of almost extremal multigraphs}

\begin{lemma}[Low degree vertex removal]
\label{lemma: removal of low deg vertices}
Let $\cP$ be a bounded decreasing property of multigraphs. Then for every $\eps\in (0, 1/2)$, there exists $\eta>0$ such that for all $n$ sufficiently large, the following hold:
\begin{enumerate}[(i)]
    \item in every multigraph $G\in \cP_n$ with $e(G)\geq (\sigma(\cP)-\eta)\binom{n}{2}$, one can remove a set of at most $\eps n$ vertices to obtain a multigraph $G'\in \cP_{n'}$ in which every vertex has degree at least $(\sigma(\cP)-\eps)n'$;
    
    \item in every multigraph $G\in \cP_n$ with $P(G)\geq e^{(\pi(\cP)-\eta)\binom{n}{2}}$, one can remove a set of at most $\eps n$ vertices to obtain a multigraph $G'\in \cP_{n'}$ in which every vertex has product-degree at least $e^{(\pi(\cP)-\eps)n'}$.
\end{enumerate}
\end{lemma}
\begin{proof}
For part (i), pick $\eta>0$ sufficiently small and $n_0$ sufficiently large, ensuring in particular that for all $n\geq n_0$, we have $\ex_\Sigma(n, \cP) \leq (\sigma(\cP)+\eta)\binom{n}{2}$.

For $n\geq 2n_0$, consider a multigraph $G\in \cP_n$ with $e(G)\geq (\sigma(\cP)-\eta)\binom{n}{2}$. Repeatedly remove vertices of minimum degree from $G$ to obtain a multigraph sequence $G_0=G$, $G_1$, $\dotsc$, with $G_i$ being a multigraph on $n-i$ vertices. If for all $i\leq \eps n$ we have $\delta(G_i) <(\sigma(\cP)-\eps)(n-i)$, then, setting $n' \defined \lfloor (1-\eps )n\rfloor\geq n_0$, we have
\begin{align*}
    (\sigma(\cP)+\eta)\binom{n'}{2} &\geq e(G_{n'})\geq e(G)-(\sigma(\cP)-\eps)\left(\binom{n}{2}-\binom{n'}{2}\right)\\
    &\geq \left(\sigma(\cP)+\eta\right)\binom{n'}{2} +\eps\left(\binom{n}{2}-\binom{n'}{2}\right)- \eta \left(\binom{n}{2}+\binom{n'}{2}\right),
\end{align*}
which for $\eta$ chosen sufficiently small with respect to $\eps$ and $n_0$ sufficiently large yields a contradiction. Part (ii) is proved mutatis mutandis.
\end{proof}

\begin{lemma}
\label{lemma: passing to subproperties}
Let $\cP, \cQ$ be bounded decreasing properties of multigraphs. 
\begin{enumerate}[(i)]
    \item If for every $\eps>0$ fixed we have that for all $n$ sufficiently large and every multigraph $G\in \cP_n$, there exists a multigraph $G'\in \cQ_n$ with $e(G)\leq e(G') +\eps \binom{n}{2}$, then $\sigma(\cP)\leq \sigma(\cQ)$.
    
    \item If for every $\eps>0$ fixed we have that for all $n$ sufficiently large every multigraph $G\in \cP_n\setminus \cQ_n$ contains a vertex of degree at most $(c+\eps)n$, then $\sigma(\cP)\leq \max \set{ \sigma(\cP\cap \cQ), c }$. In particular if $\cQ$ is the empty property, then $\sigma(\cP)\leq c$.
    
    \item If for every $\eps>0$ fixed we have that for all $n$ sufficiently large and every multigraph $G\in \cP_n$, there exists a multigraph $G'\in \cQ_n$ with $P(G)\leq P(G')e^{\eps \binom{n}{2}}$, then $\pi(\cP)\leq \pi(\cQ)$.		

    \item If for every $\eps>0$ fixed we have that for all $n$ sufficiently large every multigraph $G\in \cP_n\setminus \cQ_n$ contains a vertex of product degree at most $e^{(c+\eps)n}$, then $\pi(\cP)\leq \max \set{ \pi(\cP\cap \cQ), c }$. In particular if $\cQ$ is the empty property, then $\pi(\cP)\leq c$.
\end{enumerate}	
\end{lemma}
\begin{proof}
Parts (i) and (iii) are trivial, while parts (ii) and (iv) are proved similarly to \Cref{lemma: removal of low deg vertices} by removing minimum-degree vertices until one either obtains a contradiction or a multigraph in $\cP\cap \cQ$; see~\cite[Proposition 2.14]{falgas2024extremal} for an explicit version of this argument. 
\end{proof}

\subsection{Arithmetic averaging and the additive problem}

For all bounded, hereditary properties of multigraphs $\cP$, an easy averaging argument over $n$-vertex subgraphs shows that 
\begin{align}
\label{eq: sum-averaging bound on ex sigma}
    \ex_\Sigma(n+1, \cP) \leq\left\lfloor \left(\frac{n+1}{n-1}\right)\ex_\Sigma(n, \cP)  \right\rfloor.
\end{align}
This simple fact has a useful corollary (which in fact was one of the main tools used by Füredi and Kündgen for proving upper bounds on $\ex_\ZZ(n,s,q)$):

\begin{corollary}
\label{cor: arithmetic averaging}
Let $f \from \ZZ_{\geq 0} \to \ZZ_{\geq 0}$.
Suppose $\cP$ is a bounded hereditary property of multigraphs with $\ex_\Sigma(s, \cP) \leq f(s)$, and that for every $n \geq s$ the inequality 
\begin{align*}
    \p[\Big]{\frac{n+1}{n-1}} f(n) < f(n + 1) + 1
\end{align*}
is satisfied. Then $\ex_\Sigma(n, \cP) \leq f(n)$ for all $n \geq s$. \qed
\end{corollary}

\subsection{Weighted geometric averaging and stability}

The following simple averaging proposition was a key tool used in~\cite{falgas2024extremal} to study the Mubayi--Terry problem.

\begin{proposition}[Weighted geometric averaging]
\label{prop: geometring averaging}
Let $U$ be  a set of vertices in an $n$-vertex multigraph $G$ and let $(\alpha_u)_{u\in U}$ be non-negative real numbers with $\sum_{u\in U}\alpha_u=1$. 
Then there is a vertex $v \in U$ such that $p_G(v)\leq \prod_{u\in U} {p_G(u)}^ {\alpha_u}$.
\qed
\end{proposition}

\begin{lemma}[Stability from geometric averaging]
\label{lemma: stability geometric averaging}
Let $p>1$ and $M, t\in \ZZ_{>0}$ be fixed. Then for every $\eps>0$ fixed, there exist $\eta>0$ and $n_0\in \ZZ_{>0}$ such that for all $n\geq n_0$, the following holds: 
    
if $G$ is a multigraph on $n\geq n_0$ vertices with edge multiplicities bounded by $M$ and minimum product-degree at least $p^{(1-\eta)n}$, $U\subseteq V(G)$ is a $t$-set of vertices and  $(\alpha_u)_{u\in U}$ is a collection of non-negative reals with $\sum_{u\in U}\alpha_u=1$ such that for all $v\in V\setminus U$, $\prod_{u\in U} (w(uv))^ {\alpha_u}$ is either equal to $p$ or is at most $p^{1-2\eta}$, then there exists a partition of $V\setminus U$ into $\bigcup_{i=0}^k V_i$ such that

\begin{enumerate}[(i)]
    \item
    for every $i\in [k]$ and every $v,v'\in V_i$ and $u\in U$, $w(uv)=w(uv')$;
    
    \item
    for every $i\in [k]$ and every $v \in V_i$,  $\prod_{u\in U}(w(uv))^ {\alpha_u}= p$;
    
    \item
    $\card{V_0} \leq \eps n$.
\end{enumerate}	

Here $k<M^t$ is the number of distinct ways of assigning strictly positive integer weights to the pairs $(uv)_{u\in U}$ so as to achieve $\prod_{u\in U}(w(uv))^ {\alpha_u}= p$.
\end{lemma}
\begin{proof}
For $i\in[k]$ and each of the $k$ ways of assigning strictly positive integer weights to the pairs $(uv)_{u\in U}$ so as to achieve $\prod_{u\in U}(w(uv))^ {\alpha_u}= p$, set $V_i$ to be the collection of vertices $v\in V\setminus U$ assigning such weights to the pairs $uv$. Let $V_0$ denote the remaining vertices in $V\setminus U$. Then by the minimum product degree assumption, 
\begin{align*}
    p^{(1-\eta)n} \leq \prod_{u\in U} (p_G(u))^ {\alpha_u} \leq M^{t -1} p^{\left(n-t\right) -2\eta \card{V_0}}.
\end{align*}
For $n_0$ chosen sufficiently large and $\eta>0$ chosen sufficiently small, this implies $\card{V_0} \leq \eps n$. The lemma follows.
\end{proof}
	
\subsection{Colourful regularity and symmetrisation}

\begin{definition}
Let $\cP$ be a property of multigraphs. A pattern $P$ is $\cP$-admissible if every blow-up of $P$ belongs to $\cP$, that is, if $\Blow(P) \subseteq \cP$.
\end{definition}

In the special case where $\cP=\cA_{s,q}$, the multigraph property of being an $(s,q)$-graph, we write ``$(s,q)$-admissible'' as a shorthand for $\cA_{s,q}$-admissible.

Given a colouring of the edges of a graph $G$ and  non-empty sets $X,Y\subseteq V(G)$, the density of the colour $c$ between $X$ and $Y$, is
\begin{equation*}
    d_{c}(X,Y) \defined \frac{e_{G_c}(X,Y)}{\card{X}\card{Y}},
\end{equation*}
where by $e_{G_c}(X,Y)$ we denote the number of edges in colour $c$ in $G$ that lie between $X$ and $Y$.

\begin{definition}
\label{def:regularpair}
Given $\eps \in [0,1]$, a colouring of the edges of  a graph $G=(V,E)$  and two subsets of vertices $V_1, V_2 \subseteq V$, we say that the pair $(V_1, V_2)$ is \emph{$\eps$-regular} with respect to colour $c$ if for every $U_1 \subseteq V_1$ and $U_2 \subseteq V_2$ with $\card{U_1}  \geq \eps \card{V_1}$, $\card{U_2} \geq \eps \card{V_2}$, we have $\abs[\big]{d_c(V_1, V_2) - d_c(U_1, U_2)} \leq \eps$.
\end{definition}

\begin{definition}
\label{def:equipartition}
Given a set $V$, we say that a partition $\Pi \defined (V_1, \dotsc, V_k)$ of $V$ into $k$ parts, for some integer $k \geq 1$, is an \emph{equipartition} if $\abs[\big]{\card{V_i}  - \card{V_j}}  \leq 1$ for any $i,j \in [k]$.
\end{definition}

The following is the Colourful Regularity Lemma~\cite{Komlos1996-uq}.

\begin{lemma}[Colourful Regularity Lemma]
\label{lemma:ColRegLem}
Fix an integer $m \geq 2$.
For any $\eps > 0$ and for any fixed positive integer $k_0$, there exist integers $k_1 \defined k_1(\eps, k_0, m)$ and $n_0 \defined n_0(\eps, k_0, m)$ such that for any $n \geq n_0$, if one colours the edges of a graph $G$ on $n$ vertices with $m$ colours, then there exists an equipartition $V_1 \sqcup \dotsb \sqcup V_k$ of the vertex set $V(G)$, for some $k_0 \leq k \leq k_1$, such that all but at most $\eps k^2$ pairs $(V_i, V_j)$ are $\eps$-regular with respect to all $m$ colours.
\end{lemma}

In addition, we shall require a version of the blow-up lemma of Komlós, Sárközy and Szemerédi~\cite{komlos1997blow} due to Komlós~\cite[Lemma 23]{komlos1999blow}. Recall that the \emph{blow-up} $H(t)$ of a graph $H$ is the graph obtained from $H$ by replacing each vertex $v\in V(H)$ by an independent set $I_v$ of size $t$, and each edge $uv\in E(H)$ by a complete bipartite graph between $I_u$ and $I_v$.

\begin{lemma}[Blow-up Lemma]
\label{lemma:Blowuplem}
Fix $\delta, \eps>0$. Let $H$ be a graph of maximum degree $\Delta$. Replace each vertex $v$ of $H$ by an independent set $I_v$ of size $N$, and each edge $uv$ of $H$ by an (arbitrarily chosen) bipartite graph in such a way that the pair $(I_u, I_v)$ is an $\eps$-regular pair of density at least $\delta+\eps$. Denote the resulting graph by $H'$.

Set $\eps_0 = \delta^{\Delta}/(2+\Delta)$, and suppose that $\eps\leq \eps_0$ and $N>t/\eps_0$ both hold. Then $H'$ contains a copy of $H(t)$ as a subgraph.
\end{lemma}

Using \Cref{lemma:ColRegLem,lemma:Blowuplem}, we obtain the following corollary, which is our main tool in applying colourful regularity to our multigraph optimisation problems, by allowing us to reduce them asymptotically to pattern optimisation problems.

\begin{corollary}
\label{cor: applying colourful regularity}
Let $\cP=\Forb(\cF)$ be a bounded property of multigraphs defined by a family of forbidden multigraphs on at most $t$ vertices.
 Then for every $\eps>0$ fixed, there exist positive integers $n_1$ and $k_1$ such that for all $n\geq n_1$ the following holds: if $G \in \cP_n$ has all edges of multiplicity at least $m_-$ and multiplicity at most $m_+$, then there exists a pattern $P$ on $k$ vertices such that
\begin{enumerate}[(i)]
    \item $k  \leq k_1$;
    
    \item all loops in $P$ have multiplicity $m_-$ and all edges in $P$ have multiplicity in $[m_-, m_+]\cap \ZZ$;
    
    \item $P$ is $\cP$-admissible;
    
    \item  if $G'$ is a balanced blow-up of $P$ on $n$ vertices, then
    \begin{equation*}
        e(G')\geq e(G)-\eps \tbinom{n}{2},
        \quad \text{and} \quad
        P(G')\geq P(G)e^{-\eps \binom{n}{2}}.
    \end{equation*}

\end{enumerate}
\end{corollary}
\begin{proof}
Fix $\eps>0$. We may view the multigraph $G\in \cP_n$ as a colouring of the edges of $K_n$ with colours from the set $[m_+, m_-]\cap \mathbb{Z}$. Fix reals $\delta, \eps_1$ and an integer $k_0$ with $0<\delta<\frac{\varepsilon}{8m_+}$, $0< \eps_1 < \delta^{t-1}/(1+t)$ and $k_0 > (\eps_1)^{-1}$. Apply the colourful regularity lemma to $G$ with parameters $m_+-m_-+1$, $\eps_1$ and $k_0$. Let $\sqcup_{i=1}^k V_i$ be the equipartition guaranteed by \Cref{lemma:ColRegLem}.  Now do the following:
\begin{enumerate}[(a)]
    \item for each $i\in [k]$, set $P(i) \defined m_-$ and change the multiplicity of all edges internal to some $V_i$ to $m_-$. For parts $(V_i, V_j)$ which do not form a $\eps_1$-regular pair for some colour $c$, set $P(ij) \defined m_-$ and change the multiplicity of all edges between $V_i$ and $V_j$ to $m_-$;
    
    \item for every pair $(V_i, V_j)$ which is $\eps_1$-regular in every colour, let $P(ij)$ be the largest $c\in [m_-, m_+]\cap \mathbb{Z}$ such that $d_c(V_i, V_j)>2\delta$ (such a $c$ exists, by our choice of $\delta$), and change the multiplicities of all edges in $G$ between $V_i$ and $V_j$ to $P(ij)$.
\end{enumerate}
Let $G' \in \Blow_n(P)$ denote the multigraph thus obtained, where $P$ is the pattern on $[k]$ defined in steps (a)--(b).

Clearly in step (a) we do not create any copy of $F\in \cF$ in $G$, since we have only decreased edge multiplicities. Provided $n$ is chosen sufficiently large so that $\frac{n}{2k_1}\geq \frac{t}{\eps_1}$, it easily follows from the blow-up lemma, \Cref{lemma:Blowuplem}, together with the upper bound of $t$ on the order of multigraphs in $\cF$ that no copy of a member of $\cF$ is created in step (b) either, and in particular that $P$ is a $\cP$-admissible pattern on $k\leq k_1$ vertices.

Finally, we observe that, provided $n$ is sufficiently large  and $\eps_1$ is chosen sufficiently small 
\begin{align*}
e(G')&\geq  e(G)- (m_+-m_-)\left(\eps_1  k^2 {\lceil\frac{n}{k}\rceil}^2  +k\binom{\lceil n/k\rceil}{2} + 2\delta \binom{k}{2}{\lceil\frac{n}{k}\rceil}^2  \right)
&\geq  e(G)-\frac{\eps}{2}\binom{n}{2},
\end{align*}
and more generally for any balanced blow-up $G''$ of $P$ on $n$ vertices, $e(G'')\geq e(G)-\eps \binom{n}{2}$. The bound $P(G'')\geq P(G)\exp \left(-\eps \binom{n}{2}\right)$ follows similarly (observing that both sides are $0$ if $m_-=0$).
The result follows.
\end{proof}

A key tool in pattern optimisation will be the following simple cloning lemma, due to Räty and the first author~\cite{falgasravryraty}.

\begin{lemma}[Cloning lemma]
\label{lemma: cloning}
Let $P$ be a pattern in which every loop has multiplicity $m$. Then there exist subpatterns $P'_+$  (respectively $P'_{\times}$) of $P$ such that for every pair of vertices $ii'$ from $V(P'_{+})$ (respectively $V(P'_{\times})$) we have $P'_{+}(ii')> m$ (respectively $P'_{\times}(ii')>m$) and in addition  $\sigma(P'_+)= \sigma(P)$ (respectively $\pi(P'_{\times})=\pi(P)$).
\end{lemma}
\begin{proof}
Let $[k]$ denote the vertex set of $P$. Suppose $P(ii')\leq m$ for some pair $i,i' \in [k]$. Let $\bfx$ be a sum-optimal weighting of $P$, and let $P - \set{i}$ and $P - \set{i'}$ denote the sub-patterns of $P$ induced by $[k]\setminus \set{i}$ and $[k] \setminus \set{i'}$ respectively.  Assume without loss of generality that the inequality $\sum_{j\in [k] \setminus \set{i,i'}} P(ij)x_j \geq \sum_{j\in [k]\setminus \set{i,i'}} P(i'j)x_j$ holds. Then, we have 
\begin{align*}
    \sigma_P
    &=m (x_i+x_{i'})^2 -2(m-P(ii'))x_ix_{i'}+2x_i \sum_{j \in [k] \setminus \set{i,i'}} P(ij)x_j +2x_{i'} \sum_{j \in [k] \setminus \set{i,i'}} P(i'j)x_j \\
    &\qquad \qquad \qquad \qquad \qquad \qquad \qquad \qquad  +\sum_{j,j' \in [k] \setminus \set{i,i'}} P(jj')x_j x_{j'} \\
&\leq m(x_i+x_{i'})^2 + 2(x_i+x_{i'})\sum_{j\in [k] \setminus \set{i,i'}} P(ij) x_j + \sum_{j, j' \in [k] \setminus \set{i,i'}} P(jj') x_j x_{j'} \leq \sigma_{P - \set{i'}}.
\end{align*}
Passing to the subpattern $P - \set{i'}$ and iterating, the additive part of the lemma follows. The multiplicative part follows similarly.
\end{proof}



\section{The small ambient multiplicity regime}
\label{section: small a}

\subsection{Preliminaries: sum- and product-optimal blow-ups of graph patterns}

\begin{lemma}
\label{lemma: upper bounds for sigma p}
Let $P$ be a pattern such that all loops have multiplicity $c$. Then, we have
\begin{align*}
    \SIGMA{P}{s} \leq \floor[\Big]{\sigma_P \frac{s^2}{2} - \frac{cs}{2}}.
\end{align*}   
\end{lemma}
\begin{proof}
Suppose $P$ is a pattern on $N$ vertices. Consider a multigraph $G \in \Blow_s(P)$, and let $V(G) = \bigsqcup_{i=1}^{N} V_i$ be the $P$-partition of $G$.
Furthermore, for each $i \in [N]$, let $\card{V_i} = x_i s$.
Applying~\eqref{eq: optimisation problem for arithmetic problem}, we have
\begin{align*}
    e(G) &= \sum_{1 \leq i < j \leq N} x_i x_j P(ij) s^2 + c \sum_{i=1}^{N} \binom{x_i s}{2}
    = \frac{\bfx^{t}A_P\bfx s^2}{2} - c \sum_{i=1}^{N}  \frac{x_i s}{2} \\
    &= \frac{\bfx^{t}A_P\bfx s^2}{2} - \frac{cs}{2} \leq \frac{\sigma_P s^2}{2}-\frac{cs}{2}. 
\end{align*}
Thus, we have $\SIGMA{P}{s} \leq \sigma_P s^2/2 - cs/2$, and the lemma follows.
\end{proof}


\begin{corollary}
\label{cor: equality in sigma P when xs is rational}
Let $P$ be a pattern and $\bfx$ be a sum-optimal weighting of $V(P)$. Then for every $s$, we have
\begin{equation*}
    \SIGMA{P}{s} \leq \floor[\bigg]{ \frac{\sigma_P s^2}{2} - \frac{s}{2}\sum_{i\in V(P)} P(\set{i}) x_i }.
\end{equation*}
What is more, if $x_is$ is an integer for every $i \in V(P)$, then we have equality in the above.
\end{corollary}
\begin{proof}
The upper bound can be derived in a similar manner as in the proof of \Cref{lemma: upper bounds for sigma p}, while the lower bound is given by taking $\card{V_i} = x_i s$ for each $i\in V(P)$.    
\end{proof}

\begin{remark}
Recall from \Cref{prop: sigma P is rational} that every pattern $P$ has a rational sum-optimal weighting $\bfx$. In particular, there exists some $N\in \ZZ_{>0}$ such that for all $s\in N\ZZ_{>0}$, $s\bfx$ is an integer-valued vector and by \Cref{cor: equality in sigma P when xs is rational}, we then have $\SIGMA{P}{s} = \floor[\big]{ \frac{\sigma_P s^2}{2} - \frac{s}{2}\sum_{i \in V(P)} P(\set{i})x_i }$.   
\end{remark}

\begin{proof}[Proof of \Cref{prop: sum for optimal blow-ups of stars paths and cycles}]
\noindent\textbf{Part (i):}
A sum-optimal weighting of $K_{1, \ell}^{(a)}$ must assign weight $x$ to the central vertex and weight $(1-x)/\ell$ to every leaf vertex, where $x \in (0,1)$.
By \Cref{cor: balanced degrees in optimal pattern weightings}, we obtain the equalities
\begin{equation}
\label{eq: eqn1}
    \sigma_{K_{1,\ell}^{(a)}} = a+1 - 2x,
\end{equation}
and 
\begin{equation}
\label{eq: eqn2}
    \sigma_{K_{1,\ell}^{(a)}} = \frac{a\ell +(\ell+1)x-1}{\ell}
\end{equation}
by considering the central vertex and a leaf vertex in the star-pattern $K_{1, \ell}^{(a)}$ respectively. Solving the resulting linear equation system, we get $\sigma_{K_{1,\ell}^{(a)}}=a+\frac{\ell-1}{3 \ell +1}$, and the upper bound follows from \Cref{lemma: upper bounds for sigma p}. 
Consider $G \in \Blow_s(K_{1,\ell}^{(a)})$ with $K_{1,\ell}^{(a)}$-partition $V(G)=V_0 \bigsqcup (\sqcup_{i=1}^{\ell}V_i)$, where $V_0$ corresponds to the central vertex, and the $V_i$ correspond to the leaf vertices. Let $s=q(3\ell+1)+r$, where $r \in \set{0} \cup [3\ell]$. It can now be checked that for $\ell \in \set{2,3,4}$, setting $\card{V_0} = q(\ell + 1)+\left \lceil \frac{r}{3}\right \rceil$ and letting the $\card{V_i}$ be as equal as possible yields $e(G)=a \binom{s}{2}+ \left \lfloor \frac{\ell-1}{6\ell+2}s^2+\frac{s}{2} \right \rfloor $.
Thus, for $\ell \in \set{2,3,4}$, we have $\SIGMA{K_{1,\ell}^{(a)}}{s} \geq a\binom{s}{2}+\left \lfloor \frac{\ell-1}{6\ell+2}s^2+\frac{s}{2} \right \rfloor$, and the result follows.

\noindent\textbf{Part (ii):} We must consider $P_4$, $P_5$ and cycles of length at least $6$ separately.

\noindent\textbf{Subcase: $P_4^{(a)}$.} Let $\bfx$ be a sum-optimal weighting of $P_4^{(a)}$. For each $i \in [4]$, let $x_i$ be the weight assigned to the $i$th vertex on the path, following the natural order along the path. Then, either we would have $\sigma_{P_4^{(a)}} = \sigma_{K_{1,2}^{(a)}}= a+\frac{1}{7}$ or $\bfx$ assigns positive weight to each vertex in the pattern $P_4^{(a)}$. In the latter case, by \Cref{cor: balanced degrees in optimal pattern weightings} applied to each of the four vertices in the pattern $P_4^{(a)}$:
\begin{align}
\label{eqn3}
    \sigma_{P_4^{(a)}}& =a+x_2-x_1, \\
\label{eqn4}
    \sigma_{P_4^{(a)}}&=a + x_3-x_4, \\
\label{eqn5}
    \sigma_{P_4^{(a)}}& =a + x_1+x_3-x_2, \\
\label{eqn6}
    \sigma_{P_4^{(a)}}&=a+x_2+x_4-x_3.
\end{align}
By considering the linear combination \eqref{eqn3} + \eqref{eqn4} + $2\cdot$\eqref{eqn5} + $2\cdot$\eqref{eqn6}, we get $6 \sigma_{P_4^{(a)}} = 6a + \sum_{i=1}^4 x_i = 6a+1$.
Thus, we have $\sigma_{P_4^{(a)}} = a+\frac{1}{6} > a+\frac{1}{7}$.
By \Cref{lemma: upper bounds for sigma p}, we have $\SIGMA{P_4^{(a)}}{s}  \leq a\binom{s}{2}+ \left\lfloor \frac{s^2}{12}+\frac{s}{2}\right \rfloor$. 
Consider $G \in \Blow_s(P_4^{(a)})$, and let $V(G)=\bigsqcup_{i=1}^4 V_i$ be the $P_4^{(a)}$-partition of $G$. Here, $V_1$ and $V_2$ correspond to the endpoints of the path, while $V_3$ and $V_4$ correspond to the remaining pair of vertices. Let $s= 6q +r$, where $r \in \set{0} \cup [5]$. It is now easy to check that setting $\card{V_i} = 2q+ \left \lfloor \frac{r+i-1}{4}\right \rfloor$ for $i=3,4$ and $\card{V_i}= q+ \left \lfloor \frac{r+i-1}{4}\right \rfloor$ for $i=1,2$ yields $e(G)=a\binom{s}{2}+  \left \lfloor \frac{s^2}{12}+\frac{s}{2}\right \rfloor$. Thus, we have $\SIGMA{P_4^{(a)}}{s}  \geq a\binom{s}{2}+\left \lfloor \frac{s^2}{12}+\frac{s}{2}\right \rfloor$. Hence, $\SIGMA{P_4^{(a)}}{s}  =a\binom{s}{2}+ \left \lfloor \frac{s^2}{12}+\frac{s}{2}\right \rfloor$.

\noindent\textbf{Subcase: $P_5^{(a)}$.}   Let now $\bfx$ be a sum-optimal weighting of the pattern $P_5^{(a)}$. For each $i \in [5]$, let $x_i$ be the weight assigned to the $i$-th vertex on the path, following the natural order along the path. Then, either we would have $\sigma_{P_5^{(a)}} = \sigma_{P_4^{(a)}}=a+ \frac{1}{6}$ or $\bfx$ assigns positive weight to each vertex in the pattern $P_5^{(a)}$. In the latter case, by \Cref{cor: balanced degrees in optimal pattern weightings}, we have 
\begin{align}
\label{eqn7}
    \sigma_{P_5^{(a)}}&=a+x_2-x_1,\\
\label{eqn8}
    \sigma_{P_5^{(a)}}& =a + x_4-x_5,\\
\label{eqn9}
    \sigma_{P_5^{(a)}}&=a + x_4+x_2-x_3,\\
\label{eqn10}
    \sigma_{P_5^{(a)}}&=a+x_1+x_3-x_2,\\
\label{eqn11}
    \sigma_{P_5^{(a)}}&=a+x_3+x_5-x_4.
\end{align}
Similarly to the way we handled the case of $P_4^{(a)}$, we consider the linear combination \eqref{eqn7} + $2\cdot$\eqref{eqn9} + $2\cdot$\eqref{eqn10} + \eqref{eqn11} and rearrange terms to deduce that $\sigma_{P_5^{(a)}}=a+\frac{1}{6}$. By \Cref{lemma: upper bounds for sigma p}, we have $\SIGMA{P_5^{(a)}}{s}  \leq a\binom{s}{2}+\left \lfloor \frac{s^2}{12}+\frac{s}{2}\right \rfloor$. We also have, $\SIGMA{P_5^{(a)}}{s}  \geq \SIGMA{P_4^{(a)}}{s} =a\binom{s}{2}+\left \lfloor \frac{s^2}{12}+\frac{s}{2}\right \rfloor$, which in turn implies that $\SIGMA{P_5^{(a)}}{s} =a\binom{s}{2}+\left \lfloor \frac{s^2}{12}+\frac{s}{2}\right \rfloor$.

\noindent\textbf{Subcase: $C_\ell^{(a)}$, $\ell\geq 6$.} Now let $\bfx$ be a sum-optimal weighting of $C_\ell^{(a)}$, and let $x_i$ be the weight assigned to the $i$th vertex along the cycle, in a natural cyclic ordering. Suppose $\bfx$ assigns positive weight to every vertex. Then, by \Cref{cor: balanced degrees in optimal pattern weightings}, for each $i \in [\ell]$, we would have $\sigma_{C_\ell^{(a)}}=a+x_{i-1}+x_{i+1}-x_i$, where the indices are taken modulo $\ell$. Then,
\begin{align*}
    \ell \sigma_{C_\ell^{(a)}}=a\ell + \sum_{i=1}^{\ell}(x_{i-1}+x_{i+1}-x_i)=a\ell+\sum_{i=1}^{\ell}x_i=a\ell+1.
\end{align*}
Thus, under our assumption we would have $\sigma_{C_{\ell}^{(a)}}=a+\frac{1}{\ell}$. For $\ell\geq 6$, we note that this is at most $a+\frac{1}{6}=\sigma_{C_6^{(a)}}$.
        
\noindent\textbf{Bringing it together.}
We now conclude the proof of (ii), beginning with cycles, for which we proceed by induction on $\ell_2$. Clearly, we have $\SIGMA{C_{\ell_2}^{(a)}}{s} \geq \SIGMA{P_4^{(a)}}{s} =a\binom{s}{2}+\left \lfloor \frac{s^2}{12}+\frac{s}{2} \right \rfloor$, for all $\ell_2\geq 6$. Therefore, it suffices to prove the upper bound. Suppose $\bfx$ is a sum-optimal weighting of the pattern $C_6^{(a)}$. Then, either $\sigma_{C_6^{(a)}} = \sigma_{P_5^{(a)}}=a+\frac{1}{6}$ or $\bfx$ assigns positive weight to every vertex in the pattern $C_6^{(a)}$. In the latter case, we have $\sigma_{C_6^{(a)}}=a+\frac{1}{6}$. Thus, by \Cref{lemma: upper bounds for sigma p}, we have $\SIGMA{C_6^{(a)}}{s}  \leq a\binom{s}{2} +\left \lfloor \frac{s^2}{12}+\frac{s}{2} \right \rfloor$. We now proceed by induction on $\ell_2$. Let $\bfx$ be a sum-optimal weighting of $C_{\ell_2}^{(a)}$. Then, either $\sigma_{C_{\ell_2}^{(a)}}=\sigma_{P_{\ell_2-1}^{(a)}} \leq \sigma_{C_{\ell_2 - 1}^{(a)}}=a+\frac{1}{6}$, or $\bfx$ assigns positive weight to each vertex in the pattern $C_{\ell_2}^{(a)}$, in which case we have $\sigma_{C_{\ell_2}^{(a)}} = a +\frac{1}{\ell_2} \leq a+ \frac{1}{6}$. The upper bound then follows from \Cref{lemma: upper bounds for sigma p}. For $\ell_1 \geq 6$, we have $a \binom{s}{2}+\left \lfloor \frac{s^2}{12}+\frac{s}{2} \right \rfloor = \SIGMA{P_5^{(a)}}{s}  \leq \SIGMA{P_{\ell_1}^{(a)}}{s} \leq \SIGMA{C_{\ell_1}^{(a)}}{s} = a\binom{s}{2}+\left \lfloor \frac{s^2}{12}+\frac{s}{2} \right \rfloor$. This concludes the proof of part (ii).

\textbf{Part (iii):} Consider $G \in \Blow_s(C_\ell^{(a)})$, with part sizes as equal as possible, and the parts of size $\left \lceil \frac{s}{\ell}\right \rceil$ lying along a path. It is straightforward to check that for $3 \leq \ell \leq 5$, we have $e(G)=a\binom{s}{2}+\left \lfloor \frac{s^2}{2\ell}+\frac{s}{2}\right  \rfloor$. Thus, for $3 \leq \ell \leq 5$, we have $\SIGMA{C_\ell^{(a)}}{s}  \geq a\binom{s}{2}+ \left \lfloor \frac{s^2}{2\ell}+\frac{s}{2}\right  \rfloor $, and it suffices to prove the upper bound. Let $\bfx$ be a sum-optimal weighting of the pattern $C_\ell^{(a)}$. Then, either we would have $\sigma_{C_\ell^{(a)}} =\sigma_{P_{\ell-1}^{(a)}} \leq \sigma_{P_4^{(a)}}=a+\frac{1}{6}$, or $\bfx$ assigns positive weight to every vertex in the pattern $C_\ell^{(a)}$, in which case we have $\sigma_{C_\ell^{(a)}}=a+\frac{1}{\ell} > a+\frac{1}{6}$ for $\ell \leq 5$. By \Cref{lemma: upper bounds for sigma p}, we have $\SIGMA{C_\ell^{(a)}}{s}  \leq a\binom{s}{2}+ \left \lfloor \frac{s^2}{2\ell}+\frac{s}{2}\right  \rfloor $, and the proposition follows.
       
\noindent\textbf{Part (iv):} An optimal blow-up of this pattern $P_{+2}^{(a)}$ is a balanced blow-up (this follows from Mantel's theorem), from which the result is immediate. 
\end{proof}

\begin{lemma}
\label{lemma: subcubic girth at least 5}
Let $G$ be a graph with $\Delta(G) \leq 3$ and girth at least $5$. Then, $\sigma_{G^{(a)}} \leq a + \frac{1}{5}$.
\end{lemma}
\begin{proof}
Let $\bfx$ be a sum-optimal weighting of the pattern $G^{(a)}$, and let $G'^{(a)}$ be the sub-pattern induced by the vertices of $G^{(a)}$ which receive strictly positive weight. Note that $\sigma_{G^{(a)}}=\sigma_{G'^{(a)}}$. 
  
Suppose $\Delta(G') \leq 2$. By \Cref{prop: connected patterns best}, either $\sigma_{G'^{(a)}} \leq a$, or there exists a connected induced subgraph of $G'$, say $G''$, such that $\sigma_{G''^{(a)}}=\sigma_{G'^{(a)}}$. This graph $G''$ must either be a path, or a cycle of length at least $5$, since $G$ has girth at least $5$. By \Cref{prop: sum for optimal blow-ups of stars paths and cycles}, we have $\sigma_{G^{(a)}}=\sigma_{G''^{(a)}} \leq a+\frac{1}{5}$.
  
Now, suppose $\Delta(G')=3$. Then, $G'$ has a vertex $v_1$ of degree 3, with neighbours $v_2, v_3$ and $v_4$. For $i \in [4]$, denote by $x_i$ the weight assigned to $v_i$. For $i \in \set{2,3,4}$, we write $N(v_i)$ for the set of neighbours of $v_i$, excluding $v_1$, and $X_i$ for the sum total of weights over all vertices in  $N(v_i)$. Note that the sets $\set{x_1, x_2, x_3, x_4}$, $N(v_2)$, $N(v_3)$ and $N(v_4)$ are pairwise disjoint since $G$ has girth at least 5. By \Cref{cor: balanced degrees in optimal pattern weightings}, we have 
\begin{align}
\label{equation14}
    \sigma_{G^{(a)}} &= a + x_2 + x_3 + x_4 - x_1,  \text{ and} \\
\label{equation15}
    \sigma_{G^{(a)}} &= a + x_1 + X_i - x_i, \quad \text{for all $i \in \set{2,3,4}$.}
\end{align}
Therefore,
\begin{align*}
    5\sigma_{G^{(a)}}=2(a+x_2+x_3+x_4-x_1)+\sum_{i=2}^4 (a+x_1+X_i-x_i) = 5a+\sum_{i=1}^4 x_i+\sum_{i=2}^4 X_i \leq 5a+1.
\end{align*}
Thus, in this case as well, we have $\sigma_{G^{(a)}} \leq a+\frac{1}{5}$. This proves the proposition.   
\end{proof}

\begin{proof}[Proof of \Cref{prop: sum-optimal blow-ups for C5}]
Every graph $G \in \cC_5$ has maximum degree at most $3$ and girth at least 5. By \Cref{lemma: subcubic girth at least 5} and \Cref{lemma: upper bounds for sigma p}, we have that for every $G \in \cC_5$, $\SIGMA{G^{(a)}}{s}  \leq a\binom{s}{2}+ \left \lfloor \frac{s^2}{10}+\frac{s}{2}\right \rfloor$. It therefore suffices to prove the lower bound. 

Since every graph $G \in \cC_5 \setminus C_5$ contains a degree 3 vertex, we have $\SIGMA{G^{(a)}}{s}  \geq \SIGMA{K_{1,3}^{(a)}}{s} =a\binom{s}{2}+\left \lfloor \frac{s^2}{10}+\frac{s}{2} \right \rfloor$. The last equality follows from part (i) of \Cref{prop: sum for optimal blow-ups of stars paths and cycles}. We have already shown in part (iii) of \Cref{prop: sum for optimal blow-ups of stars paths and cycles} that $\SIGMA{C_5^{(a)}}{s} =a \binom{s}{2}+\left \lfloor \frac{s^2}{10}+ \frac{s}{2} \right \rfloor$. The proposition follows.
\end{proof}




\begin{proposition}
\label{prop: connected patterns best}
Let $G$ be a graph which is not connected. Then the following hold:
\begin{enumerate}[(i)]
    \item if $\sigma_{G^{(a)}}>a$, then there exists a connected  induced subgraph $G_+$ of $G$ such that $	\sigma_{{G_+^{(a)}}}= \sigma_{G^{(a)}}$.
    
    \item if $\pi_{G^{(a)}}>\log a$, then there exists a connected  induced subgraph $G_{\times}$ of $G$ such that $	\pi_{G_{\times}^{(a)}}= \pi_{G^{(a)}}$.
\end{enumerate}
\end{proposition}
\begin{proof}
Let $C_1, \dotsc, C_k$ denote the connected components of $G$. 
Consider a sum-optimal weighting $\bfx$ of the pattern $G^{(a)}$.  Let $I = \set{i_1, i_2, \dotsc, i_m} \subseteq [k]$ denote those indices $i$ for which $\exists v\in C_i$ with $x_{v}>0$. Then by the sum-optimality of $\bfx$, we have that the rescaled vectors $(x_v/\sum_{u\in C_i}x_u)_{v\in C_i}$ are sum-optimal for $\sigma_{G[C_i]^{(a)}}$ for every $i\in I$. In particular, we have
\begin{align*}
    a <	\sigma_{G^{(a)}}
    = \sum_{i \in I} \p[\bigg]{\sum_{v \in C_i} x_v }^2 \sigma_{G[C_i]^{(a)}} + \sum_{i\in I}\sum_{i' \in I\setminus\{i\}} a \p[\bigg]{\sum_{v \in C_i} x_v } \p[\bigg]{ \sum_{v \in C_{i'}} x_v }.	
\end{align*}
It follows that $\max_{i\in I} \sigma_{G[C_i]^{(a)}}>a$ and hence, by a standard optimisation argument, that $\sigma_{G^{(a)}}=\max_{i\in I} \sigma_{G[C_i]^{(a)}}$, as desired. This concludes the proof of part (i). Part (ii) is proved mutatis mutandis.	
\end{proof}

\begin{proposition}
\label{prop: bounded degree patterns bounded number of vertices}
Let $G$ be an $N$-vertex graph of maximum degree $\Delta$.
Then the following hold:
\begin{enumerate}[(i)]
    \item if $\sigma_G \defined \sigma_{G^{(a)}} > a$, then $N \leq (\Delta - 1)/(\sigma_G - a)$;
    
    \item if $\pi_G \defined \pi_{G^{(a)}} > \log a$, then $N \leq  \p[\big]{ \log\left((a+1)^{\Delta}(a-1)/a^{\Delta+1} \right)}/(\pi_G - \log a)$.

\end{enumerate}
\end{proposition}
\begin{proof}
For the first part, applying \Cref{cor: balanced degrees in optimal pattern weightings} to a sum-optimal weighting $\bfx$ of $G^{(a)}$, we get
\begin{align*}
    N \sigma_G - (a - 1)
    = \sum_{v \in V(G)} \p[\big]{ \sigma_G - (a-1)x_v}
    &= \sum_{v \in V(G)} \sum_{u \in V(G)\setminus\set{v}} (a + \ind_{u v \in E(G)}) x_u\\
    &\leq \sum_{v \in V(G)} \p[\big]{a(N - 1) + \Delta} x_v = a N + \Delta - a.
\end{align*}
Rearranging terms, (i) follows. Similarly for part (ii), we have
\begin{align*}
    N \pi_G-\log(a-1)
    = \sum_{v \in V(G)} \pi_G-x_{v}\log(a-1)
    &= \sum_{v \in V(G)} \sum_{u \in V(G) \setminus\set{v}} \p[\Big]{\log a + \ind_{uv \in E(G)}\log\p[\Big]{\frac{a+1}{a}} } x_u\\
    &\leq  \sum_{v\in V(G)} \p[\Big]{(N - 1)\log a + \Delta \log\p[\Big]{\frac{a+1}{a}} } x_v \\
    &= (N - 1)\log a + \Delta \log\p[\Big]{\frac{a+1}{a}},	
\end{align*}	
and again we obtain the desired conclusion by rearranging terms.
\end{proof}

Finally, we shall need some bounds on the value of $\pi_{G^{(a)}}$ for certain graph patterns $G=G^{(a)}$ in the specific case $a=2$.

\begin{proposition}
\label{prop: product extremal for cycles paths stars}
For ambient edge multiplicity $a=2$, we have the following bounds on $\pi_G$ for the graph patterns $G=G^{(2)}$:
\begin{enumerate}[(i)]
    \item for all $\ell \geq 2$, $\pi_{K_{1,\ell}} = \ell (\log 3)^2/\p[\big]{2\ell \log 3 - (\ell - 1)\log 2}$;
    
    \item $\pi_{P_4} = \p[\big]{(\log 3)^2 + (\log 2)^2 + \log 2 \log 3}/(2 \log 6)$;
    
    \item $\pi_{P_5} = \p[\big]{2 \log 2 (\log 6 \log(9/2) - \log 3 \log(3/2)}/\p[\big]{5 \log 2 \log(9/2) - \log 6 \log(3/2)}$;
    
    \item $\pi_{C_6} = (\log 3)/3 + (\log 2)/2$;
    
    \item for all $\ell \geq 3$, $\pi_{C_\ell}\geq \pi_{C_{\ell+1}}$.
\end{enumerate}
\end{proposition}
\begin{proof}
Applying \Cref{cor: balanced degrees in optimal pattern weightings}, parts (i)--(iv) are straightforward exercises in optimisation.
See the Appendix for details of the proof.

For part (v), by considering a uniform weighting of the pattern $C_\ell^{(2)}$, we note that $\pi_{C_\ell} \geq \p[\big]{(\ell - 3)\log 2 + 2\log 3}/\ell$. Let $\bfx$ be a product-optimal weighting of $C_{\ell+1}$. For $i \in [\ell+1]$, let $x_i$ denote the weight assigned to the $i$th vertex on the cycle, following a natural cyclic ordering. Then, either $\pi_{C_{\ell+1}}=\pi_{P_\ell}\leq \pi_{C_\ell}$ or $\bfx$ assigns strictly positive weight to every vertex of the $C_{\ell+1}$.  In the latter case, by \Cref{cor: balanced degrees in optimal pattern weightings}, for all $i \in [\ell+1]$, we have 
\begin{equation*}
    \pi_{C_{\ell+1}} = \log 2 + (x_{i-1} + x_{i+1}) \log(3/2) -x_i \log 2,
\end{equation*}
where the indices are taken modulo $\ell+1$. Therefore, we have 
\begin{align*}
    (\ell + 1)\pi_{C_{\ell+1}} &=(\ell + 1)\log 2 +\sum_{i=1}^{\ell+1} \p[\Big]{ (x_{i-1} + x_{i+1})\log(3/2) -x_i \log 2}\\
    &= (\ell + 1) \log 2 + 2 \log(3/2) - \log 2 = (\ell - 2) \log 2 + 2\log 3 .
\end{align*}
Finally, this gives $\pi_{C_{\ell + 1}} = \p[\big]{(\ell - 2)\log 2 + 2 \log 3}/(\ell + 1) < \p[\big]{(\ell-3) \log 2 + 2 \log 3}/\ell  \leq \pi_{C_\ell}$.
\end{proof}

\subsection{Additive case}

We first give a proof of the trivial $a=0$ case.

\begin{proof}[Proof of \Cref{prop: a=0 case additive}]
Observe that the pattern $P_{+2}=P_{+2}^{(0)}$ consisting of an edge of multiplicity $2$ and two loops of multiplicity $0$ satisfies $\sum_{P_{+2}}(s)> \binom{s}{2}$. It then follows from the Erdős--Stone theorem that for $q<\binom{s}{2}$, any $(s,q)$-graph on $n$ vertices contains $o(n^2)$ edges of multiplicity at least $2$, each of which has multiplicity at most $q$. Replacing such edges by edges of multiplicity $0$, we obtain a multigraph in which every edge has multiplicity $0$ or $1$, i.e.\ an ordinary graph. Thus, our original multigraph could have had at most $\ex(n,s,q)+o(n^2)$ edges.
\end{proof}

We now turn our attention to the case $a=1$, the first non-trivial ambient edge multiplicity for the Füredi--Kündgen problem.

\begin{proof}[Proof of \Cref{theorem: additive}]
We work with ambient edge multiplicity $a=1$, and view the graph $G$ as a graph pattern $G^{(1)}$ when convenient. \Cref{prop: sum for optimal blow-ups of stars paths and cycles} and simple arithmetic gives us the following inequalities:
\begin{align}
\label{ineq: sigma G}
    \SIGMA{K_{1,2}}{s}  \leq \SIGMA{P_4}{s} =\SIGMA{C_6}{s}  \leq \SIGMA{K_{1,3}}{s} =\SIGMA{C_5}{s}  \leq \SIGMA{K_{1,4}}{s} \leq \SIGMA{C_4}{s} \leq \SIGMA{{P_{+2}}^{(1)}}{s},
\end{align}
where $P_{+2}$ is the pattern corresponding to an edge of multiplicity $3$ with multiplicity $0$ loops around its endpoints. In particular, for $q< \SIGMA{P_{+2}}{s} $, $P$ is not $(s,q)$-admissible. Further, we have $\SIGMA{C_6}{s} =\SIGMA{C_\ell}{s} $ for all $\ell \geq 6$.


Let $s\geq 2$ and $q< \SIGMA{{P_{+2}}^{(1)}}{s} = \binom{s}{2}+\left\lfloor  \frac{s^2+2s}{4}\right\rfloor$ be fixed integers. Pick $\eps>0$ and apply \Cref{cor: applying colourful regularity} to a sum-extremal $G\in \cA_{s,q}$ on $n$ vertices, for $n$ sufficiently large. Then there exists an $(s,q)$-admissible pattern $P$ on at most a fixed number $k_1=k_1(\eps)$ of vertices in which all loops have multiplicity $0$ and all edges have  multiplicity at most $2$, such that an $n$-vertex balanced blow-up $G'$ of $P$ satisfies $e(G') \geq e(G)-\eps\binom{n}{2}$.

By the cloning lemma, \Cref{lemma: cloning}, it follows that for $n$ sufficiently large, there exists a subpattern $P'$ of $P$ in which every edge has multiplicity $1$ or $2$ and a blow-up $G''$ of $P'$ such that $e(G'')\geq 
\sigma_P \binom{n}{2} -\eps n^2\geq e(G')-2\eps \binom{n}{2}\geq e(G)-3 \eps \binom{n}{2}$.  Given our edge conditions, this subpattern $P'$ is in fact a graph pattern $P'=H^{(1)}$, for some auxiliary graph $H$ on $N$ vertices. We now analyse $H$'s structure and the bounds this gives on the size of blow-ups of $P'$ (and hence on $e(G)$).

\textbf{Case (i): $\binom{s}{2}\leq q< \SIGMA{K_{1,2}}{s} =\binom{s}{2}+\left\lfloor \frac{s^2}{14}+\frac{s}{2}\right\rfloor$}. In this range, $H$ has maximum degree $1$, and in particular is a subgraph of the graph $H'$ consisting of a perfect matching on $2\lceil N/2\rceil$ vertices. An easy optimisation shows that $\sigma_P\leq \sigma_{H'^{(1)}}=1$, whence we deduce that $\ex_\Sigma(s,q)\leq 1$ as well. This is tight, with a matching lower bound provided by the multigraph in which all edges have multiplicity $1$.

\textbf{Case (ii): $\SIGMA{K_{1,2}}{s} \leq q < \SIGMA{P_4}{s}  = \binom{s}{2}+\left \lfloor \frac{s^2}{12}+\frac{s}{2}\right\rfloor$}.  The lower bound  $\ex_\Sigma(s,q)\geq \sigma_{K_{1,2}}$ is trivial. For the upper bound, note that by~\eqref{ineq: sigma G} we have $q<\SIGMA{P_4}{s}  \leq \SIGMA{K_{1,3}}{s} $ and $q<\SIGMA{P_4}{s}  \leq C_\ell(s)$ for all $\ell \geq 3$. It follows in particular that $H$ is a graph of maximum degree two whose components consist of paths on at most three vertices. By \Cref{prop: connected patterns best}, either $\sigma_H\leq 1$ or there exists a connected subpattern $H'$ of $H$ with $\sigma_{H'}=\sigma_H$. Such a connected subpattern must be a path on at most $3$ vertices, whence $\sigma_P =\sigma_{H'} \leq \sigma_{P_3}=1+\frac{1}{7}$, with the last equality following from \Cref{prop: sum for optimal blow-ups of stars paths and cycles} part (i) applied to $P_3=K_{1,2}$ and $s$ large.

\textbf{Case (iii): $\SIGMA{P_4}{s} \leq q < \SIGMA{C_5}{s}  = \binom{s}{2}+\left \lfloor \frac{s^2}{10}+\frac{s}{2}\right\rfloor$}.  Then we have the lower bound  $\ex_\Sigma(s,q)\geq \sigma_{P_4}=1+\frac{1}{6}$, with the value of $\sigma_{P_4}$ following from \Cref{prop: sum for optimal blow-ups of stars paths and cycles} part (ii).

For the upper bound, by~\eqref{ineq: sigma G}, we have $\SIGMA{C_3}{s} \geq \SIGMA{C_4}{s} \geq  \SIGMA{K_{1,3}}{s} =\SIGMA{C_5}{s} >q$.  In particular, $H$ is a graph of maximum degree two and girth at least $6$. By \Cref{prop: connected patterns best}, either $\sigma_H\leq 1$ or there exists a connected subpattern $H'$ of $H$ with $\sigma_{H'}=\sigma_H$. Such a connected subpattern must be either a path, or a cycle on at least $6$ vertices. By \Cref{prop: sum for optimal blow-ups of stars paths and cycles} part (ii), it follows that $\sigma_{H'}\leq 1 +\frac{1}{6}$.

\textbf{Case (iv): $\SIGMA{C_5}{s} \leq q < \SIGMA{K_{1,4}}{s}  = \binom{s}{2}+\left \lfloor \frac{3s^2}{26}+\frac{s}{2}\right\rfloor$}.  Then we have the lower bound  $\ex_\Sigma(s,q)\geq \sigma_{C_5}=1+\frac{1}{5}$, with the value of $\sigma_{C_5}$ following from \Cref{prop: sum for optimal blow-ups of stars paths and cycles} part (iii).

For the upper bound, by~\eqref{ineq: sigma G} we have $\SIGMA{C_3}{s} \geq \SIGMA{C_4}{s} \geq  \SIGMA{K_{1,4}}{s} >q$.  In particular, $H$ is a graph of maximum degree three and girth at least $5$, whence by \Cref{lemma: subcubic girth at least 5} we have $\sigma_H \leq 1+\frac{1}{5}$.

\end{proof}


\begin{remark}
\label{remark: computational reduction in additive case}
Let $K_{1, \infty}$ denote the graph pattern consisting of an edge of multiplicity $2$ with a loop of multiplicity $1$ at one end and a loop of multiplicity $0$ at the other end. It is not hard to show $\SIGMA{K_{1,\infty}}{s} =\binom{s}{2} +\lfloor \frac{s^2}{6}+\frac{s}{6}\rfloor
\leq \binom{s}{2} +2\lfloor \frac{s^2}{4}\rfloor =\SIGMA{P_{+2}}{s}$.

Then our proof method in \Cref{theorem: additive} reduces the determination of $\ex_\Sigma(s,q)$ for all $q$: $\binom{s}{2} \leq q < \SIGMA{K_{1, \infty}}{s}$ to the resolution of a finite computational optimisation problem. For such $q$, we have that $(s,q)$-admissible graph patterns $G$ have degree bounded by some $\Delta=\Delta(q)$. Suppose $q\geq \SIGMA{K_{1,4}}{s} $ (the first value not covered by \Cref{theorem: additive}). Then, by \Cref{prop: bounded degree patterns bounded number of vertices} together with our arguments from the beginning of the proof of \Cref{theorem: additive}, it follows that there exists an $(s,q)$-admissible pattern $P$  on $N\leq (\Delta -1)/(3/13)$ vertices such that $\ex_\Sigma(s,q)=\sigma_P$. In addition, we have that all loops in $P$ have multiplicity $0$ and all edges in $P$ have multiplicity $1$ or $2$ (the latter by \Cref{prop: connected patterns best} and since $q<\SIGMA{P_{+2}}{s} $).

In particular, it suffices to compute $\sigma_P$ for all such patterns $P$ on at most $N$ vertices to determine $\ex_\Sigma(s,q)$; the Füredi--Kündgen problem in this range is thus reduced to a finite computational problem (albeit one that is unfeasible in practice). What is more, one can determine the approximate structure of extremal examples via sum-optimal weightings $\bfx$ of extremal patterns $P$, i.e.\ by determining the space of solutions to the optimisation problem~\eqref{eq: optimisation problem for arithmetic problem} for the set of $(s,q)$-admissible patterns $P$ with $\sigma_P=\ex_\Pi(s,q)$.
\end{remark}




\subsection{Multiplicative case}

We first deal with the trivial $a=1$ case.

\begin{proof}[Proof of \Cref{prop: a=1 case multiplicative}]
We may restrict our attention to multigraphs in which every edge has multiplicity at least $1$. As in the proof of \Cref{prop: a=0 case additive}, for $q<2\binom{s}{2}$, any $(s,q)$-graph $G$ on $n$ vertices can contain at most $o(n^2)$ edges of multiplicity $3$ or more, and no edge of multiplicity greater than $q$. Replacing such edges by edges of multiplicity $1$, we obtain a multigraph in which every edge has multiplicity $1$ or $2$, which we can view as an ordinary graph by subtracting $1$ from each edge multiplicity. Thus, our original $(s,q)$-graph $G$ could have had at most $\binom{n}{2}+\ex(n,s,q-\binom{s}{2})+o(n^2)$ edges in total, and hence, by the integer AM-GM inequality, must have satisfied $P(G)\leq 2^ {\ex(n,s,q-\binom{s}{2}) +o(n^2)}$.
\end{proof}

We now turn to the $a=2$ case, the first non-trivial ambient edge multiplicity for the Mubayi--Terry problem.

\begin{proof}[Proof of \Cref{theorem: multiplicative}]
In all that follows, the ambient edge multiplicity is $a=2$, and the graphs $G$ are identified with the graph patterns $G^{(2)}$ when convenient. We begin as in the proof of \Cref{theorem: additive}, by observing that for all $s\geq 2$ and all $q<\SIGMA{P_{+2}}{s} $, the pattern $P_{+2}=P_{+2}^{(2)}$ consisting of an edge of multiplicity $4$ with multiplicity $1$ loops around its endpoints fails to be $(s,q)$-admissible. In addition, extremal $(s,q)$-graphs only have edges of multiplicity $1$ or higher, so we may restrict our attention to multigraphs with that property.

Let $s\geq 2$ and $q< \SIGMA{P_{+2}}{s} =2\binom{s}{2}+\lfloor \frac{s^2}{4}+\frac{s}{2}\rfloor$ be fixed integers. Pick $\eps>0$ and apply \Cref{cor: applying colourful regularity} to a product-extremal $G\in \cA_{s,q}$ on $n$ vertices, for $n$ sufficiently large. Then there exists an $(s,q)$-admissible pattern $P$ on at most a fixed number $k_1=k_1(\eps)$ vertices in which all loops have multiplicity $1$ and all edges have multiplicity at most $3$,  such that an $n$-vertex balanced blow-up $G'$ of $P$ satisfies $P(G') \geq P(G)e^{-\eps\binom{n}{2}}$.

By the cloning lemma, \Cref{lemma: cloning}, it follows that for $n$ sufficiently large there exists a subpattern $P'$ of $P$ in which every edge has multiplicity $2$ or $3$ and a blow-up $G''$ of $P'$ such that $P(G'')\geq 
\exp \left(\pi_P\binom{n}{2} -\eps n^2\right) \geq P(G')\exp(-2\eps \binom{n}{2})\geq P(G)\exp(-3 \eps \binom{n}{2})$.  Given our edge conditions, this subpattern $P'$ is in fact a graph pattern $P'=H^{(2)}$, for some auxiliary graph $H$ on $N$ vertices. We now analyse $H$'s structure to bound $\pi_H=\pi_{P'}=\pi_P$,  and hence $P(G)$ (since $P(G)^{1/\binom{n}{2}} \leq e^{\pi_H +3\eps}$ and $\eps>0$ was chosen arbitrarily).

\textbf{Case (i): $2\binom{s}{2}\leq q < \SIGMA{P_4}{s}  =2\binom{s}{2} +\left \lfloor \frac{s^2}{12}+\frac{s}{2}\right\rfloor$}.  The lower bound  $\ex_\Pi(s,q)\geq 2$ follows by considering multigraphs in which every edge has multiplicity exactly $2$.

For the upper bound, note that by \Cref{prop: sum for optimal blow-ups of stars paths and cycles}, we have $\SIGMA{K_{1,3}}{s}  \geq\SIGMA{P_4}{s} $ and $\SIGMA{C_{\ell}}{s} \geq \SIGMA{P_4}{s}$ for all $\ell \geq 3$.  In particular, $H$ is a graph of maximum degree at most two whose components consist of paths on at most three vertices.

By \Cref{prop: connected patterns best}, either $\pi_H\leq \log(2)$ or there exists a connected subpattern $H'$ of $H$ with $\pi_{H'}=\pi_H$. Such a connected subpattern must be a path on at most $3$ vertices, whence $\pi_{H'} \leq \pi_{P_3}=\pi_{K_{1,2}}= 2(\log (3))^2 /\log(81/2)<\log(2)$, with the value of $\pi_{K_{1,2}}$ coming from \Cref{prop: product extremal for cycles paths stars} (i) and the inequality can be verified using e.g.\ Wolfram Alpha. Thus, we must have $\pi_H\leq \log (2)$, as required.

\textbf{Case (ii): $\SIGMA{P_4}{s} \leq q < \SIGMA{K_{1,3}}{s} =2\binom{s}{2} +\left \lfloor \frac{s^2}{10}+\frac{s}{2}\right\rfloor$}.
Then we have the lower bound  $\ex_\Pi(s,q) \geq \exp(\pi_{C_6}) = 2^{1/2} 3^{1/3}$ by considering balanced blow-ups of the $C_6$ pattern (which is $(s,q)$-admissible since $\SIGMA{C_6}{s} =\SIGMA{P_4}{s} $ by \Cref{prop: sum for optimal blow-ups of stars paths and cycles} (ii), and for which the uniform weighting is product optimal by \Cref{prop: product extremal for cycles paths stars} (iv) and a simple computation).

For the upper bound, by  \Cref{prop: sum for optimal blow-ups of stars paths and cycles}, we have  $\SIGMA{C_3}{s} \geq \SIGMA{C_4}{s} \geq  \SIGMA{C_5}{s} =\SIGMA{K_{1,3}}{s} $.
In particular, $H$ is a graph of maximum degree two and girth at least $6$.
By \Cref{prop: connected patterns best}, either $\pi_H \leq \log 2$ or there exists a connected subpattern $H'$ of $H$ with $\pi_{H'} = \pi_H$. Such a connected subpattern must be either a path, or a cycle on at least $6$ vertices.
By \Cref{prop: product extremal for cycles paths stars} (v), it follows that $\pi_{H'}\leq \pi_{C_6}$ (since $\pi_{P_\ell}\leq \pi_{C_\ell} \leq \pi_{C_6}$ for all $\ell \geq 6$, and since for all $\ell \leq 6$ we have $\pi_{P_\ell} \leq \pi_{P_6} \leq \pi_{C_6}$ by the monotonicity of $\pi_P$ under taking subpatterns).

\textbf{Case (iii): $\SIGMA{K_{1,3}}{s} \leq q < \SIGMA{K_{1,4}}{s}  = 2\binom{s}{2}+\left \lfloor \frac{3s^2}{26}+\frac{s}{2}\right\rfloor$}.  Then we have the lower bound  $\ex_\Pi(s,q)\geq \exp\left(\pi_{\Petersen}\right)\geq \sqrt[5]{2^3} \sqrt[10]{3^3} $ by considering balanced blow-ups of the $\Petersen$ pattern (which is $(s,q)$-admissible since $\SIGMA{\Petersen}{s} = \SIGMA{K_{1,3}}{s}$ by \Cref{prop: sum-optimal blow-ups for C5}).

For the upper bound, note that by \Cref{prop: sum for optimal blow-ups of stars paths and cycles}, we have $\SIGMA{C_3}{s} \geq \SIGMA{C_4}{s} \geq  \SIGMA{K_{1,4}}{s} $.  In particular, $H$ is a graph of maximum degree three and girth at least $5$. By \Cref{prop: connected patterns best}, either $\pi_H\leq \log(2)$ or there exists a connected subpattern $H'$ of $H$ on $N'$ vertices with $\pi_{H'}=\pi_H$. Further, by \Cref{prop: bounded degree patterns bounded number of vertices} part (ii) (applied with $a=2$ and $\Delta=3$), we have that either $\pi_{H'}\leq \log (2)  +\frac{\log(27/16)}{10}$, and we are done,  or $N'\leq 10$.

Suppose that the latter occurs. Let $\bfx$ be a product-optimal weighting of the vertices of the pattern $H'$. Given a vertex $u\in V(H')$ with strictly positive weight $x_u>0$, we denote by $x_{N(u)} \defined \sum_{v \st uv \in E(H)} x_v$ the sum of the weights of the neighbours of $u$ in $H'$.

\begin{claim}
\label{claim: weight of nhood is large petersen case}
For all vertices $u$ with $x_u = 1/10 + \alpha_u > 0$, we have
\begin{align*}
    x_{N(u)} \geq \frac{3}{10} + \alpha_u \, \frac{\log 2}{\log 3 -\log 2}.
\end{align*}	
\end{claim}
\begin{proof}
By \Cref{cor: balanced degrees in optimal pattern weightings} applied at the vertex $u$, we have 
\begin{align*}
    \log 2 + \frac{\log(27/16)}{10}
    &< \pi_{H'}
    = \sum_{v \neq u} \p[\Big]{\log 2 + \ind_{uv \in E(H')} \log (3/2) } x_v \\
    &= (9/10 - \alpha_u) \log 2 + x_{N(u)} \log(3/2).
\end{align*}
Rearranging terms, we obtain the desired inequality.
\end{proof}

Now consider a vertex $u$ in $H'$ receiving maximal weight $x_u$. By averaging, $x_u\geq \frac{1}{N'}\geq \frac{1}{10}$. Set $\alpha_u \defined x_u-\frac{1}{10}\geq 0$.

\begin{claim}
\label{claim: alpha u is zero then done Petersen proof}
If $\alpha_u=0$, then $\pi_{H'}\leq \log (2)  +\frac{\log(27/16)}{10}$.
\end{claim}
\begin{proof}
Indeed, if $\alpha_u=0$, then in fact averaging implies we must have $N'=10$ and $x_v=\frac{1}{10}$ for every $v\in V(H')$. Recall that $H'$ has maximum degree at most $3$ and girth at least $5$. By \Cref{claim: weight of nhood is large petersen case} we have $x_{N(v)}=\frac{3}{10}$ for every vertex $v\in V(H')$ which implies $H'$ is $3$-regular. Since there is a unique (up to isomorphism) cubic graph of girth at least $5$ on $10$ vertices, namely the Petersen graph, and since our putative product-optimal weighting $\bfx$ is uniform, it follows  that $\pi_{H'} \leq \log (2)  +\frac{\log(27/16)}{10}$.
\end{proof}

In view of \Cref{claim: alpha u is zero then done Petersen proof}, we may thus assume $\alpha_u>0$. We claim that $u$ has at least $3$ neighbours in $H'$. Indeed, otherwise by \Cref{claim: weight of nhood is large petersen case} and the maximality of $x_u$, we would have
\begin{align*}
    2x_u =\frac{2}{10}+2\alpha_u\geq  x_{N(v)}\geq \frac{3}{10} +\alpha_u \frac{\log 2}{\log3 - \log 2}.
\end{align*}
Rearranging terms yields the lower bound $\alpha_u \geq \frac{\log(3/2)}{10 \log(9/8)}>\frac{1}{3}$. On the other hand, applying \Cref{claim: weight of nhood is large petersen case} again, we have
\begin{align*}
    1\geq x_u +  x_{N(v)}\geq \frac{4}{10} +\alpha_u \frac{\log 3}{\log3 - \log 2},
\end{align*}
which after rearranging terms implies the upper bound $\alpha_u \leq\frac{3(\log 3 -\log 2 )}{5 \log 3}<\frac{1}{3}$, a contradiction. Let therefore $v_i$, $i=1,2,3$, denote the neighbours of $u$ in $H'$, and set $\alpha_{v_i} \defined x_{v_i}-\frac{1}{10}$. Further, set $Z_i$ denote the set of neighbours of $v_i$ in $V(H')\setminus \set{u}$, and let $x_{Z_i} \defined \sum_{w\in Z_i} x_w$.

Since $H'$ has girth at least $5$, the sets $Z_1$, $Z_2$, $Z_3$ and $\set{u, v_1, v_2, v_3} = \set{u} \cup N(u)$ are pairwise disjoint. Together with \Cref{claim: weight of nhood is large petersen case}, this disjointness implies 
\begin{align}
\label{eq: upper bound on second hood weight petersen case}
    \sum_i x_{Z_i}\leq 1- x_u - x_{N(u)} \leq \frac{6}{10} - \alpha_u \frac{\log 3}{\log 3 - \log 2}.
\end{align}
On the other hand, by \Cref{claim: weight of nhood is large petersen case} applied to each of the $v_i$ and to $u$, we have that 
\begin{align}
\label{eq: lower bound on second hood weight petersen case}
    \sum_i x_{Z_i} \geq \sum_i \left(\frac{3}{10}  +\alpha_{v_i}\frac{\log 2}{\log 3-\log 2} -x_u\right)&=\frac{6}{10}  -3\alpha_u +\left(x_{N(u)} -\frac{3}{10}\right)\frac{\log 2}{\log 3-\log 2}   \notag \\
    &\geq \frac{6}{10}  -3\alpha_u +\alpha_u \left(\frac{\log 2}{\log 3 -\log 2}\right)^2. 
\end{align}
Combining the lower bound~\eqref{eq: lower bound on second hood weight petersen case} with the upper bound~\eqref{eq: upper bound on second hood weight petersen case}, we get 
\begin{align*}
    0\geq &  \frac{\alpha_u }{(\log 3-\log 2)^2}\left(  \log (2)^2 +\log(3)^2 - \log (2)\log(3) -3\log(3)^2 +6\log(3)
    \log(2)-3\log(2)^2  \right)\\
    &=\frac{\alpha_u}{(\log 3-\log 2)^2}\left( \log(3)\log(2) - 2(\log (3/2))^2\right)>\frac{0.4326 \alpha_u}{(\log 3-\log 2)^2}>0,
\end{align*}
a contradiction. Thus, $\alpha_u=0$, and we have $\pi_{H'}\leq \log (2)  +\frac{\log(27/16)}{10}$, as desired, and our argument implies that this is precisely the value of $\pi_{\Petersen}$, and that it is uniquely attained by taking the uniform weighting.
\end{proof}

\begin{remark}
\label{remark: extremal examples in small a}
Here again, just as we noted in \Cref{remark: computational reduction in additive case}, letting $K_{1, \infty}$ denote the graph pattern consisting of an edge of multiplicity $3$ with a loop of multiplicity $2$ at one end and a loop of multiplicity $1$ at the other end, it is not hard to show that our methods reduce the problem of determining $\ex_\Pi(s,q)$ for $q$ in the range $2\binom{s}{2}\leq q < \SIGMA{K_{1,\infty}}{s} =2\binom{s}{2}+ \lfloor \frac{s^2}{6}+\frac{s}{6}\rfloor$ to a finite computational problem. What is more, one can read out the approximate structure of the extremal examples out of solutions to the optimisation problem~\eqref{eq: optimisation problem for geometric problem}. In particular, in this range of $q$, our \Cref{conj: geometric stability} amounts to the claim that there is a unique $(s,q)$-admissible pattern $P$ with $\pi_P = \log(\ex_{\pi}(s,q))$, and what is more that there is a unique product-optimal weighting $\bfx$ for $P$ --- provided, naturally, that one adds the technical requirement that $x_v>0$ for all $v\in V(P)$ in every product-optimal weighting of $P$, to rule out the addition of dummy vertices to $P$.
\end{remark}


\section{The large ambient multiplicity regime}
\label{section: large a results}

\subsection{Strategy}
\label{section: large a strategy}

Let $T=\TUR{\bfr}{a}$ denote a generalised Turán pattern. Our main aim in this section is to prove that, for sufficiently large ambient edge multiplicity $a$, every product-extremal $(s, \SIGMA{T}{s})$-graph $G$ must be close in edit distance to a product-optimal blow-up of $T$.

To do this, we prove for some well-chosen $S>s$ that it is enough to show this for product-extremal $(s, \SIGMA{T}{s})$-graphs $G$ that are in addition $(s', \SIGMA{T}{s'})$-graphs for all $s'\leq S$. Here we use some averaging tools previously used by Füredi and Kündgen in the additive case as well as some precise information about the values of $\SIGMA{T}{s}$ and $\pi_T$ --- obtaining such information is a straightforward but tedious exercise in optimisation, requiring calculations and case-checking, which is in fact the main bottleneck for our approach.

We refer to $s'$-sets $U$ in $G$ with $e(G[U])= \SIGMA{T}{s'}$ as heavy sets. For our well-chosen $S>s$, we show that we can find a heavy $S$-set $U$ in $G$. A crucial property of $S$, besides being suitably large, is that there is a unique up to isomorphism sum-optimal blow-up of $T$ on $S$ vertices, and that this blow-up is regular. Now while $e(G[U])=\SIGMA{T}{S}$, we do not know a priori that $G[U]$ is a blow-up of $T$. We show that this must in fact be the case by first observing that $G[U]$ must be regular, then considering the interactions of $U$ with the rest of $G$ and applying the integer AM-GM inequality and some structural graph-theoretic argument.

Finally, once we have shown that $G[U]$ is an $S$-vertex blow-up of $T$, we are able to combine this with geometric averaging arguments similar to those deployed in~\cite{falgas2024extremal} to deduce that the overall structure of $G$ must be close to a blow-up of $T$.

The main innovation from our approach is precisely the use of large, regular heavy sets to find large blow-ups of $T$ --- in~\cite{falgas2024extremal}, blow-ups of $T$ were constructed `bottom-up', starting with a single edge of multiplicity $a+1$ and finding progressively larger structured pieces inside $G$. In the last stages of the construction of a suitably large blow-up of $T$, that approach required long chains of technical ad hoc graph-theoretic arguments that did not generalise in a natural fashion nor coped with the possible presence of edges with very low multiplicity. Indeed, the blow-ups found in~\cite{falgas2024extremal} after great effort only had three vertices per part and relied crucially on edges of multiplicity less than $a$ consisting of `clone-cliques' of multiplicity $a-1$ edges. As we show, it turns out to be significantly advantageous to first find large regular heavy sets without any a priori structure and to then derive the structure indirectly by considering their interactions with the rest of the multigraph.

\subsection{Preliminaries: properties of optimal blow-ups of generalised Turán patterns}
\label{section: preliminaries large a}

Fix $r_0, d, r_d \in \ZZ_{>0}$. Let $\bfr$ denote the $(d+1)$-tuple $\bfr=(r_0, 0, \dotsc, 0, r_d)$. For $a\geq d+1$, let $T=T(a)$ denote the generalised Turán pattern $\TUR{\bfr}{a}$.

We begin by establishing asymptotic properties of product-optimal blow-ups of $T$. Set
\begin{align*}
    x_\star=x_\star(\bfr, a) \defined \frac{\log\left((a+1)/a\right)}{\log \left( (a+1)^{r_d(r_0+1)}/\left(a^{r_d} (a-d+1)^{r_0(r_d-1)}(a-d)^{r_0}\right)  \right) }.
\end{align*}

\begin{proposition}[Properties of product-optimal blow-ups of $T$]
\label{prop: product-optimal blow ups of generalised Turan}
The following hold:
\begin{enumerate}[(i)]
    \item the unique product-optimal weighting $\bfx$ of $T$ satisfies $\bfr_i = (1-r_dx_\star/r_0)$ for all $i$: $1 \leq i \leq r_0$, and $\bfr_i = x_\star$  for all $i$: $r_0 < i \leq r_0 + r_d$;
    
    \item the function $x_\star = x_\star(a)$ is strictly increasing in $a$ over the interval $a \in [d + 1, \infty)$, and tends to the limit $\lim_{a \to \infty} x_\star(a) = 1/(r_d + r_0(d r_d + 1))$;
    
    \item $\pi_T = \log a + \p[\big]{1 - (1 - r_d x_\star)/r_0} \log\p[\big]{(a+1)/a}$.
\end{enumerate}
\end{proposition}
\begin{proof}
Parts (i) and (ii) are exercises in calculus (in fact part (i) can be deduced directly from \Cref{cor: balanced degrees in optimal pattern weightings} on balanced degrees in optimal weightings), while part (iii) follows from part (i).
\end{proof}

We now turn our attention to sum-optimal blow-ups of $T$. Set $s_0 \defined r_0+r_d+ dr_0r_d+1$.

\begin{proposition}[Partition sizes in sum-optimal blow-ups of $T$]
\label{prop: partition sizes in sum-optimal blow-ups of T}
Let $G$ be a sum-optimal $n$-vertex blow-up of $T$, and let $\sqcup_{i\in [r_0+r_d]}V_i $ be a $T$-partition of $G$. Let $n= q\left(s_0-1\right)+t$, where $q\in \ZZ_{\geq 0}$ and $0\leq t <s_0-1$. Then, setting $U \defined \sqcup_{i=r_0+1}^{r_0+r_d} V_i$, we have
\begin{equation*}
    \card{U} = q r_d + \begin{cases}
        0, & \text{if $t = 0$}, \\
        k - 1 \text{ or } k, &\text{if $(k - 1)(r_0 d + 1) + 1 \leq t \leq (k - 1)(r_0 d + 1) + r_0$, $k \in [r_d]$ }, \\
        k, & \text{if $(k - 1)(r_0 d + 1) + r_0 + 1 \leq t \leq k(r_0 d + 1)$, $k \in [r_d - 1]$}, \\
        r_d, & \text{if $(r_d - 1)(r_0 d + 1) + r_0 + 1 \leq t < s_0 - 1$.}
    \end{cases}
\end{equation*}   

\end{proposition}
\begin{proof}
This is a generalisation of~\cite[Proposition 5.3]{day2022extremal}, and is proved in exactly the same way. We know from Turán's theorem and the sum-optimality of $G$ that the vertices in the set $U \defined \sqcup_{i=r_0+1}^{r_0+r_d} V_i$ are distributed between the parts $U_i \defined V_{r_0 + i}$, $i \in [r_d]$ in a balanced fashion, and similarly that the parts $V_i$, $i \in [r_0]$ have balanced sizes. We may assume without loss of generality that
$\floor[\big]{\frac{n - \card{U}}{r_0}} = \card{V_1} \leq \card{ V_2} \leq \dotsb \leq \card{V_{r_0}} =\ceil[\big]{\frac{n - \card{U}}{r_0}}$, and similarly that  $\floor[\big]{\frac{\card{U}}{r_d}} = \card{U_1} \leq \card{U_2} \leq \dotsb \leq \card{U_{r_d}} = \ceil[\big]{\frac{\card{U}}{r_d}}$.

Switching a vertex from $U_{r_d}$ to $V_1$ changes $e(G)$ by
\begin{align}
\label{eq: moving from Urd to V1}
    -\card{V_1} + (d + 1)(\card{U_{r_d}} -1) +  d \p[\big]{\card{U} - \card{U_{r_d}}}
    = - \floor[\bigg]{\frac{n-\card{U}}{r_0}} + d \card{U} + \ceil[\bigg]{\frac{\card{U}}{r_d}} - (d + 1).
\end{align}
Similarly, moving a vertex from $V_{r_0}$ to $U_1$ changes $e(G)$ by
\begin{align}
\label{eq: moving from Vr0 to U1}
    -(d + 1)\card{U_1} - d \p[\big]{\card{U} - \card{U_1}} +  \card{V_{r_0}} - 1
    = -d\card{U} -\floor[\bigg]{\frac{\card{U}}{r_d}} + \ceil[\bigg]{\frac{n-\card{U}}{r_0}} - 1.
\end{align}
Since $G$ is sum-optimal, both of these quantities must be non-positive. In the case~\eqref{eq: moving from Urd to V1}, this yields an upper bound on the value of $\card{U}$, while in the case~\eqref{eq: moving from Vr0 to U1} we obtain an almost matching lower bound on $\card{U} $. Analysing the various cases depending on the value of $n$, this yields the claimed bounds on $\card{U}$. Explicit calculations can be found in the third author's master's thesis \cite[Proposition 5.2]{Sarkar2025-ms}.
\end{proof}

\begin{corollary}[Growth rate of $\SIGMA{T}{n}$]
\label{cor: growth rate of Sigma T}
Let $n= q\left(s_0-1\right)+t$, where $q\in \ZZ_{\geq 0}$ and $0\leq t < s_0-1$.
Then $\SIGMA{T}{n + 1} -\SIGMA{T}{n}$ is equal to 
\begin{align*}
    \begin{cases}
    (a+1)n-q(dr_d+1)-dk, & \text{if } t = k(r_0 d+1), \, 0 \leq k \leq r_d-1, \\
    (a+1)n-q(dr_d+1)-dk-\ell, & \text{if } k(r_0 d+1)+1+\ell r_0 \leq t \leq k(r_0 d+1)+(\ell+1)r_0, \\
    & \hspace{0.4cm} 0 \leq k \leq r_d-1, \, 0 \leq \ell \leq d-1, \\ 
    (a+1)n-(q+1)(dr_d+1)+1, & \text{if } r_d(r_0 d+1) \leq t <s_0-1.
    \end{cases}
\end{align*}    
\end{corollary}
\begin{proof}
We may read off the value of $\SIGMA{T}{n + 1}-\SIGMA{T}{n}$ from \Cref{prop: partition sizes in sum-optimal blow-ups of T}, which tells us the part sizes in the $T$-partition $\sqcup_i V_i$ of a sum-optimal blow-up of $T$ on $n$ vertices as well as which part we should add the $(n+1)$-th vertex to.
More precisely, setting $U = \sqcup_{i>r_0} V_i$ and assuming as in the proof of \Cref{prop: partition sizes in sum-optimal blow-ups of T} that $\card{V_1} = \floor[\big]{\frac{n-\card{U}}{r_0}}$ and $\card{V_{r_0+1}} = \floor[\big]{\frac{\card{U}}{r_d}}$, we should either add a new vertex to $V_1$, which keeps the size of $U$ constant and increases the edge-sum by $(a+1)n - \floor[\big]{\frac{n-\card{U}}{r_0}}$, or add it to $V_{r_0+1}$, which increases the size of $U$ by $1$ and increases the edge sum by $(a + 1)n - d \card{U} - \floor[\big]{\frac{\card{U}}{r_d}}$.
Analysing the various cases depending on the value of $n$, this yields the claimed bounds on  $\SIGMA{T}{n + 1} - \SIGMA{T}{n}$. Again, a reader interested in explicitly worked out computations can find these in~\cite[Proposition 5.3]{Sarkar2025-ms}.
\end{proof}

\begin{corollary}
\label{cor: growth rate of ex sigmaT s_0 up to s_1}
For all $n$ with $s_0 \leq n \leq 2s_0-dr_0-1$, $\ex_\Sigma(n,s_0,\SIGMA{T}{s_0})=\SIGMA{T}{n}$. 
\end{corollary}
\begin{proof}
We proceed by induction on $n$. The base case $n=s_0$ is trivial. For the inductive step, by \Cref{cor: arithmetic averaging}, it is enough to show that 
$ \p[\big]{\frac{n+1}{n-1}} \SIGMA{T}{n} < \SIGMA{T}{n + 1} + 1$ 
holds for all $n$: $s_0\leq n < 2s_0-dr_0-1$.  Rearranging terms, this is equivalent to showing 
\begin{align}
\label{eq: ineq needed in diff of sigma T below s_1}
    \left( \frac{2}{n-1}\right)\SIGMA{T}{n} < \SIGMA{T}{n + 1} -\SIGMA{T}{n} +1.
\end{align}
The value of $\SIGMA{T}{n}$ can be read out of the bounds on $|U|$ given in \Cref{prop: partition sizes in sum-optimal blow-ups of T}, while the difference $\SIGMA{T}{n + 1} -\SIGMA{T}{n}$ is given by \Cref{cor: growth rate of Sigma T}, with $q=1$. Plugging in these bounds in the three ranges of $n-s_0+1$ featured in the statements of \Cref{prop: partition sizes in sum-optimal blow-ups of T} and \Cref{cor: growth rate of Sigma T} yields the desired inequality~\eqref{eq: ineq needed in diff of sigma T below s_1}. Detailed calculations are given in~\cite[Lemma 6.4.3]{Sarkar2025-ms}.
\end{proof}

\begin{corollary}
\label{cor: growth rate not too large}
We have $ \SIGMA{T}{n + 1}-\SIGMA{T}{n}< \p[\big]{a+1 -\frac{dr_d +1}{s_0-1}}n + 1$. 
In particular, for all $a$ sufficiently large, we have 
\begin{align}
\label{eq: growth rate in Sigma less than xstar expected}
    \SIGMA{T}{n + 1} - \SIGMA{T}{n} - 1
    < \p[\Big]{a+1 -\frac{1 - r_d x_\star}{r_0} } n.
\end{align}
\end{corollary}
\begin{proof}
The first inequality can be read out directly from \Cref{cor: growth rate of Sigma T} (this is done explicitly in~ \cite[Proposition 5.4]{Sarkar2025-ms}). The `in particular' part of the corollary then follows from \Cref{prop: product-optimal blow ups of generalised Turan} (ii), and the monotone convergence of $x_\star(a)$ to $1/\left(r_d+r_0(dr_d+1)\right)$ as $a \to \infty$.    
\end{proof}

\begin{corollary}
\label{cor: bound on diff SigmaT (t) small t}
The following statements hold:
\begin{enumerate}[(i)]
    \item For every $s'$: $2\leq s' <s_0$, we have
    \begin{align*}
        \frac{\SIGMA{T}{s' + 1} -\SIGMA{T}{s'} -1-as'}{s'} \leq 1 - \frac{d(r_d-1)+1}{r_0\left(d(r_d-1)+1\right)+r_d-1}.
    \end{align*}
    
    \item For every $s'$: $s_0\leq s'<    2s_0-dr_0-2$, we have
    \begin{align*}
        \frac{\SIGMA{T}{s' + 1} -\SIGMA{T}{s'} -1-as'}{s'} \leq  1 - \frac{d(2r_d-1)+2}{r_0\left(d(2r_d-1)+2\right)+2r_d-1}.
    \end{align*}
\end{enumerate}
\end{corollary}
\begin{proof}
Apply \Cref{cor: growth rate of Sigma T} with $n=s'$ and  $s' \leq s_0-1$ for part (i), and $s_0 \leq s' < 2 s_0 - d r_0 - 2$ for part (ii) (these two calculations are also given in full detail in \cite[Proposition A2]{Sarkar2025-ms} and \cite[Proposition A3]{Sarkar2025-ms} respectively). 
\end{proof}

Finally, we use the previous results to obtain the following generalisation of~\cite[Theorem 3.11]{day2022extremal}, which for sufficiently large values of the ambient edge multiplicity $a$ allows us to deduce the value of $\ex_\Pi(s+1, \SIGMA{T}{s+1})$ from that of $\ex_\Pi(s_0, \SIGMA{T}{s_0})$, where $s_0=r_0(dr_d+1)+r_d+1$ is our `base case'.

\begin{proposition}
\label{prop: step up from the base case}
Let $s$ be an integer with $s \geq s_0=r_0(dr_d+1)+r_d+1$. Suppose there exists $a_0\geq d+1$ such that for all $a \geq a_0$ we have
\begin{align*}
    \ex_\Pi(s, \SIGMA{T(a)}{s})=e^{\pi_{T(a)}}.
\end{align*}
Then, there exists $a_1 \geq a_0$ such that for all $a \geq a_1$ we have
\begin{align*}
    \ex_\Pi(s+1, \SIGMA{T(a)}{s + 1})=e^{\pi_{T(a)}}.
\end{align*}
Further, if $G$ is a near product-extremal $(s+1,\SIGMA{T(a)}{s + 1})$-graph on $n$ vertices, then $G$ is within edit distance $o(n^2)$ of a near product-extremal $(s,\SIGMA{T(a)}{s})$-graph on $n$ vertices.
\end{proposition}
\begin{proof}
The proof is essentially identical~\cite[Theorem 3.11]{day2022extremal}, with~\eqref{eq: growth rate in Sigma less than xstar expected} replacing~\cite[Proposition 5.5]{day2022extremal} as the input at the start of the argument, and all instances of $\PI{T}{n}$ in the proof replaced by $\left(\PI{T}{n}\right)^{1+o(1)}$.
\end{proof}

\subsection{Extremal geometric density and stability: proof of \texorpdfstring{\Cref{theorem: base case main conj is true,theorem: large a arbitrary number of deficient blobs,theorem: stability}}{Theorems 1.12, 1.15 and 1.17}}

In this section, we fix $r_0, r_d, d\in \ZZ_{>0}$. Let $\bfr$ denote the $(d+1)$-tuple $(r_0, 0, \dotsc, 0, r_d)$, and write $T=T(a)$ for the generalised Turán pattern $\TUR{\bfr}{a}$. Further, set again $s_0 \defined r_d+r_0(dr_d+1)+1$. We shall derive all of our main results from the following technical statement:

\begin{theorem}
\label{theorem: large a technical statement}
Suppose the ambient edge multiplicity $a$ satisfies
\begin{equation}
\label{eq: poly equation needed for good a}
    (a+1)^{r_d(d-1)(2r_d-1)+2r_d}(a-d)^{2r_d-1}(a-d+1)^{(2r_d-1)(r_d-1)} > a^{r_dd(2r_d-1)+2r_d}.
\end{equation}
Then $\ex_\Pi(s_0, \SIGMA{T}{s_0}) = e^{\pi_T}$, and further if $G$ is an $(s_0, \SIGMA{T}{s_0})$-graph on $n$ vertices with $P(G)\geq (\Pi_T(n))^{1-o(1)}$, then $G$ lies within $o(n^2)$ edit distance of a product-optimal blow-up of $T$.
\end{theorem}

Our arguments will use the following simple statement about polynomial inequalities.

\begin{proposition}
\label{prop: appendix poly 1}
Let $r_d, r_0, d, a \in \mathbb{Z}_{>0}$ with $a \geq d+1$. The following statements are equivalent:
\begin{align*}
    & \text{(a)} \hspace{0.3cm}(a+1)^{r_d(d-1)(2r_d-1)+2r_d} (a-d)^{2r_d-1} (a-d+1)^{(2r_d-1)(r_d-1)} 
    \geq a^{dr_d(2r_d-1)+2r_d}, \\
    &\text{(b)} \hspace{0.3cm} d(2r_d-1)+2\geq (2r_d-1)\frac{1-r_dx_\star}{r_0 r_d x_\star},\\
    &\text{(c)} \hspace{0.3cm} \frac{d(2r_d-1)+2}{r_0\left(d(2r_d-1)+2\right)+2r_d-1} \geq \frac{1-r_d x_\star}{r_0}.
\end{align*}
Moreover, strict inequality holds in one of these statements if and only if it holds in all.
\end{proposition}
\begin{proof}
The equivalence of (a) and (b) can be shown by observing that 
\begin{align*}
    \frac{1 - r_d x_{\star}}{r_0 \, r_d \,  x_{\star}}
    = \frac{r_d \log\p[\big]{\frac{a+1}{a}} + \log\p[\big]{\frac{a}{a-d}} + (r_d - 1)\log\p[\big]{\frac{a}{a-d+1}} }{r_d \log\p[\big]{\frac{a+1}{a}} }.
\end{align*}
The equivalence of (b) and (c) follows from a simple calculation.
\end{proof}

\begin{proof}[Proof of \Cref{theorem: large a technical statement}]
Consider an $(s_0, \SIGMA{T}{s_0})$-graph $G$ on $n$ vertices with the property that $P(G) \geq (\Pi_T(n))^{1-o(1)}$.
By \Cref{lemma: removal of low deg vertices}(ii), one can remove at most $o(n)$ vertices from $G$ to obtain an $(s_0, \SIGMA{T}{s_0})$-graph $G'$ on $n'=n-o(n)$ vertices with minimum product degree at least $e^{(\pi_T-o(1))n'}$. Clearly, if $G'$ is $o((n')^2)$-close in edit distance to a product-optimal blow-up of $T$, then $G$ is $o(n^2)$-close in edit distance to a product-optimal blow-up of $T$. Thus, we may assume in all that follows that $G$ has minimum product-degree at least $e^{(\pi_T-o(1))n}$. Recall from \Cref{prop: product-optimal blow ups of generalised Turan} that $e^{\pi_T} = a \p[\big]{(a+1)/a}^{1 - (1 - r_d x_\star)/r_0}$.

We begin by showing that small sets of vertices in $G$ of size $s'\leq s_0$ support no more edges than we would expect in a blow-up of $T$.

\begin{lemma}
\label{lemma: large a good subsets are smaller}
The multigraph $G$ is an $(s', \SIGMA{T}{s'})$-graph for all $s'\leq s_0$. In particular, the maximum edge multiplicity in $G$ is $a+1$.
\end{lemma}
\begin{proof}
Suppose not. Let $t\leq s_0$ be the largest integer such that $G$ fails to be a $(t,\SIGMA{T}{t})$-graph. Observe that $2 \leq t \leq s_0-1=dr_0r_d+r_0+r_d$. Then there exists a set of $t$ vertices with at least $\SIGMA{T}{t}+1$ edges. Let us denote this set by $X$. By the maximality of $t$, every vertex outside $X$ sends at most $\SIGMA{T}{t + 1} -\SIGMA{T}{t} -1 \leq at + \left(1 - \frac{d(r_d-1)+1}{r_0\left(d(r_d-1)+1\right)+r_d-1}\right)t$ edges into $X$, with the inequality following from \Cref{cor: bound on diff SigmaT (t) small t} part (i).

Consider $p \defined \prod_{v \in X}p_{G}(v)^{\frac{1}{t}}$. By the integral AM-GM inequality, \Cref{prop: integer AM-GM}, the contribution of every vertex outside $X$ to $p$ is at most $a \left( \frac{a+1}{a} \right)^{1 - \frac{d(r_d-1)+1}{r_0\left(d(r_d-1)+1\right)+r_d-1}}$. By geometric averaging, \Cref{prop: geometring averaging}, there is a vertex $v\in X$ with
\begin{align*}
    p_G(v) \leq \p[\bigg]{a \p[\Big]{\frac{a+1}{a}}^{1 - \frac{d(r_d-1)+1}{r_0\left(d(r_d-1)+1\right)+r_d-1}} }^{n+o(n)}.
\end{align*}
Now, by \Cref{prop: appendix poly 1} and our assumption~\eqref{eq: poly equation needed for good a} on $a$, we have $1 - \frac{d(r_d-1)+1}{r_0\left(d(r_d-1)+1\right)+r_d-1 }< 1-\frac{1-r_dx_\star}{r_0}$, and thus the above upper bound on $p_G(v)$ contradicts our minimum product degree assumption.
\end{proof}

Set $s_1 \defined 2s_0 -dr_0-2$. By \Cref{cor: growth rate of ex sigmaT s_0 up to s_1}, we know that for all $s'$ with $s_0\leq s' \leq s_1$, $G$ must be an $(s', \SIGMA{T}{s'})$-graph as well as being an $(s_0, \SIGMA{T}{s_0})$-graph.

\begin{definition}
A set of vertices $X\subseteq V(G)$ is called a \textbf{heavy set} if  $e_{G}(X)=\SIGMA{T}{\card{X}}$ (i.e.\ if it has the maximum number of edges allowed in $G$).   
\end{definition}

\begin{lemma}
\label{lemma: have heavy s_1-sets}
The multigraph $G$ contains a heavy $s_1$-set.   
\end{lemma}
\begin{proof}
Suppose for a contradiction that $G$ contains no heavy $s_1$-set. If every edge in $G$ has multiplicity at most $a$, then every vertex in $G$ has product degree at most $a^{n-1}$, contradicting our minimum product degree assumption. Thus, $G$ contains a heavy $2$-set. Set then $t$: $2\leq t< s_1$ to be the largest integer less than $s_1$ for which $G$ contains a heavy $t$-set $X$. By the maximality of $t$, every vertex $u$ outside $X$ sends at most a total of $\SIGMA{T}{t + 1}-\SIGMA{T}{t}-1$ edges into $X$, which by \Cref{cor: bound on diff SigmaT (t) small t} is at most 
\begin{align}
\label{eq: ub on v contribution to heavy t-set if no heavy t+1 set}
    at + \left(1-\frac{d(2r_d-1)+2}{s_1-1}\right)t.   
\end{align}
Set $p \defined \prod_{v\in X} p_G(v)^{\frac{1}{t}}$. By the integral AM-GM inequality, \Cref{prop: integer AM-GM}, it follows from~\eqref{eq: ub on v contribution to heavy t-set if no heavy t+1 set} that $u$ contributes at most $ a \left((a+1)/a\right)^{1-\frac{d(2r_d-1)+2}{s_1-1}}$ to $p$, which by \Cref{prop: appendix poly 1} part (c) and our assumption~\eqref{eq: poly equation needed for good a} on $a$ is strictly less than $a\left((a+1)/a\right)^{1-\frac{1-r_dx_\star}{r_0}}$. By geometric averaging, it follows that some vertex in $X$ has product degree at most $p\leq a \left((a+1)/a\right)^{\left(1-\frac{d(2r_d-1)+2}{s_1-1}\right)n+O(1)}$, contradicting our minimum product degree assumption on $G$. 
\end{proof}

In the next few lemmas, we show that all heavy $s_1$-sets have a very specific structure: they are blow-ups of $T$. Let $X \subseteq V(G)$ be an arbitrary heavy set on $s_1$ vertices.

\begin{lemma}
\label{lemma: heavy sets on S vertices are regular}
For every vertex $v \in X$, $d_{X}(v)=(a+1)(s_1-1)-d(2r_d-1)-1$.
\end{lemma}
\begin{proof}
By \Cref{cor: growth rate of ex sigmaT s_0 up to s_1}, every set of $s_1-1$ vertices in $G$ contains at most $\SIGMA{T}{s_1 - 1}$ edges. Thus, for every vertex $v \in X$, we have 
\begin{align*}
    \SIGMA{T}{s_1 - 1}\geq e_{G}(X\setminus\set{v}) =e_{G}(X)-d_X(v)=\SIGMA{T}{s_1}-d_X(v).
\end{align*}
Rearranging terms and applying \Cref{cor: growth rate of Sigma T}, we obtain the  lower-bound inequality $d_X(v)\geq \SIGMA{T}{s_1}-\SIGMA{T}{s_1 - 1}=(a+1)(s_1-1)-d(2r_d-1)-1$.  On the other hand, it follows from \Cref{prop: partition sizes in sum-optimal blow-ups of T} and the hand-shaking lemma that 
\begin{align*}
    \sum_{v \in X}d_{X}(v)=2  \SIGMA{T}{s_1}&= 2\left((a+1)\binom{s_1}{2}-r_0 \binom{d(2r_d-1)+2}{2}-d\binom{2r_d}{2}-r_d\right)\\
    &= s_1\left((a+1)(s_1-1)-d(2r_d-1)-1\right)\leq \sum_{v\in X}d_X(v).
\end{align*}
It follows immediately that $G[X]$ is a $\left((a+1)(s_1-1)-d(2r_d-1)-1\right)$-regular multigraph.
\end{proof}

\begin{lemma}
\label{lemma: vertices outside don't send too many edges}
Every vertex outside $X$ sends at most $(a+1)s_1-d(2r_d-1)-2$ edges into $X$. 
\end{lemma}
\begin{proof}
By \Cref{cor: growth rate of ex sigmaT s_0 up to s_1},  $G$ is an $(s_1+1 , \SIGMA{T}{s_1 + 1})$-graph. It follows that every vertex outside $X$ sends at most $\SIGMA{T}{s_1 + 1}-\SIGMA{T}{s_1}$ edges into $X$. By \Cref{cor: growth rate of Sigma T}, this latter quantity is equal to $(a+1)s_{1}-d(2r_d-1)-2$, as desired.
\end{proof}

Let $A \defined \set{v \in V(G) \setminus X \st v \text{ sends } (a+1)s_1-d(2r_d-1)-2 \text{ edges into } X }$ be the set of vertices outside $X$ sending the maximum possible number of edges into $X$.

\begin{lemma}
\label{lemma: vertices sending many edges only send edge of weight at least a and have cloning property}
For all $v \in A$, we have $w(uv) \geq a$ for all $u \in X$. Furthermore, if $w(uv) = a$ for some $u\in X$, then we have $w(uy) = w(vy)$ for all $y \in X \setminus \set{u}$.       
\end{lemma}
\begin{proof}
For any $u \in X$, we have by \Cref{lemma: heavy sets on S vertices are regular} that $d_X(u) = d_X(v) - a$, and hence 
\begin{align*}
    \SIGMA{T}{s_1}\geq  e_{G}(\left(X \setminus \set{u}\right)\cup \set{v})  =\SIGMA{T}{s_1}+d_X(v)-d_X(u)-w(uv)= \SIGMA{T}{s_1}+a-w(uv), 
\end{align*}
from which it follows that $w(uv)\geq a$. Further, if $w(uv)=a$, then we have equality in the above and thus $X'=\left(X\setminus \set{u} \right)\cup \set{v}$ is a heavy $s_1$-set. It then follows from \Cref{lemma: heavy sets on S vertices are regular} that, like $G[X]$, the multigraph $G[X']$ is  $\left((a+1)(s_1-1)-d(2r_d-1)-1\right)$-regular. Since $X'$ was obtained from $X$ by replacing $u$ by $v$, $u$ and $v$ must behave identically with respect to the vertices of $X\cap X'$, and the `furthermore' part of the lemma follows.
\end{proof}

For every vertex $v \in A$, we may thus define a set $B_{v} \defined \set{u \in X \st w(vu) = a}$ of those vertices $u$ in $X$ such that $u,v$ are clones of each other with respect to $X\setminus \set{u}$. Since every vertex in $A$ sends $(a+1)s_1-d(2r_d-1)-2$ edges into $X$ and since those edges all have multiplicity at least $a$ by \Cref{lemma: vertices sending many edges only send edge of weight at least a and have cloning property}, it follows that 
\begin{align}
\label{eq: Bv set sizes}
    \vert B_v\vert=d(2r_d-1)+2
\end{align} for all $v\in A$. Moreover, by the same lemma, all edges within $B_{v}$ have multiplicity $a$, while all edges from $B_{v}$ to $X \setminus B_{v}$ have multiplicity $a+1$.

\begin{corollary}
\label{cor: the Bv are identical or disjoint}
Let $v_1,v_2 \in A$. Then, either $B_{v_1}=B_{v_2}$ or $B_{v_1} \cap B_{v_2}=\emptyset$.
\end{corollary}
\begin{proof}
Follows immediately from our observation that for $i\in [2]$, all edges within $B_{v_i}$ have multiplicity $a$, while all edges between $B_{v_i}$ and $X \setminus B_{v_i}$ have multiplicity $a+1$.
\end{proof}

Let $m$ be the largest integer such that there exist vertices $v_1,\dotsc,v_{m} \in A$ with $B_{v_1},\dotsc,B_{v_{m}}$ being pairwise disjoint. Since $\vert B_{v_i}\vert =d(2r_d-1)+2$ for each $i \in [m]$, and $\vert X\vert =s_1=2r_d+r_0\left(d(2r_d-1)+2\right)$, we have the inequality $m \leq r_0$.  This must in fact be an equality, as we show in our next lemma.

\begin{lemma}
\label{lemma: can find good collection of Bv}
We have $m=r_0$.
\end{lemma}
\begin{proof}
For each $i \in [m]$, we let $A_i \defined \set{v \in A \st B_{v} = B_{v_i} }$.
By \Cref{cor: the Bv are identical or disjoint} and the maximality of $m$, we have $A= \cup_{i=1}^{m}A_i$.
Set $\alpha_i n \defined \vert A_i\vert $ for each $i\in [m]$, $\alpha n \defined \vert A\vert$, and consider $p \defined \prod_{u \in X}p_G(u)^{\frac{1}{s_1}}$.

By the integral AM-GM inequality (\Cref{prop: integer AM-GM}), the contribution of every vertex in $A$ to $p$ is at most $a\left(\frac{a+1}{a} \right)^{1-\frac{d(2r_d-1)+2}{s_1}}$, while the contribution of every vertex in $(V(G) \setminus X)\setminus A$ to $p$ is at most $a\left(\frac{a+1}{a} \right)^{1-\frac{d(2r_d-1)+3}{s_1}}$. By geometric averaging (\Cref{prop: geometring averaging}), there is a vertex $u \in X$ with $p_G(u) \leq p\leq \left(a\left( (a+1)/a\right)^{ 1-\frac{d(2r_d-1)+3}{s_1}+\frac{\alpha}{s_1}}\right)^{n+O(1)}$.
By our minimum product degree assumption on $G$, this implies $1-\frac{d(2r_d-1)+3}{s_1}+\frac{\alpha}{s_1}\geq 1-\frac{1-r_dx_\star}{r_0}+o(1)$, which after rearranging terms yields:
\begin{align}
\label{eq: lower bound on alpha}
    \alpha\geq \frac{r_0 - 2r_d + r_d s_1 x_\star}{r_0} + o(1).
\end{align}
Consider now a vertex $u \in B_{v_i}$; it receives no edge of multiplicity $a+1$ from $A_i$, and thus we have $p_G(u) \leq \left(a \left( (a+1)/a\right)^{1-\alpha_i}\right)^{n+o(n)}$.
By our minimum product degree assumption on $G$, this implies that $1-\alpha_i +o(1)\geq 1-\frac{1-r_dx_\star}{r_0}$, which after rearranging terms yields $\alpha_i\leq \frac{1-r_dx_\star}{r_0}+o(1)$.
It follows that $\alpha=\sum_{i=1}^{m}\alpha_i\leq m\left( \frac{1-r_dx_\star}{r_0}\right)+o(1)$. Combining this upper bound with~\eqref{eq: lower bound on alpha}, we get 
\begin{align*}
    m \geq \frac{r_0-2r_d+r_ds_1x_\star}{1-r_dx_\star}+o(1).
\end{align*}
We observe that $\frac{r_0-2r_d+r_ds_1x_\star}{1-r_dx_\star} > r_0-1$ is equivalent to $d(2r_d-1)+3 > (2r_d-1)\frac{1-r_dx_\star}{r_0r_dx_\star}$, which holds by \Cref{prop: appendix poly 1} and our assumption~\eqref{eq: poly equation needed for good a} on $a$. Thus, $m > r_0-1$. Since $m \leq r_0$, this means that $m=r_0$.
\end{proof}

Let $X' \defined X \setminus (\cup_{i=1}^{r_0}B_{v_i})$ be the collection of vertices of $X$ not included in some $B_v$, $v\in A$. Observe that $\vert X'\vert =2r_d$.

\begin{lemma}
\label{lemma: edges in X' have bounded multiplicity}
All edges in $X'$ have multiplicity at most $a-d+1$.
\end{lemma}
\begin{proof}
Suppose there exists an edge $u_1u_2$ in $X'$ with $w(u_1 u_2)=a-d+1+k$ for some $k \in [d]$. Then $\set{u_1, u_2}$, taken along with $d-k+1$ vertices from each of the sets $B_{v_i}$ yields a set of $r_0(d-k+1)+2$ vertices with $\SIGMA{T}{r_0(d-k+1)+2} +1$ edges. This contradicts the fact that $G$ is a $(s',\SIGMA{T}{s'})$-graph for all $2 \leq s' \leq s_0$.
\end{proof}

\begin{corollary}
\label{cor: large blow up found}
The multigraph $G[X]$ is a blow-up of $T$, with the $T$-partition $\sqcup_{i \in [r_0+r_d]} X_i$ of $X$ satisfying $\card{X_i} = d(2 r_d - 1) + 2$ for all $i \in [r_0]$ and $\card{X_i} = 2$ for all $i \in [r_0+r_d] \setminus [r_0]$.   
\end{corollary}
\begin{proof}
For every vertex $y \in X'$, we have 
\begin{align*}
    d_{X'}(y)
    &= d_{X}(y) - (a+1)\p[\big]{r_0(d(2 r_d - 1) + 2)}\\
    &= (a + 1)(s_1 - 1) - d(2 r_d - 1) - 1 - (a + 1)\p[\big]{r_0(d(2 r_d - 1) + 2)} \\
    &=(2 r_d - 1)(a - d + 1) - 1.
\end{align*}
Since edges in $X'$ have multiplicity at most $a-d+1$ by \Cref{lemma: edges in X' have bounded multiplicity}, $X'$ must consist of a perfect matching of edges of multiplicity $a-d$, with all other edges having multiplicity exactly $a-d+1$. We can thus write $X'$ as $\sqcup_{i=r_0+1}^{r_0+r_d}X_i$, with the $X_i$ corresponding to the edges of this perfect matching.

We had already shown in \Cref{lemma: can find good collection of Bv} that we can find $r_0$ sets $X_i \defined B_{v_i}$, $i\in [r_0]$, each of which has size $d(2r_d-1)+2$ by~\eqref{eq: Bv set sizes}, such that edges in $G[X_i]$ have multiplicity exactly $a$ and edges from $X_i$ to $X\setminus X_i$ have multiplicity $a+1$. Thus, $\sqcup_{i=1}^{r_0+r_d}X_i$ gives us a $T$-partition of $X$ with the desired properties.
\end{proof}

Having thus found a large blow-up of $T$, we can conclude our proof with a weighted geometric averaging argument similar to that found in~\cite[Lemma 3.3]{falgas2024extremal}. Given the $T$-partition of $X$ we have obtained in \Cref{cor: large blow up found}, pick one vertex $u_i$ from each of the sets $X_i$, $i \in [r_0+r_d]\setminus [r_0]$, to form a set $U$. Further, let $W \defined \sqcup_{i=1}^{r_0} X_i$ and $L \defined d(2r_d-1)+2$. Set
\begin{align*}
    p \defined \left(\prod_{v\in W} p_G(v)^{\frac{1-r_dx_\star}{Lr_0}}\right) \left(\prod_{v\in U} p_G(v)^{x_\star}\right).
\end{align*}
By weighted geometric averaging, \Cref{prop: geometring averaging}, we know that there is some vertex $v\in U\cup W$ with $p_G(v) \leq p$. We shall now bound the contribution to $p$ from vertices outside $U\cup W$, and then apply our stability version of geometric averaging, \Cref{lemma: stability geometric averaging} together with our minimum product degree assumption to deduce information about the general structure of $G$. To bound the contributions to $p$, we shall require the following simple polynomial inequalities.

\begin{proposition}
\label{prop: sub-optimal contribution 1}
For all $0 \leq k \leq d$, 
\begin{align*}
    \left(\frac{d-k}{L}\right)\left(\frac{1-r_dx_\star}{r_0}\right)-r_dx_\star\frac{\log \left(\frac{a+1}{a-d+1+k} \right)}{\log \left( \frac{a+1}{a}\right)} \leq 0,
\end{align*}
with equality if and only if $k=d$.
\end{proposition}

\begin{proof}
For $k=d$, equality holds.  Let $\ell=d-k$. To prove the proposition, we need to show that for all $1 \leq \ell \leq d$, we have 
\begin{align*}
    L > \frac{\ell \log \left( \frac{a+1}{a} \right)}{\log \left( \frac{a+1}{a+1-\ell}\right)}\left( \frac{1-r_dx_\star}{r_0r_dx_\star}\right).
\end{align*}
By the integral AM-GM inequality (\Cref{prop: integer AM-GM}), we have $(a+1)^{\ell}(a+1-\ell) \leq a^{\ell}(a+1)$, which, in turn, implies that $\log \left( \frac{a+1}{a}\right)\leq \frac{1}{\ell}\log \left( \frac{a+1}{a+1-\ell}\right)$. Thus, it suffices to show that $L > \frac{1-r_dx_\star}{r_0r_dx_\star}$. By \Cref{prop: appendix poly 1} and our assumption~\eqref{eq: poly equation needed for good a} on $a$ , we have $L =d(2r_d-1)+2 > (2r_d-1)\frac{1-r_dx_\star}{r_0r_dx_\star} \geq \frac{1-r_dx_\star}{r_0r_dx_\star}$, and the proposition follows.
\end{proof}

\begin{proposition}
\label{prop:sub-optimal contribution 2}
    The following inequality holds:
    \begin{align*}
        \frac{dr_d}{L}\left(\frac{1-r_dx_\star}{r_0}\right)-r_dx_\star\frac{\log \left(\frac{a+1}{a-d+1} \right)}{\log \left( \frac{a+1}{a}\right)} < 0.
    \end{align*}
\end{proposition}

\begin{proof}
    The inequality above is equivalent to the inequality $L > \frac{dr_d \log \left( \frac{a+1}{a} \right)}{\log \left( \frac{a+1}{a-d+1}\right)}\left( \frac{1-r_dx_\star}{r_0r_dx_\star}\right)$. By the integral AM-GM inequality (Proposition \ref{prop: integer AM-GM}), we have $(a+1)^d(a-d+1) \leq a^d(a+1)$, which, in turn, implies that $\log \left(\frac{a+1}{a}\right) \leq \frac{1}{d}\log \left( \frac{a+1}{a-d+1}\right)$. Thus, it suffices to show that $L > r_d \left(\frac{1-r_dx_\star}{r_0r_dx_\star} \right)$. By Proposition \ref{prop: appendix poly 1} and our assumption \eqref{eq: poly equation needed for good a} on $a$, we have $L=d(2r_d-1)+2 > (2r_d-1)\frac{1-r_dx_\star}{r_0r_dx_\star} \geq r_d \left(\frac{1-r_dx_\star}{r_0r_dx_\star} \right)$, and the proposition follows.
\end{proof}

\begin{lemma}
\label{lemma: contributions of outside vertices}
For every $v\in V(G)\setminus \left(U \cup W\right)$, the contribution of edges from $v$ to the quantity $p$ is bounded above by $a((a+1)/a)^{1-\frac{1-r_dx_\star}{r_0}}$, with equality attained if and only if one of the following alternatives occur:
\begin{enumerate}[(a)]
    \item $v$ sends edges of multiplicity $a$ to some $X_i\subseteq W$ and sends edges of multiplicity $a+1$ to all other vertices in $\left(U\cup W\right)\setminus X_i$;
    
    \item $v$ sends an edge of multiplicity $a-d$ to some $u\in U$, edges of multiplicity $a-d+1$ to $U\setminus \set{u}$ and edges of multiplicity $a+1$ to vertices in $W$.
\end{enumerate}
\end{lemma}
\begin{proof}
Let $m(v) \defined \max_{u \in U} w(uv)$, and let $v'\in U$ be a vertex with $w(vv')=m(v)$. We split into two cases, corresponding to our two alternatives. 

\textbf{Case I}: $m(v) = a - d + 1 + k$, where $ 1 \leq k \leq d$.
In this case, there exists a part $X_i$ to which $v$ sends at most $d-k$ edges of multiplicity $a+1$. Otherwise, picking $v,v'$ and $d-k+1$ vertices from each of the sets $\set{v''\in X_j \st w(vv'') = a + 1}$, $j \in [r_0]$, we obtain a set of $r_0(d-k+1)+2<s_0$ vertices in $G$ supporting $\SIGMA{T}{r_0(d-k+1)+2} + 1$ edges, contradicting \Cref{lemma: large a good subsets are smaller}. 
Thus, the contribution of $v$ to $p$ in this case is at most 
\begin{align*}
    &(a+1)\left(\frac{a-d+1+k}{a+1}\right)^{r_dx_\star}\left( \frac{a}{a+1} \right)^{\left(1-\frac{d-k}{L}\right)\left( \frac{1-r_dx_\star}{r_0}\right)}\\
    &=a \left( \frac{a+1}{a}\right)^{1-\frac{1-r_dx_\star}{r_0}+\left( \left(\frac{d-k}{L}\right)\left(\frac{1-r_dx_\star}{r_0}\right)-r_dx_\star\frac{\log \left(\frac{a+1}{a-d+1+k} \right)}{\log \left( \frac{a+1}{a}\right)}\right)}.
\end{align*} 
By \Cref{prop: sub-optimal contribution 1}, the contribution of $y$ to $p$ is at most $a \left( \frac{a+1}{a}\right)^{1-\frac{1-r_dx_\star}{r_0}}$, with equality if and only if $k=d$ and $v$ sends edges of multiplicity $a$ to some $X_i$, $i\in [r_0]$, and edges of multiplicity $a+1$ to $U\cup W\setminus X_i$ --- i.e.\ if and only if alternative (a) occurs.

\textbf{Case II}: $m(v)\leq a-d+1$. Suppose first of all that all edges from $v$ to $U$ have multiplicity exactly $a-d+1$. Then there exists a part $X_i$ to which $v$ sends at most $dr_d$ edges of multiplicity $a+1$. Indeed, otherwise it is easily seen that $G$ contains a set of $s_0=r_0(dr_d+1)+r_d+1$ vertices with $\SIGMA{T}{s_0}+1$ edges, a contradiction. Thus, the contribution of $v$ to $p$ is at most
\begin{align*}
    &(a+1)\left(\frac{a-d+1}{a+1}\right)^{r_dx_\star}\left( \frac{a}{a+1} \right)^{\left(1-\frac{dr_d}{L}\right)\left( \frac{1-r_dx_\star}{r_0}\right)}\\
    &=a \left( \frac{a+1}{a}\right)^{1-\frac{1-r_dx_\star}{r_0}+\left(\frac{dr_d}{L}\left(\frac{1-r_dx_\star}{r_0}\right)-r_dx_\star\frac{\log \left(\frac{a+1}{a-d+1} \right)}{\log \left( \frac{a+1}{a}\right)}\right)}.
\end{align*}

By \Cref{prop:sub-optimal contribution 2}, this is strictly less than $a \left( (a+1)/a\right)^{1-\frac{1-r_dx_\star}{r_0}}$. 

On the other hand, if $v$ sends at least one edge of multiplicity at most $a-d$ into $U$, then the contribution of $v$ to $p$ is trivially at most $a \left( \frac{a+1}{a}\right)^{1-\frac{1-r_dx_\star}{r_0}}$, with equality achieved if and only if $v$ sends one edge of multiplicity $a-d$ to $U$, $r_d-1$ edges of multiplicity $a-d+1$ to $U$, and edges of multiplicity $a+1$ to all vertices in $W$ --- i.e.\  if and only if alternative (b) occurs. 
\end{proof}

For all $i \in [r_0+r_d]\setminus[r_0]$, we set $U_i$ to be the set of vertices $v \in V(G)\setminus \left(U\cup W\right)$ such that $w(u_iv)=a-d$, $w(vv')=a-d+1$ for all $v' \in U\setminus\set{u_i}$ and $w(vv') = a + 1$ for all $v' \in W$. Further, for all $j \in [r_0]$, we define $W_{j}$ to be the set of vertices $v \in V(G) \setminus \left(U\cup W\right)$ such that $w(vv') = a$ for all $v' \in X_{j}$ and $w(vv')=a+1$ for all $v' \in \left(U\cup W\right)\setminus X_{j}$.

We let $\card{U_i} \defined \mu_in$, $\vert W_{j}\vert \defined \nu_{j}n$, and set $\mu \defined \sum_{i=r_0+1}^{r_0+r_d} \mu_i$, $\nu \defined \sum_{j=1}^{r_0} \nu_{j}$.

\begin{lemma}
\label{lemma: stability part sizes}
For all $i\in [r_0+r_d]\setminus [r_0]$, we have $\mu_i =x_\star+o(1)$, while for all $j \in [r_0]$, we have $\nu_j=(1-r_dx_\star)/r_0 +o(1)$.
\end{lemma}
\begin{proof}
By \Cref{lemma: stability geometric averaging} and our minimum product degree assumption for $G$, we immediately have that 
\begin{align}
\label{eq: mu plus nu is 1-o(1)}
    \mu+ \nu = 1-o(1).
\end{align}

Further, for each $j \in [r_0]$, consider a vertex $v \in X_{j}$. We have $p_G(v) \leq \left(a\left((a+1)/a\right)^{1-\nu_{j}}\right)^{n+o(n)}$, which by the minimum product degree assumption implies $\nu_{j} \leq \frac{1-r_dx_\star}{r_0}+o(1)$ and hence $\nu \leq (1-r_dx_\star)+o(1)$. Combined with~\eqref{eq: mu plus nu is 1-o(1)}, this gives $\mu \geq r_dx_\star+o(1)$.

Now, for each $i \in [r_0+r_d]\setminus [r_0]$, we have 
\begin{align*}
    p_G(u_i) \leq a^{n}\left( \frac{a-d+1}{a+1}\right)^{\mu n}\left( \frac{a-d}{a-d+1}\right)^{\mu_in}\left(\frac{a+1}{a} \right)^{n+o(n)}.
\end{align*}
The minimum product degree assumption on $G$, together with our lower bound $\mu \geq r_dx_\star+o(1)$ implies that $\mu_i \leq x_\star+o(1)$. Together with our lower bound $\mu\geq r_dx_\star+o(1)$, this immediately implies $\mu_i=x_\star+o(1)$ for every $i \in [r_0+r_d] \setminus [r_0]$. Thus, $\mu = r_dx_\star+o(1)$, and by~\eqref{eq: mu plus nu is 1-o(1)} we have $\nu= 1-r_dx_\star+o(1)$. Together with our upper bound on the $\nu_j$, this implies $\nu_j=(1-r_dx_\star)/r_0+o(1)$ for every $j\in [r_0]$. 
\end{proof}

We are now almost done. We make the following observations:
\begin{enumerate}[ --- ]
    \item
    Every edge $v_1 v_2$ in $\binom{U_i}{2}$ has multiplicity at most $a-d$. Otherwise, $\set{v_1,v_2}$, taken together with $dr_d+1$ vertices from each of the sets $X_{j}$, $j\in [r_0]$, and with the vertices from $U\setminus \set{u_i}$, yields a set of $s_0$ vertices with strictly more than $\SIGMA{T}{s_0}$ edges, a contradiction.
    
    \item
    Every edge $v_1 v_2$ in $U_{i_1} \times U_{i_2}$, with $i_1 \neq i_2$, has multiplicity at most $a-d+1$. Otherwise, if $w(v_1 v_2)=a-d+1+k$ for some $k \in [d]$, $\set{v_1, v_2}$, taken together with $d-k+1$ vertices from each $X_{j}$, $j\in [r_0]$, yields a set of $r_0(d-k+1)+2$ vertices with $\SIGMA{T}{r_0(d-k+1)+2} + 1$ edges, a contradiction.
    
    \item
    Every edge $v_1 v_2$ in $\binom{W_{j}}{2}$ has multiplicity at most $a$. Otherwise, $\set{v_1, v_2}$, taken together with an arbitrary vertex $u \in U$, and one vertex from each of the $X_{j'}$, $j'\in [r_0]\setminus \set{j}$, yields a set of $r_0+2$ vertices with $\SIGMA{T}{r_0+2} + 1$ edges, a contradiction.
\end{enumerate}

Set $V' \defined \left(\cup_{i=r_0+1}^{r_0+r_d} U_i\right) \cup \left(\cup_{j=1}^{r_0} W_j\right)$. By our observations above, $G[V']$ is a submultigraph of a blow-up $G'$ of $T$ with $T$-partition $W_1 \sqcup \dotsb \sqcup W_{r_0}\sqcup U_{r_0+1}\sqcup \dotsb \sqcup U_{r_0+r_d}$. Further, as shown in \Cref{lemma: stability part sizes}, we have $\card{V'} = n-o(n)$, and $G'$ has near product-optimal part sizes, so that
\begin{align}
\label{eq: product degree in G'}
    p_{G'}(v)= a\left((a+1)/a\right)^{\left(1-\frac{1-r_dx_\star}{r_0}\right)n+o(n)}
\end{align}
for every $v\in V'$. We say that an edge $v_1 v_2$ in $G[V']$ is \emph{light} if its multiplicity in $G$ is strictly less than in $G'$. It follows straightforwardly from our bound on product degrees in $G'$, \eqref{eq: product degree in G'}, together with our minimum product degree assumption on $G$ (and the facts that $\card{V\setminus V'} =o(n)$ and that all edges in $G$ have multiplicity at most $a+1$) that each vertex in $V'$ can be incident with at most $o(n)$ light edges. Thus, there is a total of $o(n^2)$ light edges in $G[V']$. It follows that one may change $G$ into a near product-optimal blow-up of $T$ by changing the $o(n^2)$ light edges together with the at most $n
\card{V \setminus V'} =o(n^2)$ edges incident with $V \setminus V'$. This concludes the proof of \Cref{theorem: large a technical statement}.
\end{proof}

With \Cref{theorem: large a technical statement} in hand, we now derive all of our main results. We shall need the following polynomial inequality, which is proved in the Appendix.

\begin{proposition}
\label{prop: polynomial identity}
The polynomial inequality~\eqref{eq: poly equation needed for good a}
is satisfied for all $a$ such that
\begin{equation*}
    a > d(2r_d-1)\left(d(2r_d-1)+1\right) + (d-1)(2r_d-1)(r_d-1)\left(r_d(d-1)(2r_d-1)+2r_d\right).
\end{equation*}
\end{proposition}
\begin{proof}[Proof of \Cref{theorem: base case main conj is true}]
Follows from \Cref{theorem: large a technical statement} and \Cref{prop: polynomial identity} applied with $r_d=1$.
\end{proof}

\begin{proof}[Proof of \Cref{cor: main conj asymptotically true}]
Follows from \Cref{theorem: base case main conj is true} and~\cite[Theorem 3.11]{day2022extremal}.
\end{proof}

\begin{proof}[Proof of \Cref{theorem: large a arbitrary number of deficient blobs}]
Follows from \Cref{theorem: large a technical statement} and \Cref{prop: polynomial identity} for the base case $s=r_0(dr_d+1)+ r_d+1$; for large $s$, follows from the base case together with \Cref{prop: step up from the base case}.
\end{proof}

\begin{proof}[Proof of \Cref{theorem: stability}]
Follows from \Cref{theorem: large a technical statement} and \Cref{prop: polynomial identity} for the base case $s=r_0(dr_d+1)+ r_d+1$; for large $s$, follows from the base case together with \Cref{prop: step up from the base case}.  
\end{proof}

\subsection{Flat intervals: proof of \texorpdfstring{\Cref{theorem: flat intervals}}{Theorem 1.18}}

In this section, we fix $r\geq 1$ and set $T = \TUR{(r)}{a}$. Note that $\SIGMA{T}{n} = a\binom{n}{2} + \ex(n, K_{r+1})$, where $\ex(n, K_{r+1})$ is the usual graph Turán number. In particular, it follows from Turán's theorem that
\begin{align}
\label{eq: growth rate turan}
    \SIGMA{T}{n+1} - \frac{n+1}{n-1} \, \SIGMA{T}{n} = n - \floor[\Big]{ \frac{n}{r}} -\frac{2}{n-1}\ex(n, K_{r+1}).    
\end{align}

To prove \Cref{conj: flat interval} on flat intervals, we need the following lemma on the Füredi--Kündgen multigraph problem.

\begin{lemma}
\label{lemma: additive averaging flat interval}
For all $s\geq 2r+1$ and all $n \geq s$, we have
\begin{equation*}
    \ex_\Sigma\p[\big]{ n, s, \SIGMA{T}{s} + \floor{(s-1)/r} - 1 } \leq \SIGMA{T}{n} + \floor{(n - 1)/r} - 1.
\end{equation*}
\end{lemma}
\begin{proof}
We proceed by induction on $n$. The base case $n=s$ is trivial. For the inductive step, suppose our claim is true for $n$. By the arithmetic averaging bound~\eqref{eq: sum-averaging bound on ex sigma} and the inductive hypothesis, we have
\begin{align*}
    \ex_\Sigma\left(n+1, s, \SIGMA{T}{s} +\left \lfloor\frac{s-1}{r} \right \rfloor-1 \right)&\leq \left\lfloor  \frac{n+1}{n-1}\left(\SIGMA{T}{n} +\left \lfloor \frac{n-1}{r} \right\rfloor -1\right)\right\rfloor.
\end{align*}
Substituting for $\frac{n+1}{n-1}\SIGMA{T}{n}$ using~\eqref{eq: growth rate turan} and rearranging terms, we see that it is enough for our purposes to show that
\begin{align*}
    2\ex(n, K_{r+1})+(n+1)\left \lfloor\frac{n-1}{r} \right \rfloor < n^2+1 
\end{align*}
holds. But this latter inequality is easily seen to be true since 
\begin{equation*}
    2\ex(n, K_{r+1}) + (n+1)\left \lfloor\frac{n-1}{r} \right \rfloor < 2 \binom{r}{2}\left( \frac{n}{r}\right)^2+ \frac{n^2}{r}=n^2< n^2+1. \qedhere
\end{equation*}
\end{proof}

\begin{proof}[Proof of \Cref{theorem: flat intervals}]
By \Cref{lemma: additive averaging flat interval}, we have
\begin{equation*}
    a \binom{n}{2}
    \leq \SIGMA{T}{n}
    \leq \ex_\Sigma\p[\big]{n, s, \SIGMA{T}{s} + \floor{(s-1)/r}-1}
    \leq \SIGMA{T}{n} + \floor{(n-1)/r} - 1
    \leq (a+1)\binom{n}{2}.
\end{equation*}
It then follows directly from the integral  AM-GM inequality (\Cref{prop: integer AM-GM}), that 
\begin{align*}
    \Pi_n(T) = a^{\binom{n}{2}}
    &\p[\Big]{\frac{a+1}{a}}^{\frac{r-1}{r}\binom{n}{2} + O(n)} \leq \ex_\Pi\p[\big]{n, s, \SIGMA{T}{s}} \\
    &\leq \ex_\Pi\p[\big]{n, s, \SIGMA{T}{s} + \floor{(s-1)/r} - 1 } \\
    &\leq a^{\binom{n}{2}} \p[\Big]{\frac{a+1}{a}}^{\ex_\Sigma\left(n,s, \SIGMA{T}{s}+ \floor{(s-1)/r}-1\right) - a\binom{n}{2}}
    \leq a^{\binom{n}{2}} \p[\Big]{\frac{a+1}{a}}^{\frac{r-1}{r}\binom{n}{2} + O(n)}.
\end{align*}
It follows that $\ex_\Pi(s, \SIGMA{T}{s} + t) = a^{1/r}(a + 1)^{(r-1)/r}$ for all $t$: $0 \leq t \leq \floor{(s-1)/r} - 1$.
\end{proof}


\section{Concluding remarks and open problems}
\label{section: concluding remarks}


\subsection{Large ambient multiplicity regime}

First and foremost remains the problem of proving \Cref{conj: large a regime generalised Turan optimal} in full. In principle, given a fixed Turán pattern $T=\TUR{\bfr}{a}$, one can use the method in our proof of \Cref{theorem: large a technical statement} to prove that for all $s\geq s_0$ and all $a$ sufficiently large, we have $\ex_\Pi(s, \SIGMA{T}{s})= e^{\pi_T}$ (here $s_0$ is the smallest number of vertices needed to `recognise' blow-ups of $T$). We have, for instance, verified this in the case where $\bfr=(1,1,1)$. The main roadblock for doing this in general is the technical \Cref{prop: partition sizes in sum-optimal blow-ups of T} on the sizes of parts in sum-optimal blow-ups of $T$: determining these is a finite optimisation problem and a purely mechanical exercise, but the case-checking involved rapidly outgrows our willingness to carry it out when $T$ is a more intricate pattern.

It may be unfeasible to tackle this problem directly, but perhaps one could exploit the recursive structure of generalised Turán patterns and their optimal blow-ups to give an indirect proof that blow-ups of generalised Turán patterns are always asymptotically product-optimal. Indeed, this is, to some extent, what Füredi and Kündgen~\cite{Furedi2002-fi} did for their additive problem: one can read out of their work that for $s, t\in \ZZ_{>0}$ fixed with $t<\binom{s}{2}$, and all $a\in \ZZ_{>0}$ sufficiently large, there is a generalised Turán pattern $T$ whose blow-ups provide asymptotically sum-extremal constructions of $(s, a\binom{s}{2}+t)$-graphs.  However, it requires significant computations to work out \emph{which} generalised Turán pattern it is, and the associated arithmetic density.

The difficulty of computing any of the numbers involved in generalised Turán patterns (size of $s_0$, growth rate of $\SIGMA{T}{n}$ and $\Pi_n(T)$ and associated partition sizes) may be a feature of the problem which makes it ultimately difficult to say significantly more about the general answer to the Mubayi--Terry problem in the large $a$ regime than we do in the present paper.

Turning to the specific results we obtained in this paper, one area where we do feel an improvement is possible is our `large $a$' condition. For some specific cases, such as $T=\TUR{\bfr}{a}$ with $\bfr=(1, 0, 1)$ and $s_0=5$, we have been able to show that as soon as the ambient edge multiplicity $a$ is large enough to make all edge multiplicities in $T$ strictly positive, then blow-ups of $T$ are product-optimal $(s_0, \SIGMA{T}{s_0})$ graphs. However, in general, our condition~\eqref{eq: poly equation needed for good a} leaves a gap. We would guess that this is not correct in general, and that blow-ups of $T$ are optimal as soon as they are possible, i.e. that one can replace `$a$ sufficiently large' by `$a$ is such that all multiplicities in $T$
are strictly positive'. In particular, there should be a way of improving the `step-up' \Cref{prop: step up from the base case} to get $a_1=a_0$.

Finally, regarding our `flat intervals' result, \Cref{theorem: flat intervals}, one could ask two further questions: can one obtain similar flat interval results after other generalised Turán patterns $T=\TUR{\bfr}{a}$ where $\bfr=(r)$, e.g.\ $\bfr=(r,1)$  or $(r, 0, 1)$? Such flat intervals should exist when $s$ is sufficiently large, and our proof of \Cref{theorem: flat intervals} should carry over provided one substitutes in appropriate results from Füredi and Kündgen for \Cref{lemma: additive averaging flat interval}. More interestingly, for $\bfr = (r)$, one can ask about the behaviour of the function $f(n) \defined \log(\ex_\Pi(n, s,q)/\Pi_n(T))$ as $q$ varies over the `flat interval' $[\SIGMA{T}{s}, \SIGMA{T}{s}+\lfloor (s-1)/r\rfloor -1]$. When (for which $q$) is $f(n)$ of order $O(1)$, $\theta(n)$, or $\omega(n)$ but $o(n^2)$ ? It would be interesting to see if an analogue of the known behaviour for the Erdős problem on $\ex(n,s',q')$, $q'\in [0, \lfloor (s')^2/4\rfloor]$ holds in this case. Similar questions could be asked about $\mathrm{ex}_{\Sigma}(n,s,q)$ with respect to its own flat intervals.

\subsection{Small ambient multiplicity regime}

At the moment, the general picture for the `small $a$' regime is somewhat unclear, for both the additive Füredi--Kündgen and the multiplicative Mubayi--Terry problem. One should note first of all that, if our \Cref{conj: large a regime generalised Turan optimal} is correct for a given $s$ and $t\in [0, \binom{s}{2})$, `small $a$' actually means `$a$ below the thresholds at which blow-ups of some fixed generalised Turán pattern $T=\TUR{\bfr}{a}$  can become sum- and product-optimal for $(s, a\binom{s}{2}+t)$-graphs. In particular, if e.g.\ $\bfr$ is the $(d+1)$-tuple $(1, 0, \dotsc, 0, 1)$, then this would mean $a < d + 1$ (for the multiplicative problem) or $a < d$ (for the additive problem).

In this paper, we have investigated small $a$ for which the edge multiplicities in extremal constructions end up being restricted to $\set{a - 1, a, a + 1}$. But for slightly larger $a$, e.g.\ $a=2$ in the additive problem or $a=3$ in the multiplicative problem, more edge multiplicities are potentially available for a putative extremal construction. A natural and interesting question to ask is whether one continues to only see either graph patterns $G^{(a)}$ or generalised Turán patterns as extremal constructions, or whether there are yet more extremal constructions that fall in neither of these categories.

In a different direction, it would be pleasing to complete the picture of the additive problem in the $a=1$ case and of the multiplicative problem in the $a=2$ case that we have begun drawing in this paper. A challenge in both cases is the difficulty of the design-theoretic problem of constructing optimal graph patterns.  We make some tentative conjectures below regarding the behaviour for a range of $q$ of length $\Omega(s^2)$ not covered by our results above.

Let $K_{1, \infty}=K_{1, \infty}^{(a)}$ denote the pattern consisting of an edge of multiplicity $a+1$, with a loop of multiplicity $a-1$ at one endpoint and a loop of multiplicity $a$ at the other endpoint (note that this is in fact a generalised Turán pattern $T = \TUR{\bfr}{a}$ with $\bfr = (1,1)$). Let also $H_{26}$ denote the point-line incidence graph of the projective plane $\mathrm{PG}(2,3)$, which is the unique up to isomorphism $4$-regular graph on $26$ vertices with girth $6$. Finally, let $\mathrm{Clebsch}$ denote the Clebsch graph; this is the unique $(4,5)$-cage on sixteen vertices. Set $\cC_{4.5} \defined \set{K_{1, 4}, H_{26}}$, $\cC_4 \defined \set{ {K_{1,5}}, {C_4}, {\mathrm{Clebsch}} }$ and $\cC_3 \defined \set{ {K_{1,\infty}}, {C_3}, {K_{3,3}} }$.

\begin{conjecture}[Additive conjecture]
\label{conj: additive}
Let $s\in \ZZ_{\geq 4}$ and for all graph patterns $G=G^{(a)}$ below, let the ambient edge multiplicity be $a=1$. Then 
\begin{enumerate}[(i)]	
    \item for all $k\geq 4$, $\SIGMA{K_{1,k}}{s} \leq q < \min\left(\SIGMA{K_{1,k+1}}{s}, \SIGMA{K_{1,\infty}}{s}\right)$, we have $\ex_\Sigma(s, q)= 1+\frac{k-1}{3k+1}$;
    \item for $q = \SIGMA{K_{1,\infty}}{s}$, we have $\ex_\Sigma(s, q) =  1 + 1/3$.
\end{enumerate}		
\end{conjecture}

What is more, for (i), collections of extremal constructions in the cases $k=4$ and $k=5$ are given by blow-ups of the graph patterns in $\cC_{4.5}$ and $\cC_4$ respectively, while for (ii), a collection of extremal constructions is given by blow-ups of the graph patterns in $\cC_3$.



\begin{conjecture}[Multiplicative conjecture]
\label{conj: mutliplicative}
Let $s\in \ZZ_{\geq 4}$ and for all graph patterns $G=G^{(a)}$ below, let the ambient edge multiplicity be $a=2$. Then
\begin{enumerate}[(i)]	
    \item for $\SIGMA{K_{1,4}}{s} \leq q < \min\left(\SIGMA{K_{1,7}}{s}, \SIGMA{K_{1,\infty}}{s}\right)$, we have $\ex_\Pi(s, q)= e^{\pi_{\Petersen}}$;
    \item for all $k\geq 7$, $\SIGMA{K_{1,k}}{s} \leq q < \min(\SIGMA{K_{1,k+1}}{s}, \SIGMA{K_{1,\infty}}{s})$, we have $\ex_\Pi(s, q)= e^{\pi_{K_{1,k}}}$;
    \item for $q = \SIGMA{K_{1,\infty}}{s}$, we have $\ex_\Pi(s, q)= e^{\pi_{K_{1,\infty}}}$.
\end{enumerate}		
\end{conjecture}

\subsection{Other questions}

As suggested by Füredi and Kündgen~\cite[Section 5]{Furedi2002-fi}, one could investigate the families of limits $\set{ \ex_\Sigma(s, q) \st s, q \in \ZZ_{>0}}$ and $\set{ \ex_\Pi(s,q) \st s, q \in \ZZ_{>0} }$.
In particular, are they well-ordered?
And what can one say if one restricts one's attention to a fixed ambient edge density, e.g.\ $a\binom{s}{2} \leq q < (a + 1) \binom{s}{2}$?
Given results of Rödl and Sidorenko~\cite{rodl1995jumping} on the jumping constant conjecture for multigraphs, the cases $a \leq 2$ seem particularly interesting.

In a different direction, in the spirit of the work of Erdős, Faudree, Jagota and Łuczak in~\cite{erdos1996large}, one could consider \Cref{problem: Furedi Kundgen,problem: Mubayi Terry} for $(s, q)$-graphs on $n$ vertices with $s$ and $q$ both varying as functions of $n$.
To formulate this with more precision, say that a multigraph $G$ contains an $(s, q)$-set if there is some $s$-set $X \subseteq V(G)$ with $e_{G}(X) \geq q$.

\begin{problem}
Let $G$ be a multigraph on $n$ vertices, and let $a\in \ZZ_{>0}$, $p\in (0,1)$ be fixed.
\begin{enumerate}[(i)]
    \item Suppose $e(G)\geq (a+p)\binom{n}{2}$. Given $s=s(n)$, what is the largest $q=q(a, p, s, n)$ such that $G$ must contain an $(s,q)$-set? 
    
    \item Suppose $P(G)\geq (a+p)^{\binom{n}{2}}$. Given $s=s(n)$, what is the largest $q=q(a, p, s, n)$ such that $G$ must contain an $(s,q)$-set? 
\end{enumerate}
\end{problem}


Finally, one could consider general decreasing bounded multigraph properties, and ask whether blow-ups of patterns always provide some of the extremal examples.

\begin{question}
Let $\cP$ be a decreasing bounded property of multigraphs. Is it the case that there exist finite $\cP$-admissible patterns $P_+$ and $P_{\times}$ such that $\ex_\Sigma(n, \cP)=(1+o(1))\SIGMA{P_+}{n}$ and 
$\ex_\Pi(n, \cP)=(1+o(1))\PI{P_{\times}}{n}$ ?
\end{question}


\section*{Acknowledgements}
Discussions leading to this paper began during a research visit by Vojtěch Dvořák, Adva Mond and Victor Souza to Umeå in late 2021, and an important part of the work was carried out over the course of Rik Sarkar's three months-long research internship in Umeå in Fall 2024. The hospitality and support of the Umeå department of mathematics and mathematical statistics for both of these stays in Umeå is gratefully acknowledged.

The authors would also like to express a special gratitude to Vojtěch Dvořák for contributing to stimulating conversations during the exploratory phase of this project. The first author would also like to thank his colleagues Klas Markström and Eero Räty for fruitful conversations on closely related topics.

The first author is supported by Swedish Research Council grant VR 2021-03687.
The fourth author is partially supported by ERC Starting Grant 101163189 and UKRI Future Leaders Fellowship MR/X023583/1.


\bibliographystyle{unsrt}
\bibliography{mubayiterry_bibliov2}

\begin{thebibliography}{10}

\bibitem{erdos1967}
Paul Erd\H{o}s.
\newblock Extremal problems in graph theory.
\newblock In {\em A Seminar on Graph Theory}, pages 54--59. Holt, Rinehart, and
  Winston, New York, 1967.

\bibitem{erdos1964}
Paul Erd{\H o}s.
\newblock Extremal problems in graph theory.
\newblock In {\em Theory of Graphs and Its Applications}, pages 29--36. Publ.
  House. Czechoslovak Acad. Sci., Prague, 1964.
\newblock Proc. Sympos. Smolenice, 1963.

\bibitem{turan1941}
Paul Tur\'an.
\newblock On an extremal problem in graph theory.
\newblock {\em Matematikai és Fizikai Lapok}, 48:436--452, 1941.

\bibitem{Dirac1963-vv}
Gabriel Dirac.
\newblock {Extensions of Tur\'{a}n's theorem on graphs}.
\newblock {\em Acta Mathematica Hungarica}, 14(3-4):417--422, 1963.

\bibitem{Erdos1946-df}
Paul Erd{\H o}s and Arthur~Harold Stone.
\newblock {On the structure of linear graphs}.
\newblock {\em Bulletin of the American Mathematical Society},
  52(12):1087--1091, 1946.

\bibitem{gol1987maximum}
AI~Gol'berg and Vladimir~Alexander Gurvich.
\newblock On the maximal number of edges for a graph with n vertices in which a
  given subgraph with k vertices has no more than l edges.
\newblock In {\em Doklady Akademii Nauk}, volume 293, pages 27--32. Russian
  Academy of Sciences, 1987.

\bibitem{griggs1998extremal}
Jerrold~R. Griggs, Mikl{\'o}s Simonovits, and George~Rubin Thomas.
\newblock Extremal graphs with bounded densities of small subgraphs.
\newblock {\em Journal of Graph Theory}, 29(3):185--207, 1998.

\bibitem{Bondy1997-ji}
John~Adrian Bondy and Zsolt Tuza.
\newblock {A weighted generalization of {T}urán's theorem}.
\newblock {\em Journal of Graph Theory}, 25(4):267--275, 1997.

\bibitem{Kuchenbrod1999-as}
John~Allen Kuchenbrod.
\newblock {\em {Extremal Problems On Integer-Weighted Graphs}}.
\newblock PhD thesis, University of Kentucky, 1999.

\bibitem{Furedi2002-fi}
Zolt\'an F\"uredi and Andr\'e K\"undgen.
\newblock Tur\'an problems for integer-weighted graphs.
\newblock {\em Journal of Graph Theory}, 40(4):195--225, 2002.

\bibitem{Mubayi2019-le}
Dhruv Mubayi and Caroline Terry.
\newblock {An Extremal Graph Problem with a Transcendental Solution}.
\newblock {\em Combinatorics, Probability and Computing}, 28(2):303--324, 2019.

\bibitem{Erdos1976-ad}
Paul Erd{\H o}s, Daniel~J Kleitman, and Bruce~Lee Rothschild.
\newblock {Asymptotic enumeration of $K_n$-free graphs}.
\newblock In {\em {Tomo II, Atti dei Convegni Lincei, No. 17}}, pages 19--27,
  1976.

\bibitem{alekseev1993entropy}
Vladimir~E Alekseev.
\newblock On the entropy values of hereditary classes of graphs.
\newblock {\em Discrete Mathematics and Applications}, 3:191--199, 1993.

\bibitem{bollobas1997hereditary}
B{\'e}la Bollob{\'a}s and Andrew Thomason.
\newblock Hereditary and monotone properties of graphs.
\newblock In {\em The Mathematics of Paul Erd{\H o}s II}, pages 70--78.
  Springer, 1997.

\bibitem{Balogh2015-al}
József Balogh, Robert Morris, and Wojciech Samotij.
\newblock {Independent sets in hypergraphs}.
\newblock {\em Journal of the American Mathematical Society}, 28(3):669--709,
  2015.

\bibitem{Saxton2015-cr}
David Saxton and Andrew Thomason.
\newblock {Hypergraph containers}.
\newblock {\em Inventiones Mathematicae}, 201(3):925--992, 2015.

\bibitem{erdos1974some}
Paul Erd{\H{o}}s.
\newblock Some new applications of probability methods to combinatorial
  analysis and graph theory.
\newblock In {\em Proceedings of the Fifth Southeastern Conference on
  Combinatorics, Graph Theory and Computing}, pages 39--51. Boca Raton, 1974.

\bibitem{Falgas-Ravry2019-ne}
Victor Falgas-Ravry, Kelly O'Connell, and Andrew Uzzell.
\newblock {Multicolor containers, extremal entropy, and counting}.
\newblock {\em Random Structures \& Algorithms}, 54(4):676--720, 2019.

\bibitem{hoppen2017rainbow}
Carlos Hoppen, Hanno Lefmann, and Knut Odermann.
\newblock A rainbow {E}rd{\H o}s--{R}othschild problem.
\newblock {\em SIAM Journal on Discrete Mathematics}, 31(4):2647--2674, 2017.

\bibitem{de2020number}
Josefran de~Oliveira Bastos, Fabricio~Siqueira Benevides, and Jie Han.
\newblock The number of gallai k-colorings of complete graphs.
\newblock {\em Journal of Combinatorial Theory, Series B}, 144:1--13, 2020.

\bibitem{pikhurko2024exact}
Oleg Pikhurko and Katherine Staden.
\newblock Exact solutions to the erdős-rothschild problem.
\newblock {\em Forum of Mathematics, Sigma}, 12:e8, 2024.

\bibitem{simonovits1968method}
Mikl{\'o}s Simonovits.
\newblock A method for solving extremal problems in graph theory, stability
  problems.
\newblock In {\em Theory of Graphs (Proc. Colloq., Tihany, 1966)}, pages
  279--319, 1968.

\bibitem{liu2022hypergraph}
Xizhi Liu and Dhruv Mubayi.
\newblock A hypergraph tur{\'a}n problem with no stability.
\newblock {\em Combinatorica}, 42(3):433--462, 2022.

\bibitem{rodl1995jumping}
Vojt{\v{e}}ch R{\"o}dl and Alexander Sidorenko.
\newblock On the jumping constant conjecture for multigraphs.
\newblock {\em Journal of Combinatorial Theory, Series A}, 69(2):347--357,
  1995.

\bibitem{frankl1984hypergraphs}
Peter Frankl and Vojt{\v{e}}ch R{\"o}dl.
\newblock Hypergraphs do not jump.
\newblock {\em Combinatorica}, 4(2):149--159, 1984.

\bibitem{de2000maximum}
Dominique De~Caen and Zolt{\'a}n F{\"u}redi.
\newblock The maximum size of 3-uniform hypergraphs not containing a fano
  plane.
\newblock {\em Journal of Combinatorial Theory, Series B}, 78(2):274--276,
  2000.

\bibitem{brown1973some}
William~G. Brown, Paul Erd{\H o}s, and Vera~T. S{\'o}s.
\newblock Some extremal problems on r-graphs.
\newblock In {\em New directions in the theory of graphs (Proc. Third Ann Arbor
  Conf., Univ. Michigan, Ann Arbor, Mich, 1971)}, pages 53--63. Academic Press
  New York, 1973.

\bibitem{ruzsa1978triple}
Imre~Z. Ruzsa and Endre Szemer{\'e}di.
\newblock Triple systems with no six points carrying three triangles.
\newblock {\em Combinatorics (Keszthely, 1976), Coll. Math. Soc. J. Bolyai},
  18(2):939--945, 1978.

\bibitem{delcourt2024limit}
Michelle Delcourt and Luke Postle.
\newblock The limit in the $(k+ 2, k)$-problem of {B}rown, {E}rd{\H{o}}s and
  {S}{\'o}s exists for all $k\geq 2$.
\newblock {\em Proceedings of the American Mathematical Society},
  152(05):1881--1891, 2024.

\bibitem{glock20246}
Stefan Glock, Felix Joos, Jaehoon Kim, Marcus K{\"u}hn, Lyuben Lichev, and Oleg
  Pikhurko.
\newblock On the (6, 4)-problem of {B}rown, {E}rd{\H{o}}s, and {S}{\'o}s.
\newblock {\em Proceedings of the American Mathematical Society, Series B},
  11(17):173--186, 2024.

\bibitem{keevash2020brown}
Peter Keevash and Jason Long.
\newblock The {B}rown--{E}rd{\H{o}}s conjecture for hypergraphs of large
  uniformity.
\newblock {\em arXiv preprint arXiv:2007.14824}, 2020.

\bibitem{shapira2025new}
Asaf Shapira and Mykhaylo Tyomkyn.
\newblock A new approach for the brown–erdős–sós problem.
\newblock {\em Israel Journal of Mathematics}, pages 1--12, 2025.

\bibitem{katona1964graph}
Gyula~O.H. Katona, T.~Nemetz, and Mikl\'os Simonovits.
\newblock On a graph-problem of {T}ur{\'a}n in the theory of graphs.
\newblock {\em Matematikai Lapok}, 15:228--238, 1964.
\newblock (in Hungarian).

\bibitem{Mubayi2020-fb}
Dhruv Mubayi and Caroline Terry.
\newblock {Extremal Theory of Locally Sparse Multigraphs}.
\newblock {\em SIAM Journal on Discrete Mathematics}, 34(3):1922--1943, 2020.

\bibitem{day2022extremal}
A.~Nicholas Day, Victor Falgas-Ravry, and Andrew Treglown.
\newblock Extremal problems for multigraphs.
\newblock {\em Journal of Combinatorial Theory, Series B}, 154:1--48, 2022.

\bibitem{falgas2024extremal}
Victor Falgas-Ravry.
\newblock On an extremal problem for locally sparse multigraphs.
\newblock {\em European Journal of Combinatorics}, 118:103887, 2024.

\bibitem{gu2024extremal}
Ran Gu and Shuaichao Wang.
\newblock On extremal problems on multigraphs.
\newblock {\em Graphs and Combinatorics}, 40(6):114, 2024.

\bibitem{Komlos1996-uq}
János Komlós and Miklós Simonovits.
\newblock {Szemerédi's Regularity Lemma and its applications in graph theory}.
\newblock In {\em {Paul Erdős is eighty}}, volume~II, pages 295--352. Janos
  Bolyai Mathematical Society, 1996.

\bibitem{komlos1997blow}
J{\'a}nos Koml{\'o}s, G{\'a}bor S{\'a}rk{\"o}zy, and Endre Szemer{\'e}di.
\newblock Blow-up lemma.
\newblock {\em Combinatorica}, 17:109--123, 1997.

\bibitem{komlos1999blow}
J{\'a}nos Koml{\'o}s.
\newblock The blow-up lemma.
\newblock {\em Combinatorics, Probability and Computing}, 8(1-2):161--176,
  1999.

\bibitem{falgasravryraty}
Victor Falgas-Ravry and Eero R\"aty.
\newblock Personal communication.

\bibitem{Sarkar2025-ms}
Rik Sarkar.
\newblock {Extremal Problems for Multigraphs}.
\newblock Master's thesis, Indian Institute of Science Education and Research
  (IISER), Pune, 2025.
\newblock
  \url{http://dr.iiserpune.ac.in:8080/jspui/bitstream/123456789/9859/1/20201122%20_Rik_Sarkar_MS_Thesis.pdf}.

\bibitem{erdos1996large}
Paul Erd{\H o}s, Ralph Faudree, Arun Jagota, and Tomasz {\L}uczak.
\newblock Large subgraphs of minimal density or degree.
\newblock {\em Journal of combinatorial mathematics and combinatorial
  computing}, 22:87--96, 1996.

\end{thebibliography}


\section*{Appendix}
\label{appendix: product optimisation}

\begin{proof}[Proof of \Cref{prop: product extremal for cycles paths stars} parts (i)--(iv)]
\noindent\textbf{Part (i):} let $\bfx$ be a product-optimal weighting of the pattern $K_{1,\ell}^{(2)}$. We may assume that $\bfx$ assigns weight $x$ to the central vertex, and weight $(1 - x)/\ell$ to each of the leaf vertices, for some $x \in (0,1)$.
By \Cref{cor: balanced degrees in optimal pattern weightings}, we have 
\begin{equation*}
    \pi_{K_{1,\ell}} = (1 - x)\log 3,
    \quad \text{and} \quad
    \pi_{K_{1,\ell}} = \frac{(\ell - 1)\log 2 + \p[\big]{\ell \log(3/2) + \log 2 } x}{\ell}.
\end{equation*}
Solving for $\pi_{K_{1,\ell}}$, we get $\pi_{K_{1,\ell}} = \ell (\log 3)^2/(2\ell \log 3 -(\ell-1)\log 2)$.

\noindent\textbf{Part (ii):} let $\bfx$ be a product-optimal weighting of the pattern $P_4^{(2)}$. For $i \in [4]$, let $x_i$ be the weight assigned to the $i$th vertex on the path, following the natural order along the path. Then, either $\pi_{P_4}=\pi_{K_{1,2}}$, or $\bfx$ assigns strictly positive weight to every vertex of the $P_4$. By \Cref{cor: balanced degrees in optimal pattern weightings}, we have the following:
\begin{align}
\label{equation18}
    \pi_{P_4} &= \log 2 + x_2 \log(3/2) - x_1 \log 2,\\
\label{equation19}
    \pi_{P_4} &= \log 2 + x_3 \log(3/2) - x_4 \log 2,\\
\label{equation20}
    \pi_{P_4} &= \log 2 + (x_1 + x_3) \log(3/2) - x_2 \log 2,\\
\label{equation21}
    \pi_{P_4} &= \log 2 + (x_2 + x_4) \log(3/2) - x_3 \log 2.
\end{align}
Taking the linear combination $(\log 2) \cdot \big($\eqref{equation18}+\eqref{equation19} $\big)+\log(3) \big( $\eqref{equation20}+\eqref{equation21}$\big)$, we obtain
\begin{align*}
    (2 \log 6) \pi_{P_4} &= 2\log 2 \log 6 + \p[\big]{ \log3 \log(3/2) - (\log 2)^2} \sum_{i=1}^4 x_i\\
    &= (\log 3)^2 + (\log 2)^2 + (\log 2)(\log 3).
\end{align*}
Thus, we have $\pi_{P_4} = \p[\big]{(\log 3)^2 + (\log 2)^2 + \log 2 \log 3}/(2 \log 6) > \pi_{K_{1,2}}$.

\noindent\textbf{Part (iii):} let $\bfx$ be a product-optimal weighting of $P_5^{(2)}$. For $i \in [5]$, let $x_i$ denote the weight assigned to the $i$th vertex on the path, following the natural order along the path. Then, either $\pi_{P_5}=\pi_{P_4}$ or $\bfx$ assigns strictly positive weight to every vertex of the $P_5$. By \Cref{cor: balanced degrees in optimal pattern weightings}, we have
\begin{align}
\label{equation22}
    \pi_{P_5} &= \log 2 + x_2 \log(3/2) - x_1 \log 2,\\
\label{equation23}
    \pi_{P_5} &= \log 2 + x_4 \log(3/2) - x_5 \log 2,\\
\label{equation24}
    \pi_{P_5} &= \log 2 + (x_1 + x_3) \log(3/2) - x_2 \log 2,\\
\label{equation25}
    \pi_{P_5} &= \log 2 + (x_3 + x_5) \log(3/2) - x_4 \log 2,\\
\label{equation26}
    \pi_{P_5} &= \log 2 + (x_2 + x_4) \log(3/2) - x_3 \log 2.
\end{align}
Set $x_1 + x_5 = 2 t_1$, $x_2 + x_4 = 2 t_2$ and $x_3 = 1 - 2(t_1 + t_2)$.
Adding Equations \eqref{equation22} and \eqref{equation23}, and dividing by 2 yields 
\begin{equation}
\label{equation27}
    \pi_{P_5} = \log 2 + t_2 \log(3/2) - t_1 \log 2.
\end{equation}
Adding Equations \eqref{equation24} and \eqref{equation25}, and dividing by 2 yields 
\begin{equation}
\label{equation28}
    \pi_{P_5} = \log 3 - t_1 \log(3/2) - t_2 \log(9/2).
\end{equation}
Finally, plugging in the new variables into Equation \eqref{equation26}, we get 
\begin{equation}
\label{equation29}
    \pi_{P_5} = 2 t_1 \log 2 + 2 t_2 \log 3.
\end{equation}
Equations \eqref{equation27}, \eqref{equation28} and \eqref{equation29} form a system of three equations in three variables. Solving for $\pi_{P_5}$, we get
\begin{equation*}
    \pi_{P_5} = \frac{2 \log 2 \p[\big]{\log 6 \log(9/2) - \log 3 \log(3/2)}}{5 \log 2 \log(9/2) - \log 6 \log (3/2)} > \pi_{P_4}.
\end{equation*}

\noindent\textbf{Part (iv):} let $\bfx$ be a product-optimal weighting of the pattern $C_6^{(2)}$. For $i \in [6]$, let $x_i$ denote the weight assigned to the $i$th vertex on the cycle, following a natural cyclic ordering. Then, either $\pi_{C_6}=\pi_{P_5}$ or $\bfx$ assigns strictly positive weight to every vertex of the $C_6$.  In the latter case, by \Cref{cor: balanced degrees in optimal pattern weightings}, for all $i \in [6]$, we have 
\begin{equation*}
    \pi_{C_6} = \log 2 + (x_{i-1} + x_{i+1}) \log(3/2) - x_i \log 2,
\end{equation*}
where the indices are taken modulo 6. Therefore, we have 
\begin{align*}
    6 \pi_{C_6} &= 6\log 2 + \sum_{i=1}^6 \p[\big]{ (x_{i-1}+x_{i+1}) \log(3/2) - x_i \log 2 }\\
    &= 6\log 2 + 2 \log(3/2) - \log 2 = 3\log 2 + 2\log 3.
\end{align*}
Thus, we have $\pi_{C_6} = (\log 3)/3 + (\log 2)/2 > \pi_{P_5}$, as required.
The proposition follows.
\end{proof}

\begin{proof}[Proof of \Cref{prop: polynomial identity}]
The polynomial inequality in the proposition is equivalent to
\begin{align*}
    \p[\Big]{1 + \frac{1}{a}}^{r_d(d - 1)(2r_d - 1) + 2r_d} \p[\Big]{ 1-\frac{d}{a}}^{2r_d-1} \p[\Big]{1 -\frac{d-1}{a}}^{(2r_d-1)(r_d-1)} > 1.
\end{align*}
Applying Bernoulli's inequality: $(1 + x)^m \geq 1 + mx$ for all $x > -1$, it suffices to show that
\begin{align}
\label{eq: a ineq to ensure poly ok}
    (a + r_d(d-1)(2r_d-1) + 2r_d)(a - d(2r_d-1))(a - (2r_d-1)(r_d-1)(d-1)) - a^3 > 0.
\end{align}
to ensure our inequality is satisfied.
The left-hand side is a quadratic function of $a$.
The coefficient of $a^2$ is $1$, and the constant term is non-negative.
The coefficient of $a$ is 
\begin{equation*}
    c_1 \defined -\p[\Big]{ d(2r_d-1)\p[\big]{d(2r_d-1) + 1} + (d-1)(2r_d-1)(r_d-1)\p[\big]{r_d(d-1)(2r_d-1) + 2r_d}}.
\end{equation*}
Thus for $a > c_1$, \eqref{eq: poly equation needed for good a} holds.
The proposition follows.
\end{proof}


\end{document}